\documentclass[12pt]{amsart}
\usepackage{amsmath}
\usepackage{amssymb}
\usepackage{amsfonts}
\usepackage{amsthm}\usepackage{mathrsfs}
\usepackage{comment}
\usepackage[normalem]{ulem}
\usepackage{enumerate}

\usepackage{geometry}
\usepackage{mathtools}
\mathtoolsset{showonlyrefs}
\usepackage{hyperref}
\hypersetup{colorlinks=true,
linktoc=all,linkcolor=black,
citecolor=black}
\usepackage{tikz}
\usepackage{bbm}
\usepackage{mathptmx}

\newcommand{\newsection}[1]{\section{#1}}
\newcommand{\bea}{\begin{eqnarray}}
\newcommand{\eea}{\end{eqnarray}}

\newcommand{\vp}{\varphi}

\newcommand{\clb}{\mathcal{B}}

\newcommand{\cls}{\mathcal{S}}

\newcommand{\scrD}{\mathscr{D}}

\newcommand{\D}{\mathbb{D}}

\newcommand{\raro}{\rightarrow}

\def\textmatrix#1&#2\\#3&#4\\{\bigl({#1 \atop #3}\ {#2 \atop #4}\bigr)}
\def\dispmatrix#1&#2\\#3&#4\\{\left({#1 \atop #3}\ {#2 \atop #4}\right)}
\newcommand{\be}{\begin{equation}}
\newcommand{\ee}{\end{equation}}
\newcommand{\ben}{\begin{eqnarray*}}
\newcommand{\een}{\end{eqnarray*}}

\newcommand{\bi}{\begin{itemize}}
\newcommand{\ei}{\end{itemize}}

\newcommand{\Z}{\mathbb{Z}}
\newcommand{\T}{\mathbb{T}}
\newcommand{\imag}{\mathrm{i}}
\newcommand{\e}{\mathrm{e}}
\newcommand{\Tope}{\mathrm{T}}
\newcommand{\Mop}{\mathrm{M}}
\newcommand{\Pop}{\mathrm{P}}
\newcommand{\Sop}{\mathrm{S}}

\newcommand{\Jop}{\mathrm{J}}
\newcommand{\Nop}{\mathrm{N}}

\newcommand{\Wvar}{W}

\newcommand{\Berg}{\mathrm{b}}
\newcommand{\re}{\mathrm{Re}}

\newcommand{\C}{\mathbb{C}}
\newcommand{\R}{\mathbb{R}}

\newcommand{\diff}{\mathrm{d}}
\newcommand{\diffs}{\mathrm{ds}}
\newcommand{\diffA}{\mathrm{dA}}

\newcommand\la{{\langle }}
\newcommand\ra{{\rangle}}

\newcommand{\Ordo}{\mathrm{O}}
\newcommand{\ordo}{\mathrm{o}}

\newcommand{\dA}{\mathrm{dA}}

\newcommand{\ima}{\mathrm{i}}

\newcommand{\Grunsky}{\mathrm{G}}
\newcommand{\Bergman}{\mathrm{P}}

\theoremstyle{definition}

\theoremstyle{plain}

\newtheorem{thm}{Theorem}[subsection]
\newtheorem{cor}[thm]{Corollary}
\newtheorem{lem}[thm]{Lemma}
\newtheorem{prop}[thm]{Proposition}
\theoremstyle{definition}
\newtheorem{defn}[thm]{Definition}
\newtheorem{rem}[thm]{Remark}

\newtheorem{problem}[thm]{Problem}

\numberwithin{equation}{subsection}
\numberwithin{figure}{section}

\begin{document}

\title{Dirichlet symbols and the nonlinear wave equation}


\author[Debnath]{Ramlal Debnath}
\address{\textbf{Ramlal Debnath:}\newline
Department of Mathematics, KTH Royal Institute of Technology,
Sweden}
\email{ramlaldebnath@gmail.com}

\author[HEDENMALM]{HAAKAN HEDENMALM}
\address{\textbf{Haakan Hedenmalm:}\newline
Department of Mathematics and Computer Sciences,
St Petersburg State University, Russia\newline
Beijing Institute for Mathematical Sciences and Applications,
Huairou District, Beijing 101408, China\newline
Department of Mathematics,
University of Aveiro, Aveiro, Portugal}
\email{haakan00@gmail.com}




\begin{abstract}
We develop the geometry of conformal mappings through Dirichlet operator
symbols and the nonlinear wave equation. These symbols, introduced in the
work of Hedenmalm and Shimorin (2020), are holomorphic functions on the
bidisk associated with contractions on $L^2(\D)$. For Grunsky operators,
the nonlinear wave equation and the related Riccati equation induce a
hierarchy of differential relations along the diagonal. We use this jet
hierarchy to obtain alternative coordinates and a characterization of the
manifold $\Sigma/\C$ of normalized exterior univalent functions modulo
translations, without an explicit univalence test. This manifold may be
viewed as an extension of the universal Teichm\"uller space beyond the
quasicircle boundary setting.

The diagonal jets also give rise to a hierarchy of \emph{diagonal series
Schwarz\-ian derivatives}. Already the fourth Schwarz\-ian $\Sop_4$ leads to
differential equations connected with the Weierstrass elliptic function
$\wp$ and Painlev\'e I, revealing a comparison between the geometry of
conformal mappings and that of complex tori. As a quantitative application
of the framework, we interpret asymptotic variance through the
Littlewood-Paley formula as an asymptotic density of quadratic
nonlinearity energy under logarithmic averaging. We introduce Schwarz\-ian
asymptotic variance and obtain a two-sided estimate for its maximal
averaged squared hyperbolically normalized amplitude, with lower bound
$6(1-\rho_{\mathbb{H}})\gtrsim5.686$ from hyperbolic discrepancy density and
upper bound $12.566563\ldots$, in comparison with the classical
pointwise bound $36$.
\end{abstract}

\maketitle


\newsection{Overview}

The purpose of this paper is to develop a differential and
operator-theoretic language for conformal mapping. The Dirichlet-symbol
formalism turns the geometry of Grunsky operators into a nonlinear wave
equation. Its diagonal jets then lead both to a characterization of the
conformal mapping manifold and to a hierarchy of higher Schwarz\-ian
derivatives. The connections with elliptic functions illustrate how these
invariants can relate distinct complex geometries. Our study of
asymptotic variance in the final section is a quantitative application of
this framework, rather than its point of departure.

In Section \ref{sec: 1}, we study the Dirichlet symbols of contractions on
$L^2(\D)$ introduced by Hedenmalm and Shimorin  \cite{GAF} in the context of
holomorphic correlations of two copies of the Gaussian analytic function
derived from the Dirichlet space on the unit disk. We show how they relate
in a one-to-one fashion to a Hankel part of the contraction, and obtain a
characterization of these Dirichlet symbols
associated with a contractive multiplication operator in terms of a linear
wave equation. This parallels the Grunsky operator case which is characterized
by a nonlinear wave equation. The following Section \ref{sec:NLWE} is devoted
to the study of that nonlinear wave equation, as well as to the related
Riccati equation considered by Dov Aharonov \cite{Dov}. Both equations supply
characterizations of the Dirichlet symbols of Grunsky operators.
Then, beginning in Section \ref{sec:NLWE} and continuing in Section \ref{S3},
we expand the nonlinear wave equation as well as the Riccati equation of
Aharonov along the diagonal in the sense of jets. This quickly leads to
the nonlinearity (pre-Schwarz\-ian) and the Schwarz\-ian derivative, where the
Schwarz\-ian derivative naturally falls in a family of \emph{diagonal series}
Schwarz\-ian derivatives. We develop relations between the various jet
components, where starting from the function $\log\psi'$, understood
as the diagonal restriction of the Dirichlet symbol itself, we can build
up all the  other diagonal jet components through a sequence of identities. 
Particularly interesting is how knowing the index-diagonal components gives
all the others from a simple combinatorial argument based on symmetry,
whereas the index-diagonal components get derived from the nonlinear wave
equation or the Riccati equation. The next level Schwarz\-ian derivative,
$\Sop_4(\psi)$, may be expressed in terms of the Schwarz\-ian derivative
$\Sop_2(\psi)=\frac16\Sop(\psi)$ via a nonlinear differential operator which
connects with the Painlev\'e I equation and the Weierstrass differential
equation for the elliptic function $\wp$. The appearance of $\wp$
suggests a comparison between the differential invariants of conformal
mappings and the complex geometry of tori. In Section \ref{sec:parameter},
we obtain one of the main results, which is a way to parametrize $\Sigma/\C$,
understood as an extended universal Teichm\"uller space, in terms of diagonal
jets.
The resulting characterization of $\Sigma/\C$ does not require any
explicit testing of univalence.
Finally, in Section \ref{sec:asymp}, we take Curtis McMullen's notion of
\emph{asymptotic variance} and based on \cite{geometric zero packing},
we interpret it, through the Littlewood--Paley comparison, as an
asymptotic density of squared nonlinearity amplitude under logarithmic
averaging in the hyperbolic plane. We proceed to
define a similar \emph{Schwarz\-ian asymptotic variance} as the $L^2$ average
of the Schwarz\-ian derivative amplitude of $(1-|z|^2)^4|\Sop(\vp)|^2$ in
the hyperbolic plane. We obtain the
two-sided bound for the maximal averaged squared amplitude,
$6(1-\rho_{\mathbb{H}})\leq
\sup_{\vp\in\mathscr S}\sigma_2(\Sop(\vp))^2
\leq12.566563\ldots$,
considerably smaller at the upper end than the maximal squared pointwise
amplitude $6^2=36$. Here $\rho_{\mathbb{H}}$ is the hyperbolic discrepancy density
from \cite{geometric zero packing}; the lower bound and its numerical
consequence are established in \cite{folding-topology}.
The corresponding bound for the nonlinearity has
also been improved by the finite-scale folding method. 
Here, $\Sop(\vp)$ is the Schwarz\-ian
derivative of $\varphi\in\mathscr{S}$, and the analogous statement is
valid for $\psi\in \Sigma$ as well.

\newsection{Introduction}\label{sec: 1}

\subsection{Basic notation}
We write $\R$ for the real line and $\C$ for the complex
plane. Moreover, we write $\mathbb{C}_{\infty}=\mathbb{C}\cup\{\infty\}$
for the extended complex plane (the Riemann sphere). For a complex
variable $z=x+\ima y \in \mathbb{C}$, let
\[
\diffs(z):=\frac{| \diff z|}{2\pi}, \quad \dA(z):= \frac{\diff x \diff y}{\pi}
\]
denote the normalized arc length and area measures as indicated.
Moreover, we shall write
\[
\varDelta_z:=\frac{1}{4} \left( \frac{\partial}{\partial x^2}+
\frac{\partial}{\partial y^2}\right)
\]
for the normalized Laplacian, and 
\[
\partial_{z}:=\frac{1}{2} \left(\frac{\partial}{\partial x}-\ima
\frac{\partial}{\partial y}\right),\quad 
\bar{\partial}_{z}:=\frac{1}{2} \left(\frac{\partial}{\partial x}+
\ima \frac{\partial}{\partial y}\right)
\]
for the standard complex derivatives; then $\varDelta_z$ factors as
$\varDelta_z=\partial_{z}\bar{\partial}_{z}$. Often, we will drop the
subscript for these differential operators when it is obvious from the
context with respect to which variable they apply. We let $\mathbb{D}$
denote the open unit disk, $\mathbb{T}:=\partial \D$, and $\D_{\e}$
the (punctured) exterior disk:
\[
\D:=\{ z\in \mathbb{C}: |z|<1 \},\quad
\D_{\e}:=\{ z\in \mathbb{C}: |z|>1 \}.
\]
Often, we will add the point infinity to the exterior disk $\D_\e$, in which
case it really becomes a disk. If we do so, we mention it explicitly in
the text.
As for holomorphic functions in $\D_\e$, if the function is bounded near
infinity, it extends holomorphically across $\infty$ (Liouville theorem). 
More generally, we write
\[
\D(z_0,r):=\{ z\in \mathbb{C}: |z-z_0|<r \}
\]
for the open disk of radius $r$ centered at $z_0$. 
Let $\D^2=\{ (z,w)\in \mathbb{C}^2:\; |z|<1,\; |w|<1  \}$ be the open
unit bidisk in $\mathbb{C}^2$. If $\zeta\in\mathbb{C}^2$ and $r>0$,
then $\D^2(\zeta,r)$ denotes the open bidisk in $\mathbb{C}^2$ of radius
$r$ centered at $\zeta$.  
We will find it useful to introduce the sesquilinear forms
$\la \cdot,\cdot\ra_{\mathbb{T}}$ and $\la \cdot,\cdot\ra_{{\Omega}}$,
(where $\Omega=\D$, or $\D_{\e}$) as given by
\[
\la f,g\ra_{\mathbb{T}}:=\int_{\mathbb{T}}f(z)\bar{g}(z)\diff s(z),\quad
\la f,g\ra_{{\Omega}}:=\int_{{\Omega}}f(z)\bar{g}(z)\dA (z),
\]
where in the first case, $f\bar{g}\in L^2(\mathbb{T})$ is required,
and in the second, we need that  $f\bar{g}\in L^2({\Omega})$.
These are standard Lebesgue spaces with respect to the normalized area
measure $\dA$. Here, more generally, for a given complex-valued function
$f$, we denote by $\bar{f}$ the function whose values are the complex
conjugates of $f$. To simplify the notation further, we write
\[
\la f \ra_{\mathbb{C}}=\la f,1 \ra_{\mathbb{C}},\quad \la f
\ra_{{\Omega}}=\la f,1 \ra_{{\Omega}}.
\]
For operators $\Tope$ on a Hilbert space, we let  $\Tope^*$
denote the adjoint, while $\bar{\Tope}$ means the operator defined by
\[
\bar{\Tope}f=\overline{\Tope\bar{f}}.
\]

\subsection{The standard weighted Bergman spaces}
\label{ss:weightedBergman}
For $0<p<+\infty$ and $-1<\alpha<+\infty$, we introduce the
scale of standard weighted Lebesgue spaces $L^p_{\alpha}(\D)$ of
(equivalent classes of) Borel measurable functions $f:\D\raro\mathbb{C}$ with
\[
\|f\|^p_{L^p_{\alpha}(\D)}:=(\alpha+1)\int_{\D} |f(z)|^p(1-|z|^2)^{\alpha} \dA (z)
<+\infty.
\]
Then $L^p_\alpha(\D)$ is a Hilbert space for $p=2$, a Banach space for
$1\le p<+\infty$, and a quasi-Banach space for $0<p<1$. 
We write that $f\in A^p_{\alpha}(\D)$ to indicate that $f$ is holomorphic in
$\D$ and that $f\in L^p_{\alpha}(\D)$. In this case, we will often write
$\|\cdot\|_{A^p_{\alpha}(\D)}$ in place of $\|\cdot\|_{L^p_{\alpha}(\D)}$.
The spaces $A^p_{\alpha}(\D)$ are known as \emph{the standard weighted
Bergman spaces} on $\D$.
For $\alpha=0$, then, we recover the Bergman spaces $A^p_{0}(\D)=A^p(\D)$.
For $\alpha\le -1$, the norm expression does not make sense because of the
factor $(\alpha+1)$, but if it is removed, then the weighted Bergman space
would necessarily be trivial in that case $\alpha\le-1$. However, if we let
$\alpha$ approach $-1$ and keep the factor $(\alpha+1)$, then, in the limit
we have, e.g., for polynomials $f$,
\[
\lim_{\alpha\raro -1^{+}}\|f\|^p_{L^p_{\alpha}(\D)}=\int_{\T}|f|^p\diffs
=\|f\|^p_{H^p(\D)},
\]
where on the right-hand side appears the Hardy space $H^p(\D)$ norm
(quasi-norm if $0<p<1$), given by
\[
\|f\|^p_{H^p(\D)}:=\sup_{0<r<1}\int_{\T} |f(r\zeta)|^p
\diffs(\zeta)<+\infty.
\]
In this sense, we may assert that $H^p(\D)$ arises as the limit of the
spaces $A^p_{\alpha}(\Omega)$ as $\alpha\raro -1^{+}$.


Similarly, we define the standard weighted Bergman space
$A^p_{\alpha}(\D_{\e})$ on $\D_{\e}$, which consists of holomorphic
functions $f:\D_{\e}\raro \mathbb{C}_{\infty}$ with 
\[
\|f\|^p_{A^2_{\alpha}(\D_{\e})}=\int_{\D_e}|f(z)|^p(1-|z|^{-2})^{\alpha}
\frac{\dA (z)}{|z|^{4}}<+\infty.
\]
Let $H^p(\D_{e})$ be the Hardy space on the exterior disk $\D_{\e}$, that is, 
\[
H^p(\D_{e})=\left\{f:\D_{\e}\raro\mathbb{C}_{\infty}: \sup_{1<r<+\infty}
\int_{\partial \D} |f(r\zeta)|^p\diffs(\zeta)  \right\}.
\]
The Hardy space $H^p(\D_{e})$ appears as the limit of the spaces
$A^p_{\alpha}(\D_{\e})$ as $\alpha\raro -1^{+}$.

\subsection{The Bloch space and the Bloch seminorm}
The \emph{Bloch space} consists of those holomorphic functions
$g:\D\raro \mathbb{C}$ that are subject to seminorm boundedness condition
\[
\|g\|_{\clb(\D)}:= \sup_{z\in \D}(1-|z|^2)|g^{\prime}(z)|<+\infty.
\]
Let $\text{Aut}(\D)$ denote the group of bijective sense-preserving
M\"{o}bius automorphisms $\D\to\D$. By direct calculation
\[
\|g\circ \gamma\|_{\clb(\D)}=\|g\|_{\clb(\D)},
\]
for all $\gamma\in \text{Aut}(\D)$ and $g\in\clb(\D)$. This says
that the Bloch seminorm is invariant under all 
M\"{o}bius automorphisms of $\D$. The subspace
\[
\clb_0(\D):=\Big\{g\in \clb(\D):
\lim_{|z|\raro 1^{-}}(1-|z|^2)|g^{\prime}(z)|=0\Big\}
\]
is called the \emph{little Bloch space}.
An immediate observation we can make at this point is that provided
that $g(0)=0$, we have the estimate
\[
|g(z)|\leq \|g\|_{\clb(\D)}\int_0^{|z|}\frac{\diff t}{1-t^2}
=\frac{1}{2}\|g\|_{\clb(\D)}\log \frac{1+|z|}{1-|z|},\quad z\in \D,
\]
which is sharp pointwise.

Similarly, we introduce the Bloch space $\clb(\D_{\e})$ on $\D_{\e}$,
which consists of holomorphic functions $g:\D_{\e}\raro \mathbb{C}_{\infty}$
with
\[
\|g\|_{\clb(\D_{\e})}:=\sup_{z\in \D_{\e}}(|z|^2-1)|g^{\prime}(z)|<+\infty.
\]

\subsection{Dirichlet operator symbols}

For $z\in \D$, let $s_z$ denote the Szeg\H{o} kernel
\[
s_z(\zeta):=\frac{1}{1-\bar{z}\zeta},\qquad \zeta\in \D.
\]
For functions in the Bergman space $A^2(\D)$, taking the inner product
with $s_{z}$ is the same as finding the average
\[
\la f, s_z\ra_{\D}=\int_0^1 f(zt)\diff t,\qquad f\in A^2(\D).
\]
In the definition below, we flip the roles of $z,w$ compared with \cite{GAF}.
We do this to adhere to the notational convention for integral operators
used later on. 

\begin{defn}
Let $\Tope$ be a bounded linear operator on $L^2(\D)$. The
\emph{Dirichlet operator symbol associated with $\Tope$} is the function 
\[
\mathscr{P}[\Tope](z,w):=\la \Tope(\bar{s}_w), s_z \ra_{\D},\qquad (z,w)\in\D^2,
\]
which is holomorphic in the bidisk $\D^2$. The \emph{diagonal part} of the
Dirichlet operator symbol is the diagonal restriction
$\oslash\mathscr{P}[\Tope](z)=\mathscr{P}[\Tope](z,z)$.
\end{defn}

\begin{rem}
Sometimes the associated function
$\mathscr{Q}[\Tope](z,w)=zw\mathscr{P}[\Tope](z,w)$ is more natural to work
with, and then by slight abuse of terminology we refer to it as the
Dirichlet operator symbol of $\Tope$ as well. The same goes for its diagonal
part $\oslash\mathscr{Q}[\Tope]$. 
\end{rem}

As we apply the Cauchy-Schwarz inequality, we arrive at the pointwise estimate
\begin{multline*}
|\mathscr{P}[\Tope](z,w)|\leq \|T\|\|s_w\|_{L^2(\D)}\|s_z\|_{L^2(\D)}
\\
=\|\Tope\|\left(  \log \frac{1}{1-|z|^2}\right)^{\frac{1}{2}}
\left(  \log \frac{1}{1-|w|^2}\right)^{\frac{1}{2}},\quad z,w\in \D.
\end{multline*}
While this is best possible pointwise, perhaps we can do better if we
instead consider various weighted averages.
A case in point is that of estimating the average growth of the diagonal
part of the operator symbol:
\[
\oslash\mathscr{P}[\Tope](z)=\la \Tope(\bar{s}_z), s_z \ra_{\D}, \qquad z\in\D.
\]
If $\Tope=\Mop_{\mu}$, the operator of multiplication by a bounded function
$\mu\in L^{\infty}(\D)$, then we recover the well-known Bergman projection
of $\mu$:
\begin{equation}
\label{bergman projection}
\Bergman \mu(z):=
\oslash\mathscr{P}[\Mop_{\mu}](z)=\la \Mop_{\mu}(\bar{s}_z), s_z \ra=
\int_{\D} \frac{\mu(\xi) }{(1-z\bar{\xi})^2}\dA (\xi), \qquad z\in\D.
\end{equation}
In view of this, we may think of $\oslash\mathscr{P}[\Tope]$ as extending the
notion of the Bergman projection to the setting of general bounded operators
$\Tope:\,L^2(\D)\to L^2(\D)$.
 
In this section, we shall study the functions $Q(z,w)=\mathscr{Q}[\Tope](z,w)$
as holomorphic functions on the bidisk $\D^2$. Our first observation
is the following.

\begin{lem}
\label{lem:Q-1}
If $Q(z,w)=\mathscr{Q}[\Tope](z,w)$ for a bounded operator $\Tope$ on
$L^2(\D)$, then
\[
Q(z,0)=Q(0,w)=0,\qquad z,w\in\D,
\]
and, consequently,
\[
Q(z,w)=\int_{0}^{w}\int_0^z \partial_\xi\partial_\eta Q(\xi,\eta)\,\diff \xi
\diff \eta.
\]
\end{lem}

\begin{proof}
From the definition of $\mathscr{Q}[\Tope]$ the first property is immediate.
After that, the second property is a consequence of Calculus.
\end{proof}

We now establish that the Dirichlet
operator symbol $\mathscr{Q}[\Tope](z,w)$ for a given bounded operator
$\Tope:L^2(\D)\to L^2(\D)$ captures a piece of $\Tope$, namely the compression
$\Pop\Tope\bar{\Pop}$, where
$\Pop=\Pop_{A^2(\D)}:L^2(\D)\to L^2(\D)$
stands for the orthogonal projection onto $A^2(\D)$, and hence $\bar\Pop$ is
the orthogonal projection onto the space $\mathrm{conj}\,A^2(\D)$ of
complex-conjugates from $A^2(\D)$.

\begin{prop}\label{uniqueness:1}
Let $\Tope$ be a bounded operator on $L^2(\D)$, and put
$\Pop=\Pop_{A^2(\D)}$ and $Q(z,w)=\mathscr{Q}[\Tope](z,w)$. Then
\[
\Pop\Tope\bar\Pop f(z)=\int_\D \partial_z\partial_w Q(z,w)f(w)\,\diffA(w),
\qquad z\in\D.
\]
\end{prop}

\begin{proof}
We first observe that for $f,g\in L^2(\D)$,
\[
\Pop f(z)=\langle f,\Berg_z\rangle_\D\quad\text{and}\quad
\bar\Pop g(z)=\langle g,\bar \Berg_z\rangle_\D,
\]
for any $z\in\D$, where $\Berg_z(\xi):=(1-\bar z\xi)^{-2}$ is the
Bergman kernel.
A simple calculation gives that
\begin{equation*}
\partial_z\partial_w Q(z,w)=
\partial_z\partial_w \mathscr{Q}[\Tope](z,w)=\langle \Tope
\bar \Berg_w,\Berg_z\rangle_\D=\Pop\Tope[\bar\Berg_w](z),
\end{equation*}
so that
\begin{equation*}
\int_\D \partial_z\partial_w Q(z,w)\,f(w)\diffA(w)=
\int_\D\Pop\Tope[\bar\Berg_w](z)f(w)\diffA(w)
=\Pop\Tope\bar\Pop f(z),
\end{equation*}
as claimed. 
%
\end{proof}

Let $\scrD(\D)$ denote the classical Dirichlet space on the disk $\D$.
It consists of all holomorphic functions $f$ on $\D$ whose derivative
$f^{\prime}$ is square area-integrable, that is, 
\begin{equation}
\int_{\D}|f^{\prime}(z)|^2 \dA(z)<+\infty.
\label{eq:Dir-1}
\end{equation}
The expression \eqref{eq:Dir-1} only defines a Hilbert seminorm on $\scrD(\D)$,
as the constants disappear after differentiation. One way to fix this problem
is to consider the subspace $\scrD_0(\D)$ consisting of all functions
$f\in\scrD(\D)$ that vanish at the origin:
\[
\scrD_0(\D):=\{ f\in \scrD(\D):\, f(0)=0 \}.
\]
The expression \eqref{eq:Dir-1} then defines a Hilbert space norm on
$\scrD_0(\D)$.
We try to characterize the Dirichlet operator symbols associated with
contractions on $L^2(\D)$ via pairs of orthonormal sequences in $\scrD_0(\D)$.
The first installment reads as follows.

\begin{thm}
Let  $\{a_j\}_j$ and $\{b_j\}_j$ be two orthonormal
bases of $\scrD_0(\D)$, while $\{t_j\}_j$ is a sequence of real numbers with
$0\le t_j\le1$. Here, the index runs over all $j\in\Z_{>0}$. Next, we consider
the holomorphic function
\[
Q(z,w)=\sum_{j=1}^{+\infty}t_j a_j(z)b_j(w),\qquad (z,w)\in\D^2.
\]
Then there exists a contraction $\Tope$ on $L^2(\D)$ for which
\[
Q(z,w)=\mathscr{Q}[\Tope](z,w),\qquad z,w\in \D.
\]
\end{thm}

\begin{proof}
In view of Proposition \ref{uniqueness:1}, we choose as $\Tope$ the operator 
\[
\Tope{f}(z):=\int_{\D} \partial_z\partial_w Q(z,w) f(w) \dA (w),\qquad
z\in\D.
\]
We need to check that $\Tope$ defines a contraction on $L^2(\D)$, and that
$\Tope=\Pop\Tope\bar\Pop$. Then, by Lemma \ref{lem:Q-1} and Proposition
\ref{uniqueness:1}, we must have $Q=\mathscr{Q}[\Tope]$ on $\D^2$.
We calculate that
\[
\partial_z\partial_w Q(z,w)=
\sum_{j=1}^{+\infty}t_j a_j'(z)b_j'(w),\qquad (z,w)\in\D^2,
\]
so that $(\partial_z\partial_w Q)(z,\cdot)\in A^2(\D)$ for $z\in\D$, since
the fact that $\{a_j\}_j$ and $\{b_j\}_j$ are orthonormal bases in
$\scrD_0(\D)$ entails that

\[
\int_\D|\partial_z\partial_w Q(z,w)|^2\diffA(w)=\sum_{j=1}^{+\infty}
t_j^2|a_j'(z)|^2\le \sum_{j=1}^{+\infty}|a_j'(z)|^2=\frac{1}{(1-|z|^2)^2},
\qquad z\in\D.
\]
In the last step, we used the assumption that $0\le t_j\le1$. In view of the
above definition of the operator $\Tope$, It now follows that $\Tope=\Tope\bar
\Pop$ since if $f\in L^2(\D)\ominus \mathrm{conj}\,A^2(\D)$, we get
$\Tope f=0$.
We next verify that $\Tope$ defines a contraction on $L^2(\D)$. To this end,
we calculate that
\begin{multline*}
\Tope f(z)=\int_{\D} \partial_z\partial_w Q(z,w) f(w) \dA (w)
=\int_{\D} \sum_{j=1}^{+\infty}t_j a_j^{\prime}(z)b_j^{\prime}(w)f(w)
\dA (w)
\\
=\sum_{j=1}^{+\infty}t_j a_j^{\prime}(z) \int_{\D}  b_j^{\prime}(w){f(w)} \dA (w)
=\sum_{j=1}^{+\infty}t_j a_j^{\prime}(z) \langle f, \bar b_j^{\prime}\rangle_\D.
\end{multline*}
As $\{b_i^{\prime}\}_{i=1}^{+\infty}$ forms an orthonormal basis for $A^2(\D)$,
the assumption that $f\in L^2(\D)$ entails that the sequence
$\{\langle f,\bar b_j'\rangle_\D\}_j$ is in $\ell^2$, and hence, in a second
step, that $\Tope f=\Pop\Tope f$. Moreover, a Plancherel-type argument
shows that
\begin{multline*}
\|\Tope f\|_{A^2(\D)}^2=\bigg\|
\sum_{j=1}^{+\infty}t_j  \langle f, \bar b_j^{\prime}\rangle_\D
\,a_j^{\prime}\bigg\|^2_{A^2(\D)}=
\sum_{j=1}^{+\infty}t_j^2 \big|\langle f, \bar b_j^{\prime}\rangle_\D\big|^2
\\
\le \sum_{j=1}^{+\infty}\big|\langle f, \bar b_j^{\prime}\rangle_\D\big|^2
=\|\bar\Pop f\|^2_{L^2(\D)}\le \|f\|_{L^2(\D)}^2. 
\end{multline*}
In particular, $\Tope$ acts contractively on $L^2(\D)$. Finally, since
it has already been established that $\Tope=\Tope\bar\Pop$ as well as
$\Tope=\Pop\Tope$, it follows that $\Tope=\Pop\Tope\bar\Pop$ as well.
\end{proof}

While it would appear that the converse of the above theorem fails in general,
it is not so far from being true, as evidenced by the following theorem.
The proof approach is basically analogous to that of Theorem 1.9.3 in
\cite{GAF}.

\begin{thm}
Let $Q:\D^2\raro \mathbb{C}$ be a holomorphic function. If
$Q=\mathscr{Q}[\Tope]$ holds for a contraction $\Tope$ on $L^2(\D)$, then 
\[
Q(z,w)=\lim_{n\raro+\infty} \sum_{j=1}^{n}t_{j,n} a_{j,n}(z)b_{j,n}(w).
\]
where $\{t_{j,n}\}_{j=1}^n\subset [0,1]$ while $\{a_{j,n}\}_{j=1}^{n}$ and
$\{b_{j,n}\}_{j=1}^{n}$ are orthonormal sequences in $\scrD_0(\D)$ for each
$n\in \Z_{>0}$.
\end{thm}

\begin{proof}
Let $\Tope$ be a contraction on $L^2(\D)$ and let $Q=\mathscr{Q}[\Tope]$. It
follows from Proposition \ref{uniqueness:1} that
\[
\Pop \Tope\bar{\Pop} {f}(z)=
\int_{\D} \partial_z\partial_w \mathscr{Q}[T](z,w) {f(z)} \dA (w),\qquad
z\in\D.
\]
Let $\Jop:A^2(\D)\raro\text{conj}\,A^2(\D)$ be the isometric linear map
given by $\Jop e_j=\bar{e}_j$ for $j=1,2,3,\dots$, where
$e_j(z):=j^{-\frac12}z^{j-1}$ denotes the standard orthonormal basis in $A^2(\D)$.
We consider the contraction
$\tilde{\Tope}:=\Pop \Tope\bar\Pop\Jop:\,A^2(\D)\to A^2(\D)$.
If $\{\beta_1,\beta_2,\beta_3,\dots\}$ denotes any orthonormal basis of
$A^2(\D)$ while $\Pi_n$ is the orthogonal projection of  $A^2(\D)$ onto
the $n$-dimensional subspace $\text{span}\{\beta_1,\dots,\beta_n\}$, then,
since $\Pi_n f\to f$ and $\Pi_n g\to g$ in norm as $n\to+\infty$ for fixed
$f,g\in A^2(\D)$, we have that
\[
\la \tilde\Tope f,g \ra_\D=
\lim_{n\raro\infty}
\la\tilde{\Tope}\Pi_n f,\Pi_n g\ra_\D
=\lim_{n\to+\infty}\la \Pi_n \tilde{\Tope}\Pi_n f,g\ra_\D.
\]
As the operator $\Pi_n \tilde{\Tope}\Pi_n$ has finite rank $\le n$, it
admits a singular value decomposition. Hence it can be written in the form
\begin{equation}
\label{eq:SVD}
\Pi_n \tilde{\Tope}\Pi_n f=
\sum_{j=1}^{n} t_{j,n} \la f, u_{j,n}\ra_\D v_{j,n},
\qquad f\in A^2(\D),
\end{equation}
where $t_{j,n}\in [0,1]$, and $\{v_{j,n}\}_{j=1}^{n}$ and $\{u_{j,n}\}_{j=1}^{n}$
are orthonormal sequences in $A^2(\D)$.
If we let $\bar s_w=\Jop r_w$, where $r_w=\Jop^\star\bar{s}_w\in A^2(\D)$,
we obtain that
\begin{multline}
\label{eq:SVD-0}
Q(z,w)=zw\la \Tope\bar{s}_w, s_z \ra=zw\la \Pop\Tope\bar\Pop\bar{s}_w, s_z
\ra_\D=
zw \la\Pop\Tope\bar\Pop\Jop r_w, s_z  \ra_\D
\\
=zw\la \tilde{\Tope}r_w, s_z\ra_\D=
\lim_{n\raro\infty}zw\la  \Pi_n\tilde{\Tope} \Pi_{n}r_w, s_z\ra_\D.
\end{multline}
Moreover, in view of the singular value decomposition \eqref{eq:SVD}, we have
\begin{equation}
\label{eq:SVD-1}
\la  \Pi_n\tilde{\Tope} \Pi_{n}r_w, s_z\ra_\D=
\sum_{j=1}^{n}t_{j,n}\la r_w, u_{j,n}\ra \la v_{j,n}, s_z\ra_\D
=\sum_{j=1}^{n}t_{j,n}\la  \Jop^\star \bar s_w,u_{j,n}\ra_\D
\la v_{j,n}, s_z\ra_\D,
\end{equation}
for all $z,w\in\D$.
For $j=1,\dots,n$, we put
\[
a_{j,n}(z):=z\langle v_{j,n},s_z\rangle_\D,\quad
b_{j,n}(w):=w\langle \Jop^\star \bar s_w,u_{j,n}\rangle_\D=
w\langle \bar s_w,\Jop u_{j,n}\rangle_\D,  
\]
so that $a_{j,n}(0)=b_{j,n}(0)=0$, and
\[
a_{j,n}'(z)=\langle v_{j,n},\Berg_z\rangle_\D=v_{j,n}(z),\quad
b_{j,n}'(w)=
\langle \bar \Berg_w,\Jop u_{j,n}\rangle_\D=\overline{\Jop u_{j,n}(w)}.  
\]
Since $\{u_{j,n}\}_j$ and $\{v_{j,n}\}_j$ are orthonormal in $A^2(\D)$,
and hence that $\{\Jop u_{j,n}\}_j$ is orthonormal in $\mathrm{conj}\,A^2(\D)$,
it follows that $\{a_{j,n}\}_j$ and $\{b_{j,n}\}_j$ are orthonormal in
$\scrD_0(\D)$, and by \eqref{eq:SVD-0} combined with \eqref{eq:SVD-1},
\[
Q(z,w)=\lim_{n\to+\infty}\sum_{j=1}^{n}t_{j,n}\,a_{j,n}(z)b_{j,n}(w),
\]
as claimed.
\end{proof}

\subsection{Dirichlet operator symbols of multiplication operators}

It was mentioned earlier that when $\Tope=\Mop_\mu$, the operator of
multiplication by $\mu\in L^\infty(\D)$, with $\|\mu\|_{L^{\infty}(\D)}\leq 1$,
the diagonal part $\oslash \mathscr{P}[\Mop_\mu]$ agrees with the Bergman
projection $\Pop\mu$ of $\mu$, see \eqref{bergman projection}.
We now try to analyze the structure of the operator symbol
$\mathscr{Q}[\Mop_\mu](z,w)=zw\mathscr{P}[\Mop_\mu](z,w)$
as a function of two variables. As it turns out, it
solves
a linear damped wave equation.

\begin{thm}
Let $\mu\in L^{\infty}(\D)$ with $\|\mu\|_{L^{\infty}(\D)}\leq 1$ and
let $Q:=\mathscr{Q}[M_{\mu}]$.  Then $Q$ has
\[
Q(z,0)=Q(0,w)=0,\qquad z,w\in\D,
\]
while it also solves the linear damped wave equation
\[
\partial_z\partial_w Q(z,w)=
\frac{z^2\partial_z Q(z,w)-w^2\partial_w Q(z,w)}{zw(z-w)},\quad (z,w)\in \D^2.
\]
\end{thm}

\begin{proof}
We first note that
\begin{multline*}
Q(z,w)=\mathscr{Q}[M_{\mu}](z,w)
=zw\la M_{\mu}(\bar{s}_w), \; s_z\ra_{\D} 
=zw\int_{\D}  (M_{\mu} \bar{s}_w) (\xi)  \overline{s_z (\xi)} \dA (\xi)
\\
=zw\int_{\D}\frac{\mu(\xi)}{(1-z\bar{\xi})(1-w\bar{\xi})} \dA (\xi),
\qquad (z,w)\in \D^2.
\end{multline*}
In particular, this shows that
\[
Q(z,z)=z^2\int_{\D}\frac{\mu(\xi)}{(1-z\overline{\xi})^2} \dA (\xi)
=z^2\Pop\mu (z),\quad z\in \D,
\]
where $\Pop$ denotes the Bergman projection, cf. equation
\eqref{bergman projection}.
Next, we differentiate $Q(z,w)$ with respect to $z$ and $w$, respectively,
to obtain
\[
\partial_z Q(z,w)=w\la M_{\mu} \bar{\Berg}_z, s_w\ra
=w\int_{\D} \frac{\mu(\xi)}{(1-z\bar{\xi})^2(1-w\bar{\xi})} \dA (\xi).
\]
and
\[
\partial_w Q(z,w)=z\la M_{\mu} \bar{s}_z, \Berg_w\ra
=z\int_{\D} \frac{\mu(\xi)}{(1-z\bar{\xi})(1-w\bar{\xi})^2} \dA (\xi).
\]
From the algebraic identity
\[
\frac{1}{(1-z\bar{\xi})^2(1-w\bar{\xi})^2}
=\frac{1}{z-w}\left( \frac{z}{(1-z\bar{\xi})^2(1-w\bar{\xi})}
-\frac{w}{(1-z\bar{\xi})(1-w\bar{\xi})^2} \right),
\]
which holds for all $(z,w)\in \D^2$ with $z\neq w$, it follows that
\begin{equation}
\label{eq:algid-1}
\int_{\D}\frac{\mu(\xi)}{(1-z\bar\xi)^2(1-w\bar{\xi})^2}\dA(\xi)
=\frac{z^2\partial_zQ(z,w)-w^2\partial_wQ(z,w)}{zw(z-w)}.
\end{equation}
Again, differentiating $Q(z,w)$ with respect to $z$ and $w$ yields
\[
\partial_z \partial_w Q(z,w)=\la M_{\mu} \bar{b}_z,b_w\ra
=\int_{\D}\mu(\xi) \overline{\Berg_z(\xi)}\overline{\Berg_w(\xi)}
\dA (\xi),
\]
which we identify with the left-hand side of \eqref{eq:algid-1}. 
It now follows that $Q$ solves the linear damped wave equation
\[
\partial_z\partial_w Q(z,w)=
\frac{z^2\partial_z Q(z,w)-w^2\partial_w Q(z,w)}{zw(z-w)},\qquad z,w\in \D.
\]
The proof of the theorem is complete.
\end{proof}

The converse of the above theorem does not hold in general. After all,
since the damped wave equation is linear, we may my multiply $Q$ by a scalar
and the equation remains the same. But the equality $Q=\mathscr{Q}[\Mop_\mu]$
is not invariant under scalar multiplication, since $\mu$ is supposed to be
in the unit ball of $L^\infty(\D)$. So to get a necessary and sufficient
condition, we need to take the unit ball condition into account.
Here, we use a result of Coifman, Rochberg, and Weiss \cite{CRW}, which
asserts that the Bergman projection $\Bergman$ maps $L^{\infty}(\D)$ onto
the Bloch space $\clb(\D)$. Thus, $\Bergman L^{\infty}(\D)$ can be identify
with $\clb(\D)$ using the alternative norm 
\[
\|g\|_{\Bergman L^{\infty}(\D)}=\inf\big\{\|\mu\|_{L^{\infty}(\D)} :
\mu\in L^{\infty}(\D)
\quad\text{and}\quad g=\Bergman\mu\big\}.
\]
In this sense, in view of the Hahn-Banach theorem,
$\Bergman L^{\infty}(\D)$ is isometrically isomorphic to the
dual of $A^1(\D)$, the set of all holomorphic and area-integrable functions on
$\D$, with respect to the dual pairing
\[
\la f,g\ra_{\D}:=\lim_{r\raro 1^-} \la f_r,g\ra_{\D},
\]
where $f\in A^1(\D)$,  $f_r(z)=f(rz)$, and $g=P\mu$, for
$\mu\in L^{\infty}(\D)$. It then follows that
\[
\la f,g \ra_{\D}=\la f, P\mu \ra_{\D}=\la f,\mu \ra_{\D}.
\]
The (perhaps non-unique) optimal $\mu\in L^\infty(\D)$ for a given $g\in\Pop
L^\infty(\D)$ with $g=\Pop\mu$ comes from application of the Hahn-Banach
theorem. This argument was used in the analysis of $\Pop L^\infty(\D)$ in
the context of the concept of asymptotic variance in \cite{tail variance},
\cite{geometric zero packing}. 

\begin{thm}
Suppose $Q:\D^2\raro\mathbb{C}$ is a holomorphic function with 
\[
Q(z,0)=Q(0,w)=0,\qquad z,w\in\D,
\]
which solves the linear damped wave equation
\[
\partial_z\partial_w Q(z,w)=
\frac{z^2\partial_z Q(z,w)-w^2\partial_w Q(z,w)}{zw(z-w)},\quad (z,w)\in \D^2.
\]  
Then $Q$ has the form $Q=\mathscr{Q}[\Mop_\mu]$ for some $\mu$ in the closed
unit ball of $L^{\infty}(\D)$ if and only if
\begin{equation}
\label{eq:L}
\Big|\int_{\D} z^{-2}Q(z,z) \overline{f(z)} \dA (z)\Big|\le \|f\|_{A^1(\D)},
\qquad f\in A^1(\D).  
\end{equation}
\end{thm}

\begin{proof}
We already checked that $Q=\mathscr{Q}[\Mop_\mu]$ with $\mu$ in the unit ball
of $L^\infty(\D)$ has the indicated properties, so we focus our efforts on the
``if'' part of the assertion. 

Substituting $\xi:=1/z$ and $\eta:=1/w$, the above damped wave equation
simplifies to
\[
\partial_{\xi}\partial_{\eta} Q(\xi^{-1},\eta^{-1})=
\frac{\partial_{\xi}Q(\xi^{-1},\eta^{-1})-
\partial_{\eta}Q(\xi^{-1},\eta^{-1})}{\xi-\eta},
\]
which is seen to be equivalent to the relation
\[
\partial_{\xi}\partial_{\eta} \left\{(\xi-\eta)Q(\xi^{-1},\eta^{-1})\right\}=0,
\]
a standard wave equation. 
Hence, locally in $\D_\e$ we have
\[
(\xi-\eta)Q(\xi^{-1},\eta^{-1})=G_1(\xi)+G_2(\eta),
\]
where $G_1,G_2$ are locally defined holomorphic functions on $\D_{\e}$,
except that we will need to rule out possible branching at infinity.
Letting  $\xi\raro \eta$, we find $G_1(\eta)=-G_2(\eta)$.
Thus, the relation simplifies to
\[
(\xi-\eta)Q(\xi^{-1},\eta^{-1})=G_1(\xi)-G_1(\eta),
\]
that is,
\[
Q(\xi^{-1},\eta^{-1})=\frac{G_1(\xi)-G_1(\eta)}{\xi-\eta}.
\]
Along the diagonal, we find that
\[
G_1^{\prime}(\xi)=Q(\xi^{-1}, \xi^{-1}).
\]
Moreover, since
$Q(\xi^{-1},\xi^{-1})=\Ordo(\xi^{-2})$ holds, no branching at infinity will
occur when we declare that $G_1(\xi)$ is any suitable primitive to
$Q(\xi^{-1},\xi^{-1})$. This way, $G_1(\xi)$ gets to be globally defined on
$\D_\e$, and we write $H_1(z):=G_1(z^{-1})$ for $z\in\D$. By the alluded-to
duality between $A^1(\D)$ and $\Pop L^\infty(\D)$, the integral
bound \eqref{eq:L} entails that $\oslash Q(z)=Q(z,z)$ is of the form
\[
Q(z,z)=z^2\Pop\mu(z),\qquad z\in\D,
\]
for some $\mu$ in the unit ball of $L^\infty(\D)$. But then
\[
H_1'(z)=-z^{-2}G_1'(z^{-1})=-z^{-2}Q(z,z)=-\Pop\mu(z),\qquad z\in\D,
\]
and the formula
\[
H_1(z)=-z\int_\D\frac{\mu(\tau)}{1-z\bar\tau}\,\dA(\tau)
\]
supplies a suitable primitive. Calculating backwards, then, we obtain that
\begin{multline*}
Q(z,w)=\frac{G_1(z^{-1})-G_1(w^{-1})}{z^{-1}-w^{-1}}
=zw\,\frac{H_1(z)-H_1(w)}{w-z}
\\
=\frac{zw}{w-z}\int_\D
\Big(\frac{z}{1-z\bar\tau}-\frac{w}{1-w\bar\tau}\Big)\mu(\tau)\dA(\tau)
=zw\int_\D \frac{\mu(\tau)}{(1-z\bar\tau)(1-w\bar\tau)}\dA(\tau),
\end{multline*}
which we identify with $\mathscr{Q}[\Mop_\mu](z,w)$. This completes the
proof of the theorem.
\end{proof}

\subsection{Conformal mappings: the standard classes
\texorpdfstring{$\mathcal{S}$ and $\Sigma$}{S and Sigma}}
It is a central theme in the theory of conformal mapping to analyze
the local dilation/contraction/rotation of the mapping in question.
It is natural to study such conformal mappings collectively. 
One such collection is the standard class $\cls$ of univalent functions
$\vp:\D\raro \mathbb{C}$ subject to the normalization
$\vp(0)=0$ and $\vp^{\prime}(0)=1$. We consider the function
$h_{\vp}(z):=\log \vp^{\prime}(z)$, which may be referred to as the
\emph{local complex distortion exponent}. As a logarithm, it is fixed by the
requirement that $h_{\vp}(0)=0$. A classical estimate of
$h_{\vp}$ (due to Koebe and Bieberbach) is the inequality
\[
|(1-|z|^2)h_{\vp}^{\prime}(z)-2\bar{z}|=\bigg|
(1-|z|^2)\frac{\vp^{\prime \prime}(z)}{\vp^{\prime}(z)}-2\bar{z}\bigg|\leq 4.
\]
In particular, $h_{\vp}$ is in the Bloch space, with seminorm estimate
$\|h_{\vp}\|_{\clb(\D)}\leq 6$.
On the other hand, Becker's univalence criterion asserts that if
$\vp:\D\raro \mathbb{C}$ is a function which is locally univalent,
that is, $\vp^{\prime}(z)\neq 0$ for all $z\in \D$, with normalization
$\vp(0)=0$ and $\vp'(0)=1$, and in addition,
$\|h_{\vp}\|_{\clb(\D)}\leq 1$, then $\vp$ is necessarily univalent. Note that
the function $h_{\vp}$ makes sense even if $\vp$ is not univalent, as soon as
$\varphi'$ is zero-free.
The Bloch seminorm bound $1$ which appears here is best possible
(see \cite{becker, B_P}). The behavior of $h_{\vp}=\log \vp^{\prime}(z)$
may acquire additional boundary growth if the image domain $\vp(\D)$ is
unbounded, because, after all, the derivative $\vp^{\prime}$ is taken with
respect to the Euclidean structure in the image $\vp(\D)\subset \mathbb{C}$.
To avoid taking such effects into consideration, we can pass to the
univalent function $\psi:\D_{\e}\raro \mathbb{C}_{\infty}$ given by
\[
\psi(\zeta)=\frac{1}{\vp(1/\zeta)}, \quad \zeta\in \D_{\e},
\]
which has $\psi(\zeta)=\zeta+O(1)$ as $\zeta\raro\infty$ and hence is
an element of the class $\Sigma$. As for $\psi$, we know that the complement
of the image domain $\psi(\D)$ is a compact continuum which does not divide
the plane, contains the origin, and has diameter at most 4.
Moreover, any function $\psi\in \Sigma$ has a Laurent series expansion of
the form
\[
\psi(z)=z+b_0+\frac{b_1}{z}+\frac{b_2}{z^2}+\cdots,\quad z\in \mathbb{D}_e.
\]

\subsection{Grunsky operators}

For a given $\psi\in \Sigma$, we can define the 
Grunsky matrix $\Grunsky_{\psi}=(\sqrt{jk}\gamma_{j,k})_{j,k=1}^{+\infty}$,
where $\gamma_{j,k}$ are complex numbers such that
\begin{equation}
\label{eq:log1}
\log\frac{\psi(z)-\psi(w)}{z-w}=\sum_{j,k=1}^{+\infty}\gamma_{j,k}z^{-j}w^{-k}
\end{equation}
holds for $z,w\in\D_{\e}$, $z\neq w$. 
Here, it is implicit that we use an appropriate choice of the logarithm
such that the left-hand side defines a holomorphic function on the
external bidisk $\D_{\e}^2$ with value $\log 1=0$ for $z=w=\infty$. 
The classical \emph{Grunsky inequality} \cite{Grunsky} assert that
the matrix $\Grunsky_{\psi}$ acts contractively on $\ell^2(\mathbb{Z}_{>0})$,
that is,
\begin{equation}\label{eq:grunsky}
\left\lvert \sum_{j,k= 1}^{+\infty} \sqrt{jk}\gamma_{j,k}\alpha_j\beta_k
\right\lvert^2\leq \sum_{j= 1}^{+\infty} |\alpha_j|^2\sum_{k=1}^{+\infty}|\beta_k|^2
\end{equation}
for all sequences $\{\alpha_j\}_{j=1}^{+\infty},
\{\beta_j\}_{j=1}^{+\infty}$ from $\ell^2(\mathbb{Z}_{>0})$. 
By differentiation of both sides of the relation \eqref{eq:log1}, we see that
\begin{equation}\label{eq:log2}
\frac{\psi^{\prime}(z)\psi^{\prime}(w)}{(\psi(z)-\psi(w))^2}
-\frac{1}{(z-w)^2}=
\partial_z\partial_w \log\frac{\psi(z)-\psi(w)}{z-w}=
\sum_{j,k=1}^{+\infty}jk\gamma_{j,k}z^{-j-1}w^{-k-1},
\end{equation}
where $z,w\in \D_{\e}$.
If we have two Laurent series 
\begin{equation}\label{eq:log3}
g(z)=\sum_{n=2}^{+\infty}a_nz^{-n},\quad h(z)=\sum_{n=2}^{+\infty}b_nz^{-n},
\end{equation}
which are finite in the sense that only finitely many coefficients are nonzero,
we calculate that
\begin{multline}\label{eq:log4}
\int_{\D_{\e}} \int_{\D_{\e}}
\left(   \frac{\psi^{\prime}(z)\psi^{\prime}(w)}{(\psi(z)-\psi(w))^2}
-\frac{1}{(z-w)^2}  \right)   \overline{g(z)h(w)} \dA (z) \dA (w)
\\
= \int_{\D_{\e}} \int_{\D_{\e}}\sum_{j,k=1}^{+\infty}jk
\gamma_{j,k}z^{-j-1}w^{-k-1}\sum_{m,n=2}^{+\infty}
\bar{a}_m\bar{b}_n\bar{z}^{-m}\bar{w}^{-n} \dA (z) \dA (w)
\\
=\sum_{j,k=1}^{+\infty}\sum_{m,n=2}^{+\infty}jk\gamma_{j,k}\bar{a}_m\bar{b}_n
\int_{\D_{\e}} z^{-j-1} \bar{z}^{-m}
 \dA (z)
\int_{\D_{\e}} w^{-k-1} \bar{w}^{-n} \dA (w)
\\
=\sum_{j,k=1}^{+\infty} \gamma_{j,k}\bar{a}_{j+1}\bar{b}_{k+1}.
\end{multline}
If in the Grunsky inequality \eqref{eq:grunsky} we agree that
\[
\alpha_n=\frac{\bar{a}_{n+1}}{\sqrt{n}}, \quad \beta_n
=\frac{\bar{b}_{n+1}}{\sqrt{n}},
\]
for all $n=1,2,3,\dots$, we find that the Grunsky inequality expresses that 
\begin{multline}\label{eq:log5}
\left\lvert \int_{\D_{\e}} \int_{\D_{\e}}
\left(   \frac{\psi^{\prime}(z)\psi^{\prime}(w)}{(\psi(z)-\psi(w))^2}
-\frac{1}{(z-w)^2}  \right)   \overline{g(z)h(w)} \dA (z) \dA (w)
\right\rvert^2
\\
\leq\int_{\D_{\e}} \int_{\D_{\e}} |g(z)h(w)|^2 \dA (z) \dA (w).
\end{multline}
In analogy with the Grunsky operator acting on coefficients, we consider
the integral Grunsky operator $\Gamma_{\psi}$ given by
\[
\Gamma_{\psi}f(z)= \int_{\D_{\e}}
\left(   \frac{\psi^{\prime}(z)\psi^{\prime}(w)}{(\psi(z)-\psi(w))^2}
-\frac{1}{(z-w)^2}  \right)   f(w) \dA (w),
\]
and observe that in view of \eqref{eq:log5}, $\Gamma_{\psi}$ acts
contractively $\mathrm{conj}\,A^2(\D_{\e})\to A^2(\D_{\e})$.
Moreover, we have the  inequality
\[
\int_{\D_{\e}}  \left\lvert
\frac{\psi^{\prime}(z)\psi^{\prime}(w)}{(\psi(z)-\psi(w))^2}
-\frac{1}{(z-w)^2}  \right\rvert^2\dA (w)\leq \frac{1}{(|z|^2-1)^2},
\qquad z\in \D_{\e},
\]
and hence the function
\[
w\mapsto \frac{\psi^{\prime}(z)\psi^{\prime}(w)}{(\psi(z)-\psi(w))^2}
-\frac{1}{(z-w)^2}
\]
is in $A^2(\D_{\e})$. It now follows that $\Gamma_{\psi}f=0$ for all
$f\in L^2(\D_{\e})\ominus \mathrm{conj}\,A^2(\D_{\e})$, and hence
$\Gamma_{\psi}:$ $L^2(\D_{\e})\to A^2(\D_{\e})$ defines a contraction.

For the inverse function $\gamma:\D\raro \D_{\e}$, defined by
$\gamma(z)=1/z$, let
\[
U_{\gamma}f(z):=f(\gamma(z))\gamma^{\prime}(z),\quad
\bar{U}_{\gamma}f(z):=f(\gamma(z))\overline{\gamma^{\prime}(z)}
\]
be the associated unitary transformations from $L^2(\D)$ onto $L^2(\D_{\e})$.
We define the Grunsky operator $\Gamma_{\phi}$ on $L^2(\D)$ based on the
Grunsky operator $\Gamma_{\psi}$ on $L^2(\D_{\e})$ via the relationship
\[
\Gamma_{\phi}:=U_{\gamma}\Gamma_{\psi}\bar{U}_{\gamma},
\]
where $\psi(z)=1/\vp(1/z)$.
When we simplify the expressions involved, we find that
\[
\Gamma_{\phi}f(z)=
\int_{\D} \left(
\frac{\phi^{\prime}(z)\phi^{\prime}(w)}{(\phi(z)-\phi(w))^2}-\frac{1}{(z-w)^2}
\right) f(w) \dA (w),\quad f\in L^2(\D),
\]
which acts contractively on $L^2(\D)$ given that the Grunsky operator
$\Gamma_\psi$ is a contraction on $L^2(\D_\e)$.

\subsection{Dirichlet symbols for Grunsky operators}

The Dirichlet symbol associated with the Grunsky operator $\Gamma_{\vp}$
is given by
\[
\mathscr{Q}[\Gamma_{\vp}](z,w)=zw\mathscr{P}[\Gamma_{\vp}](z,w)
=\langle\Gamma_{\vp}\bar s_w,s_z\rangle_\D  ,\qquad
(z,w)\in \D^2.
\]
We can in fact calculate that
\begin{equation}
\mathscr{Q}[\Gamma_{\vp}](z,w)=\langle \Gamma_{\vp}\bar s_w,s_z\rangle_\D
=\log\frac{zw(\vp(z)-\vp(w))}{(z-w)\vp(z)\vp(w)},\qquad (z,w)\in\D^2.
\label{eq:DirsymbGrunsky-0}
\end{equation}

Given that the Grunsky operators are more naturally defined in the context of
the exterior disk $\D_\e$, it makes sense to extend the notion of Dirichlet
symbols to the exterior disk for general bounded operators on $L^2(\D_\e)$ as
well.
So, suppose that $\Tope$ is a bounded linear operator on $L^2(\D_{\e})$.
Then the Dirichlet operator symbol $\mathscr{Q}_\e[\Tope]$ associated with
$\Tope$ on the exterior disk $\D_\e$ is given by much the same expression,
\begin{equation}
\mathscr{Q}_\e[\Tope](z,w):=\la \Tope(\overline{s^\e_w}),
s_z^\e \ra_{\D_{\e}},\qquad
(z,w)\in\D_{e}^2,
\label{eq:Qe-def-1}
\end{equation}
where
\[
s^\e_w(\xi):=\frac{1}{\xi(1-\xi\bar{w})},\qquad \xi,w\in\D_{\e}.
\]
Then $s^\e_w\in A^2(\D_e)$, and it depends conjugate-holomorphically on $w$,
with limit $s^\e_w\to0$ as $|w|\to+\infty$. As a consequence,
$\mathscr{Q}_\e[\Tope]$ is holomorphic in the exterior bidisk $\D_\e^2$,
and vanishes at infinity:
\[
\mathscr{Q}_\e[\Tope](z,\infty)=\mathscr{Q}_\e[\Tope](\infty,w)=0,\qquad
z,w\in\D_\e.  
\]
Moreover, $s^\e_w(\xi)$ has derivative
\[
\bar\partial_w s^\e_w(\xi)=\frac{1}{(1-\bar w\xi)^2}=:\Berg_w^\e(\xi),
\]
the Bergman kernel for $A^2(\D_e)$. This means that if $\Tope=\Gamma_\psi$,
the Grunsky operator associated with the mapping $\psi$ in the class $\Sigma$
of normalized conformal mappings on $\D_\e$, we may calculate that
\begin{equation}
\mathscr{Q}_\e[\Gamma_\psi](z,w)=\log\frac{\psi(z)-\psi(w)}{z-w},\qquad
z,w\in\D_\e.  
\label{eq:DirsymbolGrunsky-1}
\end{equation}
with the understanding that the branch of the logarithm at infinity is such
that the right-hand side vanishes when $z=\infty$ or $w=\infty$. In
particular, then, the diagonal restriction equals
\[
\oslash\mathscr{Q}_\e[\Gamma_{\psi}](z)=\mathscr{Q}_\e[\Gamma_\psi](z,w)=
\log\psi'(z), \qquad z\in \D_{\e}.
\]

\subsection{Asymptotic variance}
\label{ss:asymp}
Let $f:\D\raro \mathbb{C}$ be a holomorphic function. Following McMullen
\cite{mcmullen}, we define its \emph{asymptotic variance} by the formula
\[
\sigma(f)^2:=\limsup_{r\raro 1^{-}}
\frac{1}{\log\frac{1}{1-r^2}}\int_{\mathbb{T}}|f(r\zeta)|^2\diff s(\zeta).
\]
If the boundary behavior of $f$ is similar to that of a Bloch function,
especially in cases which can be described as "dynamical", $\sigma(f)^2$
captures very well the typical boundary growth of the given function $f$.
Moreover, from a probabilistic point of view, it is possible to think of
the evolution of the function  $r\raro f(r\zeta)$ as a sort of Brownian
motion in time $\log \frac{1+r}{1-r}= \log \frac{1}{1-r^2}+\Ordo(1)$ as
$r\to1^-$.
We note that since we do not assume that the analytic function $f$ is in
the Bloch space $\clb(\D)$, it is not for certain that the asymptotic
variance is finite. However, it turns out that the asymptotic variance
is always finite for the diagonal part of a Dirichlet operator symbol
$f(z)=\oslash\mathscr{Q}[\Tope](z)$ (Theorem 1.5.3 in \cite{GAF}), although
these are not always in the Bloch space.
This fact is rather subtle, as holomorphic functions with the growth control
$|f(z)|=\Ordo(\log \frac{1}{1-|z|^2})$ need not have a finite asymptotic
variance, as follows from the work of Abakumov and Doubtsov \cite{A_D}.
Such growth control is of Korenblum type (see \cite{B_L}), and it is
fulfilled automatically for Bloch functions as well as for the diagonal
parts of Dirichlet symbols of bounded operators.

\begin{rem}
If we write
\[
\sigma^{\langle r\rangle}(f)^2:=
\frac{1}{\log\frac{1}{1-r^2}}\int_{\mathbb{T}}|f(r\zeta)|^2\diff s(\zeta),
\]
so that the asymptotic variance $\sigma(f)^2$ is the upper limit of
this quantity as $r\to1^-$, we can use Taylor expansion,
\[
f(z)=\sum_{n=0}^{+\infty}\hat f(n)z^n,\qquad z\in\D,
\]
together with the Plancherel identity to see that
\begin{equation*}
\sigma^{\langle r\rangle}(f)^2
=\frac{1}{\log\frac{1}{1-r^2}}\bigg(|\hat f(0)|^2+
\sum_{n=1}^{+\infty}n|\hat f(n)|^2\,\frac{r^{2n}}{n}\bigg),
\end{equation*}
and hence $\sigma^{\langle r\rangle}(f)^2$ equals a weighted average of $n\mapsto
n|\hat f(n)|^2$ for each $r$ if $f(0)=\hat f(0)=0$, and approximately so for
$r$ close to $1$ if not. This suggests interpreting $\sigma(f)^2$ as an asymptotic density
of coefficient energy, understood through logarithmically weighted averages
of $n|\hat f(n)|^2$. The upper limit allows for fluctuations, so this
density interpretation does not require an ordinary limit to exist. 
\end{rem}

The concept of asymptotic variance extends naturally to the setting of the
exterior disk $\D_\e$ as well. If
$g:\D_{\e}\raro\mathbb{C}_{\infty}$ is a holomorphic function, then
its (exterior) asymptotic variance is given by
\[
\sigma_\e(g)^2:=\limsup_{R\raro 1^{+}}\frac{1}{\log\frac{1}{R^2-1}}
\int_{\mathbb{T}}|g(R\zeta)|^2\diff s(\zeta).
\]
We now return to the usual context of asymptotic variance on the unit disk
$\D$. Let $0<r<1$ and suppose $f$ is a holomorphic function on $\D$. Let
$f_r(z):=f(rz)$, $z\in \D$. The Littlewood-Paley identity
implies that
\begin{equation*}
\int_{\mathbb{T}} |f(r\zeta)|^2 \diff s(\zeta)=|f(0)|^2
+r^2\int_{\D}|f^{\prime}(rz)|^2 \log \frac{1}{|z|^2} \dA (z).
\end{equation*}
By Taylor expansion, we have that
$\log\frac{1}{t}=1-t+\Ordo(1-t)^2$ as $t\to1^-$, and as we apply it with
$t=|z|^2$, it follows that
\begin{equation}
\limsup_{r\raro 1^{-}}\frac{1}{\log\frac{1}{1-r^2}}
\int_{\mathbb{T}} |f(r\zeta)|^2 \diff s(\zeta)=
\limsup_{r\raro 1^{-}}\frac{1}{\log\frac{1}{1-r^2}}
\int_{\D}|f^{\prime}(rz)|^2 (1-|z|^2) \dA (z).
\label{eq:sigma-1.1}
\end{equation}
We can take the Littlewood-Paley identity one step further, and obtain that
\begin{multline*}
\int_{\mathbb{T}} |f(r\zeta)|^2 \diff s(\zeta)=|f(0)|^2
+r^2|f^{\prime}(0)|^2 \\+r^4\int_{\D}|f^{\prime\prime}(rz)|^2
\left\{\left(1+|z|^2\right)\log
\frac{1}{|z|^2}-2\left(1-|z|^2\right)\right\} \dA (z).
\end{multline*}
This identity comes from Green's formula and the potential theory of the
biharmonic operator, which was important, e.g., in the context of Hele-Shaw
flow on curved surfaces \cite{HedShim1}, \cite{HedPer}, \cite{HedOlof}.
Next, Taylor expansion gives that
\begin{equation*}
\log\frac{1}{t}=-\log t=1-t+\frac{(1-t)^2}{2}+\frac{(1-t)^3}{3}+
\Ordo(1-t)^4,
\end{equation*}
as $t\to1^-$, so that
\begin{multline*}
(1+t)\log\frac{1}{t}-2(1-t)=
(1+t)\bigg(1-t+\frac{(1-t)^2}{2}+\frac{(1-t)^3}{3}+
\Ordo(1-t)^4\bigg)-2(1-t)
\\
=\frac16(1-t)^3+\Ordo(1-t)^4,
\end{multline*}
again as $t\to1^-$. By setting $t=|z|^2$, it now follows that
\begin{multline}
\limsup_{r\raro 1^{-}}\frac{1}{\log\frac{1}{1-r^2}}\int_{\mathbb{T}}
|f(r\zeta)|^2 \diff s(\zeta)
\\
= \limsup_{r\raro 1^{-}}
\frac{1}{\log\frac{1}{1-r^2}}\frac{1}{6}\int_{\D}
|f^{\prime\prime}(rz)|^2 (1-|z|^2)^3 \dA (z).
\label{eq:sigma-1.2}
\end{multline}
Next, in view of the identities \eqref{eq:sigma-1.1} and
\eqref{eq:sigma-1.2}, it becomes natural to introduce, for a holomorphic
function $f:\,\D\to\C$, the higher order asymptotic variances
\[
\sigma_1(f)^2:=\limsup_{r\raro 1^{-}}
\frac{1}{\log\frac{1}{1-r^2}}\int_{\D}
|f(rz)|^2 (1-|z|^2) \dA (z).
\]
and
\[
\sigma_2(f)^2:=\limsup_{r\raro 1^{-}}
\frac{1}{\log\frac{1}{1-r^2}}\,\int_{\D}
|f(rz)|^2 (1-|z|^2)^3 \dA (z).
\]

\begin{prop}
Given a holomorphic function $f:\,\D\to\C$, we have the identity of
asymptotic variances
\[
\sigma(f)^2=\sigma_1(f')^2=\frac16\sigma_2(f'')^2.
\]
\label{prop:asymp1}
\end{prop}

\begin{proof}
This is the assertion of \eqref{eq:sigma-1.1} and \eqref{eq:sigma-1.2}. 
\end{proof}

It is perhaps surprising that we may change the formula for these variances 
and the asymptotics remains the same. 

\begin{prop}
\label{prop:compare-1}  
Given a holomorphic function $f:\D\to\C$, and $\theta$
with $0\le\theta<2$, we have that
\[
\sigma_k(f)^2=\limsup_{r\raro 1^{-}}
\frac{1}{\log\frac{1}{1-r^2}}\int_{\D}
|f(rz)|^2 (1-r^{\theta} |z|^2)^{2k-1} \dA (z),\qquad k=1,2.
\]
\end{prop}

\begin{proof}
For $k=1,2$ and $0\le r<1$, let $I_k(r)$ and $J_k(r)$ denote the integrals
\[
I_k(r):=\int_\D|f(rz)|^2(1-|z|^2)^{2k-1}\dA(z), \quad
J_k(r):=\int_\D|f(rz)|^2(1-r^\theta|z|^2)^{2k-1}\dA(z).   
\]
Moreover, since $1-|z|^2\le 1-r^\theta|z|^2$ holds on $\D$, we have the
inequality $I_k(r)\le J_k(r)$. On the other hand, the simple change-of-variables
$\zeta=r^{\theta/2}z$ in the integral defining $J_k(r)$ gives that
\begin{multline}
J_k(r)=\int_\D|f(rz)|^2(1-r^\theta|z|^2)^{2k+1}\dA(z)=
r^{-\epsilon}\int_{\D(0,r^{\theta/2})}
|f(r^{1-\frac12\theta}\zeta)|^2(1-|z|^2)^{2k+1}\dA(z)
\\
\le r^{-\theta}\int_{\D}|f(r^{1-\frac12\theta}\zeta)|^2(1-|z|^2)^{2k+1}\dA(z)=
r^{-\theta}I_{k}(r^{1-\frac12\theta}).   
\end{multline}
The assertion now follows from the asymptotics
\[
\log\frac{1}{1-r^2}=\log\frac{1}{1-r^{2-\theta}}+\Ordo(1)  
\]
as $r\to1^-$. 
\end{proof}


\subsection{ The nonlinearity and the Schwarz\-ian derivative}
If $\vp:\D\raro \mathbb{C}$ is a locally univalent function
(that is, $\vp^{\prime}(z)\neq 0$ holds for all $z\in \D$), then its
\emph{nonlinearity} (or pre-Schwarz\-ian derivative) is given by
\[
\mathrm{N}(\vp)(z)=
h_\vp'(z)=\frac{\vp^{\prime\prime}(z)}{\vp^{\prime}(z)}.
\]
Here, we write $h_\vp=\log\vp'$ as before. The nonlinearity is a
fundamental tool in geometric function theory.
We note the transformation rule
\[
\mathrm{N}(L\circ \vp)=\mathrm{N}(\vp)
\]
for all non-constant affine complex-linear maps $L$. More generally, 
if $f$ is locally univalent on the image domain $\vp(\D)$, then the chain rule
for the nonlinearity reads 
\[
\mathrm{N}(f\circ \vp)=\vp^{\prime} \mathrm{N}(f)\circ \vp +\mathrm{N}(\vp).
\]
Becker's univalency criterion asserts that if $\vp$ is locally univalent
in the unit disk $\D$, and if 
\[
|\mathrm{N}(\vp)(z)|\leq \frac{1}{1-|z|^2},\qquad z\in \D,
\]
then $\vp$ is actually univalent in $\D$. On the other hand, it is well-known
that every univalent analytic function $\vp$ in the unit disk satisfies 
\[
|\mathrm{N}(\vp)(z)|\leq \frac{6}{1-|z|^2},\quad z\in \D.
\]
There is a natural second order nonlinear derivative of $\vp$, called the
\emph{Schwarz\-ian derivative} of $\vp$, and it is given by the expression
\[
\Sop(\vp)(z):=\left( \frac{\vp^{\prime\prime}(z)}{\vp^{\prime}(z)}
\right)^{\prime}-\frac{1}{2}\left( \frac{\vp^{\prime\prime}(z)}{\vp^{\prime}(z)}
\right)^2.
\]
Analogously to how the nonlinearity measures the local deviation from an
complex-linear affine transformation, the Schwarz\-ian derivative $\Sop(\vp)$
measures the local deviation $\vp$ from being a M\"{o}bius map. 
The Schwarz\-ian derivative satisfies the cocycle identity
\begin{equation}
\Sop(f\circ \vp)=(\vp^{\prime})^2\,(\Sop(f))\circ \vp +\Sop(\vp).
\label{eq:cocycle-1}
\end{equation}
It can be checked that $\Sop (m)=0$ holds if and only if $m$ is a M\"obius
transformation $m(z)=(az+b)/(cz+d)$, where $ad-bc\neq 0$.
From the above cocycle identity we see that
\[
\Sop(m\circ\vp)=\Sop(\vp)
\]
holds for all M\"obius transformations $m$. The analogue of Becker's univalency
criterion for the nonlinearity is known as Nehari's theorem \cite{Nehari}
for the Schwarz\-ian derivative, and it asserts that if
\[
|\Sop(\vp)(z)|\leq \frac{2}{(1-|z|^2)^2},\qquad z\in \D,
\]
then $\vp$ must be univalent in $\D$. On the other hand, the optimal
pointwise estimate for the Schwarz\-ian derivative of a univalent function
$\vp\in\mathscr{S}$ is  
\[
|\Sop (f)(z)|\leq \frac{6}{(1-|z|^2)^2},\qquad z\in\D,
\]
so, clearly, there is a gap in our understanding of the growth of the
Schwarz\-ian derivative.

Next, let us relate the nonlinearity and the Schwarz\-ian derivative to
the Dirichlet symbol of the Grunsky operator $\Gamma_\vp$.
We first observe that by \eqref{eq:DirsymbGrunsky-0}, we have that
\begin{equation}
\oslash\mathscr{Q}[\Gamma_\vp](z)=\mathscr{Q}[\Gamma_\vp](z,z)
=z^2\langle\Gamma_\vp(\bar s_z),s_z\rangle_\D=\log\frac{z^2\vp'(z)}{\vp(z)^2},
\end{equation}
and, moreover, that it follows that
\begin{multline}
\partial_z\oslash\mathscr{Q}[\Gamma_\vp](z)=
2\oslash\partial_z\mathscr{Q}[\Gamma_\vp](z)
=2z\langle\Gamma_\vp(\bar s_z),\Berg_z\rangle_\D=
\frac{\vp''(z)}{\vp'(z)}-2\frac{\vp'(z)}{\vp(z)}+\frac{2}{z}
\\
=\mathrm{N}(\vp)(z)-2\frac{\vp'(z)}{\vp(z)}+\frac{2}{z},
\label{eq:NL-1}
\end{multline}
a small correction added to the nonlinearity, 
and that
\begin{equation}
\oslash(\partial_z\partial_w\mathscr{Q}[\Gamma_\vp])(z)
=\langle \Gamma_\vp (\bar\Berg_z),\Berg_z\rangle_\D
=\frac16\mathrm{S}(\vp)(z).
\label{eq:NL-2}
\end{equation}

\section{Grunsky operators and the nonlinear wave equation}
\label{sec:NLWE}

\subsection{Dirichlet operator symbols for Grunsky operators}

Our goal in this section is to characterize Dirichlet symbols for
Grunsky operators by examining their behavior on the diagonal.
In \cite{GAF}, such Dirichlet symbols were characterized by a nonlinear
wave equation.
We localize the nonlinear wave equation in terms of jets along the diagonal.
This gives a sequence of recurrence relations for successive partial
derivatives of the Dirichlet symbol restricted to the diagonal. While the
nonlinear wave equation itself is concerned with a holomorphic function
of two complex variables, the recurrence relations involve only functions
of a single complex variable. 

We recall that the Dirichlet symbol associated with the Grunsky
operator $\Gamma_{\vp}$, where  $\vp\in \mathcal{S}$, is given by
(see \eqref{eq:DirsymbGrunsky-0})
\[
Q(z,w)=\mathscr{Q}[\Gamma_{\vp}](z,w)=
\log\frac{zw(\vp(z)-\vp(w))}{\vp(z)\vp(w)(z-w)},\qquad (z,w)\in \D^2,
\]
with diagonal restriction
\[
\oslash Q(z)=\oslash \mathscr{Q}[\Gamma_{\vp}](z)=
\log \frac{z^2\vp^{\prime}(z)}{(\vp(z))^2},\qquad z\in \D.
\]
Similarly, if $\psi\in \Sigma$, then the exterior Dirichlet symbol is
\[
Q_{\e}(z,w)=\mathscr{Q}_\e[\Gamma_{\psi}](z,w)
=\log \frac{\psi(z)-\psi(w)}{z-w},\qquad (z,w)\in \D_{\e},
\]
and 
\[
\oslash Q_{\e}(z)=\oslash \mathscr{Q}_\e[\Gamma_{\psi}](z)
=\log \psi^{\prime}(z),\qquad z\in \D_{\e}.
\]
Here, as previously, $(\oslash f)(z)=f(z,z)$ denotes the restriction of
$f(z,w)$ to the diagonal. Note that the two Dirichlet symbols are essentially
the same, as if we choose the change-of-variables $z'=1/z$, $w'=1/w$, and
\[
\psi(z)=\frac{1}{\vp(z')},\quad \vp(w)=\frac{1}{\psi(w')},
\]
then
\[
\frac{\psi(z)-\psi(w)}{z-w}=
\frac{z'w'(\vp(z')-\vp(w'))}{(z'-w')\vp(z')\vp(w')}
\]
and hence
\[
Q_\e(z,w)=Q(z',w').
\]
It should be mentioned that we need to presuppose that $\psi$ does not
take on the value $0$ for the above identification of the classes
$\mathscr{S}$ and $\Sigma$ to work. However, by replacing
$\psi$ by $\tilde\psi:=\psi-c$ where $c\in\C\setminus\psi(\D_\e)$
we can make sure that this condition is fulfilled. Moreover, the
Grunsky operator $\Gamma_\psi$ remains unperturbed by this change. 
The analogues of the relations \eqref{eq:NL-1} and \eqref{eq:NL-2} in
the context of the exterior disk $\D_\e$ are
\begin{equation}
\partial_z(\oslash Q_\e)(z)=\partial_z\oslash\mathscr{Q}_\e[\Gamma_\psi](z)=
2\oslash\partial_z\mathscr{Q}_\e[\Gamma_\psi](z)
=\frac{\psi''(z)}{\psi'(z)}
=\mathrm{N}(\psi)(z),\qquad z\in\D_\e,
\label{eq:NL-1'}
\end{equation}
and 
\begin{equation}
\oslash(\partial_z\partial_w Q_\e)(z)=
\oslash(\partial_z\partial_w\mathscr{Q}_\e[\Gamma_\psi])(z)
=\frac16\mathrm{S}(\psi)(z),\qquad z\in\D_\e,
\label{eq:NL-2'}
\end{equation}
where the nonlinearity $\mathrm{N}[\psi]$ and the Schwarz\-ian derivative
$\Sop[\psi]$ are given by the same expressions but in the context of $\D_\e$.
We now take a look at the information content of the diagonal restrictions
$\oslash Q_\e$, $\oslash \partial_z Q_\e$, and $\oslash \partial_z\partial_w
Q_\e$. 

\begin{prop}
Suppose $\psi_1,\psi_2\in\Sigma$, and, for $j=1,2$, put
$Q_{\e,j}(z,w):=\mathscr{Q}_\e[\Gamma_{\psi_j}](z,w)$. Then each of the
following relations are equivalent:

\noindent{$(a)$} We have that $\psi_2(z)=\psi_1(z)+C$ holds for some
constant $C$.

\noindent{$(b)$} The relation $Q_{\e,1}=Q_{\e,2}$ holds on $\D_\e^2$.

\noindent{$(c)$} The relation $\oslash Q_{\e,1}=\oslash Q_{\e,2}$
holds on $\D_\e$.

\noindent{$(d)$} The relation $\oslash(\partial_z Q_{\e,1})=
\oslash(\partial_z Q_{\e,2})$ holds on $\D_\e$. 

\noindent{$(e)$} The relation $\oslash(\partial_z\partial_w Q_{\e,1})=
\oslash(\partial_z\partial_w Q_{\e,2})$ holds on $\D_\e$.
\label{prop:5equiv}
\end{prop}

\begin{proof}
Under condition $(a)$, we have that
\[
\frac{\psi_1(z)-\psi_1(w)}{z-w}=\frac{\psi_2(z)-\psi_2(w)}{z-w},
\]
so that, consequently,
\[
Q_{\e,1}(z,w)=\log\frac{\psi_1(z)-\psi_1(w)}{z-w}
=\log\frac{\psi_2(z)-\psi_2(w)}{z-w}=Q_{\e,2}(z,w),
\]
and $(b)$ follows. Clearly, $(b)$ implies all the other conditions $(c)$,
$(d)$, and $(e)$. It now remains to demonstrate that each of the conditions
$(b)$, $(c)$, $(d)$, and $(e)$, implies $(a)$. Moreover, since $(b)$ implies
$(c)$, $(d)$, and $(e)$, it is enough to show that each of the conditions
$(c)$, $(d)$, and $(e)$ implies property $(a)$. We first consider $(c)$, which
in concrete terms says that
\[
\log\psi_1'=\log\psi_2'.
\]
But then we exponentiate both sides, and find that $\psi_1'=\psi_2'$, which
is equivalent to having $\psi_2=\psi_1+C$ for some constant $C$, i.e.,
assertion $(a)$. We then turn to condition $(d)$, which amounts to having
\[
\frac{\psi_1''}{\psi_1'}=\frac{\psi_2''}{\psi_2'}.
\]
The two sides being logarithmic derivatives, the condition is equivalent to
having $\psi_2'=A\psi_1'$ for some constant $A\ne0$. By the normalization
condition for the class $\Sigma$, $\psi_1'(\infty)=\psi_2'(\infty)=1$,
so the only possible constant $A$ is $A=1$. Hence we get $\psi_1'=\psi_2'$
with solution $\psi_2=\psi_1+C$ for a constant $C$, i.e., condition $(a)$.
Finally, we consider condition $(e)$. In concrete terms, it asserts the
equality of Schwarz\-ian derivatives $\Sop(\psi_1)=\Sop(\psi_2)$.
We composition factor $\psi_2=f\circ\psi_1$, where $f:=\psi_2\circ\psi_1^{-1}$.
and apply the cocycle law \eqref{eq:cocycle-1} for the Schwarz\-ian derivative,
which holds on the exterior disk as well:
\begin{equation}
\Sop(\psi_2)=\Sop(f\circ \psi_1)=
(\psi_1^{\prime})^2\,(\Sop(f))\circ \psi_1 +\Sop(\psi_1).
\label{eq:cocycle-2}
\end{equation}
Given the equality $\Sop(\psi_1)=\Sop(\psi_2)$, we are left with
\[
(\psi_1^{\prime})^2\,(\Sop(f))\circ \psi_1=0,
\]
which is the same as having
\[
\Sop(f)=0\quad\text{on}\,\,\,\psi_1(\D_\e).
\]
But then $f$ must actually equal a M\"obius mapping $f(z)=(az+d)/(cz+d)$ with
$ad-bc\ne0$. On the other hand, $f=\psi_2\circ\psi_1^{-1}$ must fix the point
at infinity with $f'(\infty)=1$ given the normalizations imposed on
$\psi_1,\psi_2\in\Sigma$, and the only alternative is the affine mapping
$f(z)=z+C$ for some constant $C$. So, $\psi_2=f\circ\psi_1=\psi_1+C$ must hold,
i.e., condition $(a)$.
The proof is complete.
\end{proof}


\subsection{The nonlinear wave equation}

We base our further analysis on the following theorem, which was obtained
in \cite{GAF}.

\begin{thm}\label{Thm 1}
Let $Q:\D^2\raro\mathbb{C}$ be a holomorphic function. Then there exists
$\vp\in \cls$ such that
\[
Q=\mathscr{Q}[\Gamma_{\vp}]
\]
if and only if

\noindent{$(a)$} $Q(z,0)=0$ and $Q(0,w)=0$ for all $z,w\in \D$, and

\noindent{$(b)$}  $Q$ solves the nonlinear wave equation
\begin{equation*}\label{Eq1}
\partial_z\partial_wQ+(\partial_zQ)(\partial_wQ)
-\frac{z^2\partial_zQ-w^2\partial_wQ}{zw(z-w)}=0.
\end{equation*}
\end{thm} 

It is convenient to transfer the content of the theorem to the setting of
the exterior disk $\D_\e$. The shape of the nonlinear wave equation gets
a little simpler as a result. As it is essentially an equivalent formulation
of the above Theorem \ref{Thm 1}, no separate proof is supplied. 


\begin{thm}\label{Thm_A}
Let $Q_{\e}:\mathbb{D}^2_e\to \mathbb{C}_{\infty}$ be a holomorphic
function. Then there exists a function $\psi$ in the class $\Sigma$ such that
\[
Q_{\e}=\mathscr{Q}[\Gamma_{\psi}]
\]
if and only if 

\noindent{$(a)$} $Q_{\e}(z,\infty)=0$ and $Q_{\e}(\infty,w)=0$ for all
$z,w\in \D_{\e}$, and

\noindent{$(b)$} $Q_{\e}$ solves the nonlinear wave equation
\begin{equation*}\label{Eq2}
\partial_z\partial_wQ_{\e}+(\partial_zQ_{\e})(\partial_wQ_{\e})
-\frac{\partial_zQ_{\e}-\partial_wQ_{\e}}{z-w}=0.
\end{equation*}
\end{thm}


\subsection{Diagonal analysis of the nonlinear wave equation}
\label{ss:diagonal-nlweq}

We fix a diagonal point $(\alpha,\alpha)\in \mathbb{D}_{\e}^2$, and let
$Q_\e:=\mathscr{Q}_\e[\Gamma_\psi]$ for a normalized mapping $\psi\in\Sigma$.
By Taylor's formula applied to the point $(\alpha,\alpha)$, we find that
\[
Q_\e(\alpha+\xi,\alpha+\eta)
=\sum_{j,k=0}^{+\infty}
\frac{\partial_{z}^j\partial_{w}^kQ_\e(\alpha,\alpha)}{j!k!}\xi^j\eta^k
=\sum_{j,k=0}^{+\infty}\frac{(\oslash\partial_{z}^j\partial_{w}^kQ_\e)
(\alpha)}{j!k!}\xi^j\eta^k,
\]
for all $(\xi,\eta)\in \D^2(0, |\alpha|-1)$.
By taking successive partial derivatives with respect to $z$ and $w$,
we find that for $a,b=0,1,2,\ldots$,
\[
(\partial_{z}^a\partial_w^bQ_\e)(\alpha+\xi,\alpha+\eta)=
\sum_{j,k=0}^{+\infty}\frac{\oslash(\partial_{z}^{j+a}
\partial_{w}^{k+b}Q_\e)(\alpha)}{j!k!}\xi^j\eta^k.
\]
As $Q_\e$ satisfies the nonlinear wave equation of Theorem
\ref{Thm_A}, that is,
\begin{equation*}
\partial_z\partial_wQ_{\e}+(\partial_zQ_{\e})(\partial_wQ_{\e})
-\frac{\partial_zQ_{\e}-\partial_wQ_{\e}}{z-w}=0,
\end{equation*}
what happens if we insert the point $(z,w)=(\alpha+\xi,\alpha+\eta)$, with
$\xi=0$? We arrive at the equation
\begin{multline}
\sum_{k=0}^{+\infty}\frac{\eta^k}{k!}\Big(\oslash(\partial_w^{k+1}Q_\e)(\alpha)-
\oslash(\partial_z\partial_w^k Q_\e)(\alpha)\Big)
\\
-\sum_{k=1}^{+\infty}\frac{\eta^k}{(k-1)!}
\bigg(\oslash(\partial_z\partial_w^k Q_\e)(\alpha)+\sum_{l=0}^{k-1}\binom{k-1}{l}
(\oslash\partial_z\partial_w^l Q_\e)(\alpha)
(\oslash\partial_w^{k-l}Q_\e)(\alpha)\bigg)=0
\end{multline}
with convergence for $\eta\in\D(0,|\alpha|-1)$. Upon identifying Taylor
coefficients, we find for $k=0$ that
\begin{equation}
\oslash(\partial_zQ_\e)(\alpha)=\oslash(\partial_wQ_\e)(\alpha),
\label{eq:condition_0}
\end{equation}
whereas for $k=1,2,3,\ldots$,
\begin{multline}
\oslash(\partial_w^{k+1}Q_\e)(\alpha)-
\oslash(\partial_z\partial_w^k Q_\e)(\alpha)
\\
-k
\bigg(\oslash(\partial_z\partial_w^k Q_\e)(\alpha)+\sum_{l=0}^{k-1}\binom{k-1}{l}
(\oslash\partial_z\partial_w^l Q_\e)(\alpha)
(\oslash\partial_w^{k-l}Q_\e)(\alpha)\bigg)=0,
\label{eq:condition_1}
\end{multline}
or, equivalently, if we suppress indication of the point $\alpha\in\D_\e$,
\begin{equation}
\oslash(\partial_w^{k+1}Q_\e)
-(k+1)\oslash(\partial_z\partial_w^k Q_\e)-k\sum_{l=0}^{k-1}\binom{k-1}{l}
(\oslash\partial_z\partial_w^l Q_\e)
(\oslash\partial_w^{k-l}Q_\e)=0.
\label{eq:condition_2}
\end{equation}
Let us first observe how to understand the relations \eqref{eq:condition_0} and
\eqref{eq:condition_2}, in light of the identity
\begin{equation}
\oslash(\partial_z^j\partial_w^k Q_\e)=\oslash(\partial_z^k\partial_w^j
Q_\e),
\label{eq:symmetry_1}
\end{equation}
which is a consequence of the symmetry relation
\begin{equation}
Q_\e(z,w)=Q_\e(w,z),
\end{equation}
which in turn we read off of \eqref{eq:DirsymbolGrunsky-1}. Note that
\eqref{eq:condition_0} is just the instance $j=1$, $k=0$ of
\eqref{eq:symmetry_1}. 
Using the symmetry \eqref{eq:symmetry_1}, we may rewrite
\eqref{eq:condition_2} more conveniently:
\begin{equation}
\oslash(\partial_z^{k+1}Q_\e)
-(k+1)\oslash(\partial_z^k\partial_w Q_\e)-k\sum_{l=0}^{k-1}\binom{k-1}{l}
(\oslash\partial_z^l\partial_w Q_\e)
(\oslash\partial_z^{k-l}Q_\e)=0.
\label{eq:condition_3}
\end{equation}
As we insert $k=1$ into \eqref{eq:condition_3}, using \eqref{eq:symmetry_1},
we find that
\begin{equation}
\label{eq:k=1}
\oslash(\partial_z^2Q_\e)-2\oslash(\partial_z\partial_wQ_\e)
-(\oslash\partial_zQ_\e)^2=0.
\end{equation}
If we return to the notational convention $\oslash(f)(z)=f(z,z)$ used
earlier, it follows from the bivariate chain rule that
\begin{equation}
\partial_z\oslash Q_\e=\oslash ((\partial_z+\partial_w) Q_\e)
=2\oslash (\partial_z Q_\e),
\label{eq:CR-1}
\end{equation}
and 
\begin{equation}
\partial_z^2\oslash Q_\e=\oslash ((\partial_z+\partial_w)^2 Q_\e)
=2\oslash (\partial_z^2 Q_\e)+2\oslash(\partial_z\partial_w Q_\e),
\label{eq:CR-2}
\end{equation}
so that the relation \eqref{eq:k=1} just asserts that
\begin{equation}
\label{eq:Schwarzian}
\oslash(\partial_z\partial_wQ_\e)-
\frac{1}{6}\partial^2_z\oslash Q_\e+\frac{1}{12}(\partial_z\oslash Q_\e)^2=0.
\end{equation}
In view of \eqref{eq:NL-1'} and \eqref{eq:NL-2'}, it follows that
\eqref{eq:Schwarzian} expresses the relationship
\begin{equation}
\Sop(\psi)-\partial_z\mathrm{N}(\psi)+\frac12\,(\mathrm{N}(\psi))^2=0,
\end{equation}
which may be thought of as a definition of the Schwarz\-ian derivative:
\begin{equation}
\Sop(\psi):=\partial_z\mathrm{N}(\psi)-\frac12\,(\mathrm{N}(\psi))^2
=\bigg(\frac{\psi''}{\psi'}\bigg)'-
\frac12\bigg(\frac{\psi''}{\psi'}\bigg)^2.
\end{equation}
Changing the parameter $k$ to $k+1$ in \eqref{eq:condition_3},
we find the following relationship.

\begin{cor}
{\rm{(Localized nonlinear wave equation)}}   
If $\psi\in\Sigma$ and $Q_\e:=\mathscr{Q}_\e[\Gamma_\psi]$ is the associated
Dirichlet operator symbol, then, for all $k=0,1,2,\ldots$, we have
\begin{equation}
\label{eq:condition_4}
\oslash (\partial_z^{k+2}Q_\e)-
(k+2)\oslash(\partial_z^{k+1}\partial_wQ_\e)
-(k+1)\sum_{l=0}^{k} \binom{k}{l}
(\oslash\partial_z^l\partial_wQ_\e)(\oslash \partial_z^{k+1-l}Q_\e)=0.
\end{equation}
\label{cor:local-1}
\end{cor}

\subsection{The Riccati equation considered by Aharonov}
\label{ss:Riccati-Aharonov}
If we let $R_\psi$ denote the function
\[
R_\psi(z,w):=\frac{\psi'(z)}{\psi(z)-\psi(w)}
\]
then
\[
\partial_z Q_\e(z,w)=\partial_z\log\frac{\psi(z)-\psi(w)}{z-w}=
\frac{\psi'(z)}{\psi(z)-\psi(w)}-\frac{1}{z-w}=R_\psi(z,w)-\frac{1}{z-w}.
\]
It was observed e.g. by Aharonov \cite{Dov} that this function solves
the Riccati equation
\[
\partial_z R_\psi(z,w)=\frac{\psi''(z)}{\psi'(z)}\,R_\psi(z,w)-(R_\psi(z,w))^2,
\]
and it is now possible to express this Riccati equation in terms of $Q_\e$
in place of $R_\psi$:
\[
\partial_z^2Q_\e(z,w)+(\partial_zQ_\e(z,w))^2
+\frac{2}{z-w}\big(\partial_zQ_\e-(\oslash\partial_zQ_\e)(z)\big)
-2(\oslash\partial_zQ_\e)(z) \partial_zQ_\e(z,w)=0.
\]
We can use this equation Riccati equation to characterize our operator
symbols, in a fashion similar to the nonlinear wave equation of Theorem
\ref{Thm_A}. Note, however, that the Riccati equation uses values at a
generic point $(z,w)$ as well as diagonal values, so the equation is
nonlocal and hence not a partial differential equation in any usual sense.


\begin{thm}
Let $q:\D_{e}^2\raro\mathbb{C}$ be a holomorphic function. Then 
\[
q(z,w)=Q_\e(z,w)=\mathscr{Q}_\e[\Gamma_{\psi}](z,w)
\]
holds for some univalent function $\psi\in\Sigma$ if and only if 

\noindent $(a)$ $q(\infty,w)=0$ for each $w\in\D_e$, and

\noindent{$(b)$} $q$ solves the Riccati equation
\[
\partial_z^2q(z,w)+(\partial_zq(z,w))^2+
\frac{2}{z-w}\big(\partial_zq(z,w)-(\oslash\partial_zq)(z)\big)
-2(\oslash\partial_zq)(z) \partial_zq(z,w)=0.
\] 
\label{thm:Riccati}
\end{thm}

\begin{proof}
The necessity follows from the Riccati identity established above and
from the normalization at infinity.

For the converse, let $q$ satisfy (a) and (b), and put
\[
a(z):=(\oslash\partial_z q)(z),\qquad
F(z,w):=(z-w)\e^{q(z,w)}.
\]
The function $F$ is holomorphic on $\D_\e^2$, vanishes precisely when
$z=w$, and satisfies
\[
\frac{\partial_z F(z,w)}{F(z,w)}
=\partial_zq(z,w)+\frac1{z-w}
\]
away from the diagonal. The Riccati equation now gives
\begin{equation}\label{eq:riccati-global-F}
\partial_z^2F(z,w)=2a(z)\partial_zF(z,w).
\end{equation}
Both sides are holomorphic across the diagonal, so the identity holds
throughout the bidisk.

Condition (a), together with holomorphy at infinity, implies for each
fixed $w$ that
\[
F(z,w)=z-w+\mathrm O(1),\qquad
\partial_zF(z,w)=1+\mathrm O(|z|^{-1})
\quad (z\to\infty).
\]
Moreover, $a(z)=\mathrm O(|z|^{-2})$ as $z\to\infty$.
For any fixed $w_1,w_2$, equation \eqref{eq:riccati-global-F}
shows that
\[
h(z):=\partial_zF(z,w_1)-\partial_zF(z,w_2)
\]
solves $h'=2ah$ and tends to zero at infinity. The integrating factor
for this first-order equation is holomorphic and nonvanishing near
infinity, with a finite nonzero limit there. It follows that $h=0$
near infinity, and then throughout $\D_\e$ by the identity theorem.
Thus $\partial_zF(z,w)$ is independent of $w$.

Fix $w_0\in\D_\e$ and put $\psi(z):=F(z,w_0)$. Since
$\partial_zF(z,w)=\psi'(z)$, the difference
$F(z,w)-\psi(z)$ is constant in $z$. Evaluating at $z=w$, where
$F(w,w)=0$, gives
\[
F(z,w)=\psi(z)-\psi(w).
\]
As $F(z,w)=(z-w)\e^{q(z,w)}$, this identity implies that $\psi$
is univalent on $\D_\e$. The asymptotic behavior of $F(z,w_0)$
shows that $\psi(z)=z+\mathrm O(1)$ at infinity, so $\psi\in\Sigma$.
Finally, $q$ and the normalized holomorphic logarithm of
$(\psi(z)-\psi(w))/(z-w)$ have identical exponentials.
Their difference is a constant multiple of $2\pi\mathrm{i}$, and condition
(a) forces that constant to be zero. Hence
\[
q(z,w)=\log\frac{\psi(z)-\psi(w)}{z-w}
=\mathscr{Q}_\e[\Gamma_\psi](z,w),
\]
as required.
\end{proof}

We now localize the Riccati equation of Theorem \ref{thm:Riccati}, using
the Taylor expansion
\begin{equation}
\label{eq:localRiccati-1}
\partial_z Q_\e(z,w)=\sum_{j=0}^{+\infty}
\frac{1}{j!}\big(\oslash{\partial_z\partial_w^j Q_\e}(z)\big)(w-z)^j,
\end{equation}
valid when $(z,w)\in\D_\e^2$ and $|w-z|<|z|-1$. We then have
\begin{multline}
\label{eq:localRiccati-2}
\partial_z^2 Q_\e(z,w)=\sum_{j=0}^{+\infty}
\frac{1}{j!}\big(\partial_z(\oslash{\partial_z\partial_w^j Q_\e})(z)\big)
(w-z)^j-\sum_{j=1}^{+\infty}
\frac{1}{(j-1)!}\big(\oslash{\partial_z\partial_w^j Q_\e}(z)\big)
(w-z)^{j-1},
\end{multline}
and
\begin{equation}
\label{eq:localRiccati-3}
\partial_z Q_\e(z,w)-\oslash(\partial_z Q_\e)(z)=\sum_{j=1}^{+\infty}
\frac{1}{j!}\oslash{\partial_z\partial_w^j Q_\e}(z)(w-z)^j,
\end{equation}
so that
\begin{equation}
\label{eq:localRiccati-4}
\frac{1}{z-w}\big(\partial_z Q_\e(z,w)-\oslash(\partial_z Q_\e)(z)\big)
=-\sum_{j=1}^{+\infty}
\frac{1}{j!}\oslash{\partial_z\partial_w^j Q_\e}(z)(w-z)^{j-1}.
\end{equation}
At the same time, simply squaring in \eqref{eq:localRiccati-1} gives that
\begin{equation}
\label{eq:localRiccati-5}
\big(\partial_z Q_\e(z,w)\big)^2=\sum_{j,j'=0}^{+\infty}
\frac{1}{j!(j')!}\big(\oslash{\partial_z\partial_w^j Q_\e}(z)\big)
\big(\oslash{\partial_z\partial_w^{j'} Q_\e}(z)\big)(w-z)^{j+j'}.
\end{equation}

\begin{cor}
{\rm(Localized Riccati equation)}
The diagonal localization of the Riccati equation of Theorem \ref{thm:Riccati}
splits into a sequence of identities indexed by $k=0,1,2,\ldots$.
The equation for $k=0$ reads
\[
\partial_z(\oslash\partial_z Q)(z)-3\oslash(\partial_z\partial_w Q_\e)(z)
-\big(\oslash\partial_z Q_\e(z)\big)^2 =0, 
\]
whereas for $k=1,2,3,\ldots$, it reads
\begin{multline}
\partial_z\oslash(\partial_z\partial_w^k Q_\e)(z)
-\frac{k+3}{k+1}\oslash(\partial_z\partial_w^{k+1}Q_\e)(z)
\\
+\sum_{l=1}^{k-1}
\binom{k}{l}\big(\oslash(\partial_z\partial_w^{k-l}Q_\e)(z)\big)
\big(\oslash(\partial_z\partial_w^{l}Q_\e)(z)\big)=0.
\end{multline}
\label{cor:Riccati}
\end{cor}

\begin{proof}
These equations follows from the Riccati equation of Theorem 
\ref{thm:Riccati}, if we insert the Taylor expansions
\eqref{eq:localRiccati-1}, \eqref{eq:localRiccati-2},
\eqref{eq:localRiccati-3}, \eqref{eq:localRiccati-4}, and
\eqref{eq:localRiccati-5}. For $j\ge1$, there is cancellation of terms
corresponding to the indices $k=0$ and $k=j-1$ in the sum.   
\end{proof}

\begin{rem}
\noindent{$(a)$}
Up to aesthetical differences in notation, the Riccati equation of Theorem
\ref{thm:Riccati} and its diagonal localization Corollary \ref{cor:Riccati}
are basically from Aharonov's paper \cite{Dov}.

\noindent{$(b)$} The content of Corollaries \ref{cor:local-1} and
\ref{cor:Riccati} are at a deeper level equivalent.
However, for our purposes, Corollary \ref{cor:Riccati} packages the information
in a more accessible form, as it involves the diagonal restrictions of
expressions that are mixed derivatives of $Q_\e$.
\end{rem}

\section{Diagonal analysis and diagonal series Schwarz\-ian derivatives}
\label{S3}

\subsection{The structure of Dirichlet symbols of Grunsky operators}


To facilitate our analysis of Dirichlet operator symbols, we shall have use
for the following combinatorial lemma.

\begin{lem}
For $n=1,2,3,\ldots$ and $m=0,\ldots,n$, we have that
\[
\sum_{k=0}^{\min\{m,n-m\}}\frac{(-1)^k}{n-k} \binom{n-k}{k}
\binom{n-2k}{m-k}=\frac{1}{n}\big(\delta_{m,0}+\delta_{m,n}\big), 
\]
where the delta is in the sense of Kronecker.
\label{lem:combin-1}
\end{lem}

\begin{proof}
We first observe that when $m=0$, the left-hand side expression equals
\[
\frac{1}{n} \binom{n}{0}
\binom{n}{0}=\frac{1}{n},
\]
which agrees with the right-hand side. Similarly, when $m=n$, the
left-hand side expression equals
\[
\frac{1}{n} \binom{n}{0}
\binom{n}{n}=\frac{1}{n},
\]
which again agrees with the right-hand side. It remains to deal with the
intermediate range $1\le m\le n-1$. To this end, we observe that in that
range, 
\begin{multline}
\label{eq:Pochh-1}
\sum_{k=0}^{\min\{m,n-m\}}\frac{(-1)^k}{n-k} \binom{n-k}{k}
\binom{n-2k}{m-k}=\sum_{k=0}^{\min\{m,n-m\}}(-1)^k\frac{(n-k)!}{k!(m-k)!(n-m-k)!}, 
\\
=\frac{1}{m!}\sum_{k=0}^{m}(-1)^k\binom{m}{k}
(n-m-k+1)_{m-1},
\end{multline}
because the additional terms all must vanish in the sum. Here, we use the
standard Pochhammer symbol notation $(x)_j:=x(x+1)\cdots(x+j-1)$ for an integer
$j=1,2,3,\ldots$, with the added rule that $(x)_0=1$. As a function of $n$,
the expression $(n-m+1)_{m-1}$ is a polynomial of degree $m-1$, and hence
its iterated difference of order $m$ must vanish so that the expression
\eqref{eq:Pochh-1} equals $0$. The proof is complete.
\end{proof}

In terms of notation, we use the expression $\lfloor\cdot\rfloor$ to
denote the integer part of a given real number. We continue to analyze
holomorphic functions of two variables with a symmetry property.

\begin{thm}
Let $q=q(z,w)$ be holomorphic in an open domain $\Omega\subset\C^2$,
and suppose that the symmetry $q(z,w)=q(w,z)$ holds whenever
$(z,w)\in\Omega$ and $(w,z)\in\Omega$. Moreover, let
$\Omega_{\mathrm{diag}}:=\{z\in\C:\,(z,z)\in\Omega\}$ be the diagonal
part of $\Omega$. 
Then, for any $n=1,2,3,\ldots$, we have
\begin{equation*}
\oslash(\partial_z^{n}q)(z)=
\frac{n}{2}\sum_{k=0}^{\lfloor \frac{n}{2} \rfloor}
\frac{(-1)^k}{n-k}\binom{n-k}{k}\,\partial_z^{n-2k}
\oslash(\partial_z^k\partial_w^kq)(z),\qquad z\in\Omega_{\mathrm{diag}}.
\end{equation*}
\label{thm:dz-q-1}
\end{thm}

\begin{proof}
By the bivariate chain rule, we have that
\begin{equation}
\begin{aligned}
\partial_z^{n-2k}\oslash(\partial_z^k\partial_w^kq)(z)
&=\oslash\big((\partial_z+\partial_w)^{n-2k}\partial_z^k\partial_w^kq\big)(z)
\\
&=\sum_{j=0}^{n-2k}\binom{n-2k}{j}\oslash\big(
\partial_z^{n-j-k}\partial_w^{j+k} q\big)(z)
\end{aligned}
\end{equation}
so that the sum in question reduces to
\begin{multline*}
\frac{n}{2}\sum_{k=0}^{\lfloor \frac{n}{2} \rfloor}
\frac{(-1)^k}{n-k}\binom{n-k}{k}\,\partial_z^{n-2k}
\oslash(\partial_z^k\partial_w^kq)(z)
\\
=
\frac{n}{2}\sum_{k=0}^{\lfloor \frac{n}{2} \rfloor}\sum_{j=0}^{n-2k}
\frac{(-1)^k}{n-k}\binom{n-k}{k}\binom{n-2k}{j}\oslash\big(
\partial_z^{n-j-k}\partial_w^{j+k} q\big)(z)
\\
=\frac{n}{2}\sum_{k=0}^{\lfloor \frac{n}{2} \rfloor}\sum_{m=k}^{n-k}
\frac{(-1)^k}{n-k}\binom{n-k}{k}\binom{n-2k}{m-k}\oslash\big(
\partial_z^{n-m}\partial_w^{m} q\big)(z)
=\sum_{m=0}^{n}C_{m,n}\oslash\big(\partial_z^{n-m}\partial_w^{m} q\big)(z),
\end{multline*}
where the indicated coefficients $C_{m,n}$ are given by
\[
C_{m,n}:=\frac{n}{2}\sum_{k=0}^{\min\{m,n-m\}}
\frac{(-1)^k}{n-k}\binom{n-k}{k}\binom{n-2k}{m-k}=\frac12\big(\delta_{m,0}+
\delta_{m,n}\big),
\]
in accordance with Lemma \ref{lem:combin-1}. 
This gives that
\begin{equation*}
\frac{n}{2}\sum_{k=0}^{\lfloor \frac{n}{2} \rfloor}
\frac{(-1)^k}{n-k}\binom{n-k}{k}\,\partial_z^{n-2k}
\oslash(\partial_z^k\partial_w^kq)(z)
=\frac12\big(\oslash(\partial_z^{n}q)(z)+\oslash(\partial_w^n q)(z)\big),
\end{equation*}
and since the symmetry assumption gives that
\[
\oslash\big(\partial_z^{j}\partial_w^{k} q\big)(z)
=\oslash\big(\partial_z^{k}\partial_w^{j} q\big)(z),  
\]
the assertion of the theorem follows.
\end{proof}

We continue our study of the Dirichlet symbols
$Q_\e=\mathscr{Q}_\e[\Gamma_\psi]$ associated with a univalent mapping
$\psi\in\Sigma$. By applying Theorem \ref{thm:dz-q-1} to
$q(z,w)=\partial_z^j\partial_w^j Q_\e$, we arrive at the following.

\begin{cor}
For integers $n=1,2,3,\ldots$ and $m=0,\ldots,n-1$, we
have the identity
\begin{equation*}
\oslash(\partial_z^{n}\partial_w^m Q_\e)(z)=
\frac{n-m}{2}\sum_{k=m}^{\lfloor \frac{m+n}{2} \rfloor}
\frac{(-1)^{k+m}}{n-k}\binom{n-k}{k-m}\,\partial_z^{n+m-2k}
\oslash(\partial_z^k\partial_w^k Q_\e)(z),\qquad z\in\D_\e.
\end{equation*}
\label{cor:offdiag-1}
\end{cor}

\begin{rem}
For integers $m,n\in\Z_{\ge0}$ with $m\ne n$, Corollary \ref{cor:offdiag-1}
asserts that the diagonal restriction $\oslash(\partial_z^n\partial_w^m Q_\e)$
may be expressed as a linear combination of derivatives of
the diagonal restrictions $\oslash(\partial_z^k\partial_w^k Q_\e)$, where
$k\in\Z_{\ge0}$ ranges in the interval between $\min\{m,n\}$ and the integer
part of the average $\lfloor\frac{m+n}{2}\rfloor$.
\label{rem:4.4}
\end{rem}

\subsection{Diagonal series of Schwarz\-ian derivatives}

The issue of defining higher Schwarz\-ian derivatives has been considered
earlier in the literature, possibly first by Dov Aharonov \cite{Dov} who
introduced the \emph{Aharonov invariants}
\[
\phi_{k+1}(\psi)(z)=\oslash(\partial_z^k\partial_w Q_\e)(z),\qquad
k=0,1,2,\ldots,
\]
The first two are $\phi_1(\psi)=\frac12\mathrm{N}(\psi)$ and
$\phi_2(\psi)=\frac16\Sop(\psi)$, corresponding to the nonlinearity
and the Schwarz\-ian derivative.
The Aharonov invariants have been studied further, by, e.g.,
Harmelin and Schippers.
Furthermore, Bertilsson \cite{Bertilsson}, who studied under Carleson,
proposed an alternative approach to the introduction of higher Schwarz\-ians
\cite{Bertilsson}, starting from the identity
\[
(\psi'(z))^{\frac12}\partial_z^2(\psi'(z))^{-\frac12}=-\frac12\Sop(\psi)(z),
\]
which then led way to the sequence of expressions
\[
(\psi'(z))^{\frac{k}2}\partial_z^{k+1}2(\psi'(z))^{-\frac{k}2},\qquad k=1,2,3,\ldots.
\]
For a physical perspective on higher Schwarz\-ian derivatives, we mention
Galajinsky's contribution \cite{Galaj}.

We believe that the following definition is appropriate.

\begin{defn}
The \emph{diagonal series of Schwarz\-ian derivatives} of a
univalent function $\psi\in\Sigma$ is the sequence of expressions  
\[
\Sop_{2n}(\psi)(z):=\oslash(\partial_z^n\partial_w^n Q_\e)(z),\qquad
n=0,1,2,\ldots,
\]
where
\[
Q_\e(z,w)=\mathscr{Q}_\e[\Gamma_\psi](z,w)=\log\frac{\psi(z)-\psi(w)}{z-w}
\]
is the Dirichlet symbol of the associated Grunsky operator $\Gamma_\psi$. 
\end{defn}

\begin{rem}
The first in the diagonal series of Schwarz\-ian derivatives is of course
$\Sop_0(\psi)=\oslash Q_\e=\log\psi'$, while the next is
\[
\Sop_2(\psi)(z)=\oslash(\partial_z\partial_w Q_\e)=\frac16\Sop(\psi),
\]  
where $\Sop(\psi)$ is the usual Schwarz\-ian derivative.
Hence, $\Sop_2(\psi)$ involves the third derivative $\psi'''$ of $\psi$
as the highest degree derivative.
The next one, $\Sop_4(\psi)=\oslash(\partial_z^2\partial_w^2 Q_\e)$, then
involves the fifth derivative $\psi^{(5)}$ of $\psi$, and this is the
typical pattern: 
$\Sop_{2n}(\psi)$ will have $\psi^{(2n+1)}$ as its highest degree derivative of
$\psi$. This is in contrast with the Aharonov invariants $\phi_{k+1}(\psi)$,
in which the highest derivative of $\psi$ in $\psi^{(k+2)}$.
\end{rem}

Corollary \ref{cor:offdiag-1} can now be reformulated in terms of the diagonal
series of Schwarz\-ian derivatives.

\begin{cor}
For integers $n=1,2,3,\ldots$ and $m=0,\ldots,n-1$, we
have the identity
\begin{equation*}
\oslash(\partial_z^{n}\partial_w^m Q_\e)(z)=
\frac{n-m}{2}\sum_{k=m}^{\lfloor \frac{m+n}{2} \rfloor}
\frac{(-1)^{k+m}}{n-k}\binom{n-k}{k-m}\,\partial_z^{n+m-2k}
\Sop_{2k}(\psi)(z),\qquad z\in\D_\e.
\end{equation*}
\label{cor:offdiag-3}
\end{cor}

The proof of Corollary \ref{cor:offdiag-3} actually only uses the symmetry
$Q_\e(z,w)=Q_\e(w,z)$ and not the full strength of the fact that
$Q_\e=\mathscr{Q}_\e[\Gamma_\psi]$ holds, but in return it only applies when
$m<n$. So, how can we
obtain a recursion formula which calculates $\Sop_{2n}(\psi)=\oslash
(\partial_z^n\partial_w^n Q)$ in terms of lower-order diagonal series
Schwarz\-ian derivatives? To this end, we need
to invoke the nonlinear wave equation or alternatively, the Riccati equation,
localized in Corollaries \ref{cor:local-1} and \ref{cor:Riccati}.
For our purposes, the Riccati equation is more informative, so we use
Corollary \ref{cor:Riccati}, with even $k=2j$, $j=1,2,3,\ldots$, while
using the symmetry
$\oslash\partial_z^m\partial_w^n Q_\e=\oslash\partial_z^n\partial_w^m Q_\e$:
\begin{multline}
\partial_z\oslash(\partial_z^{2j}\partial_w Q_\e)(z)
-\frac{2j+3}{2j+1}\oslash(\partial_z^{2j+1}\partial_w Q_\e)(z)
\\
+\sum_{l=1}^{2j-1}
\binom{2j}{l}\big(\oslash(\partial_z^{2j-l}\partial_wQ_\e)(z)\big)
\big(\oslash(\partial_z^{l}\partial_w Q_\e)(z)\big)=0.
\label{eq:Riccati2.01}
\end{multline}
As a first step, we invoke Corollary \ref{cor:offdiag-3} with $n=2j$ and
$m=1$, which gives that
\begin{equation}
\partial_z \oslash(\partial_z^{2j}\partial_w Q_\e)(z)=
(j-\tfrac12)\sum_{k=1}^{j}
\frac{(-1)^{k-1}}{2j-k}\binom{2j-k}{k-1}\,\partial_z^{2j+2-2k}
\Sop_{2k}(\psi)(z).
\label{eq:prep-1.0}
\end{equation}
Next, with $n=2j+1$ and $m=1$, Corollary \ref{cor:offdiag-3} gives that
\begin{multline}
\oslash(\partial_z^{2j+1}\partial_w Q_\e)(z)=
j\sum_{k=1}^{j+1}
\frac{(-1)^{k-1}}{2j+1-k}\binom{2j+1-k}{k-1}\,\partial_z^{2j+2-2k}
\Sop_{2k}(\psi)(z)
\\
=(-1)^{j}\Sop_{2j+2}(\psi)(z)+j\sum_{k=1}^{j}
\frac{(-1)^{k-1}}{2j+1-k}\binom{2j+1-k}{k-1}\,\partial_z^{2j+2-2k}
\Sop_{2k}(\psi)(z),
\label{eq:prep-1.1}
\end{multline}
so that by a combination of \eqref{eq:prep-1.0} and \eqref{eq:prep-1.1}, we
obtain, after some simplification, that
\begin{multline}
\partial_z\oslash(\partial_z^{2j}\partial_w Q_\e)(z)
-\frac{2j+3}{2j+1}\oslash(\partial_z^{2j+1}\partial_w Q_\e)(z)
\\
=
(-1)^{j-1}\frac{2j+3}{2j+1}
\Sop_{2j+2}(\psi)(z)
+\sum_{k=1}^{j}(-1)^{k-1}C_{k,j}\,\partial_z^{2j+2-2k}
\Sop_{2k}(\psi)(z),
\label{eq:prep-1.1'}
\end{multline}
where we introduce the coefficients
\begin{equation}
C_{k,j}:=\frac{j-\frac12}{2j-k}\binom{2j-k}{k-1}
-\frac{j(2j+3)}{(2j+1)(2j+1-k)}\binom{2j+1-k}{k-1}.
\label{eq:prep-1.1''}
\end{equation}

\begin{thm}
Given $\psi\in\Sigma$, for $j=1,2,3,\ldots$, we have the recursion formula
\begin{multline}
\Sop_{2j+2}(\psi)(z)=\frac{2j+1}{2j+3}
\sum_{k=1}^{j}(-1)^{j+k-1}C_{k,j}\partial_z^{2j+2-2k}\Sop_{2k}(\psi)(z)
\\
+(-1)^j\frac{2j+1}{2j+3}
\sum_{l=1}^{2j-1}\binom{2j}{l}\big(\oslash(\partial_z^{2j-l}\partial_w
Q_\e)(z)\big)\big(\oslash(\partial_z^{l}\partial_w
Q_\e)(z)\big),
\end{multline}
Here $C_{k,j}$ is given by \eqref{eq:prep-1.1''}. The remaining
diagonal restrictions are
\[
\oslash(\partial_z^{2j-l}\partial_w Q_\e),\qquad
\oslash(\partial_z^{l}Q_\e),\qquad
\oslash(\partial_z^{2j+1-l}Q_\e).
\]
They are calculated using Corollary \ref{cor:offdiag-3}. 
\label{thm:recursion-1}
\end{thm}

\begin{proof}
The identity results from Riccati identity \eqref{eq:Riccati2.01} if we
implement the formula \eqref{eq:prep-1.1'}. 
\end{proof}

\begin{rem}
In view of Corollaries \ref{cor:offdiag-1} and \ref{cor:offdiag-3}, combined
with Remark \ref{rem:4.4}, we may interpret Theorem \ref{thm:recursion-1}
as expressing $\Sop_{2j+2}(\psi)$ as a linear comination of derivatives of 
lower-order Schwarz\-ians $S_{2k}(\psi)$, with $1\le k\le j$, together with a
sum of bilinear expressions of the form
\[
\big(\partial_z^{2j-l+1-2k}\Sop_{2k}(\psi)\big)
\big(\partial_z^{l+1-2k'}\Sop_{2k'}(\psi)\big),
\]
where
\[
1\le k\le\Big\lfloor j-\frac{l-1}{2}\Big\rfloor,\quad
1\le k'\le 2\Big\lfloor\frac{l+1}{2}\Big\rfloor,\quad 1\le l\le 2j-1.
\]
In particular, $\Sop_{2j+2}(\psi)$ is a differential polynomial in the
lower-order Schwarz\-ians $\Sop_{2k}(\psi)$, where $k=1,\ldots,j$. By
induction, every $\Sop_{2j+2}(\psi)$ is therefore a differential polynomial in
the basic Schwarz\-ian $\Sop_2(\psi)$. Moreover, Corollary \ref{cor:offdiag-3}
implies that every diagonal restriction
$\oslash\partial_z^n\partial_w^m Q_\e$, where $m,n=1,2,3,\ldots$, is likewise
a differential polynomial in $\Sop_2(\psi)$. 
\end{rem}

\subsection{Calculating lower order diagonal series Schwarz\-ian derivatives}

For a given normalized conformal mapping $\psi\in\Sigma$, we know that
$\Sop_0(\psi)=\log\psi'$ and that $\Sop_2(\psi)=\frac16\Sop(\psi)$, where
$\Sop(\psi)$ is the usual Schwarz\-ian derivative. It turns out that the
next item in the diagonal series of Schwarz\-ian derivatives, the
expression $\Sop_4(\psi)$, also has a simple clean expression.

\begin{thm}
\label{thm:S4}
We have that
\[
\Sop_4(\psi)(z)=\frac{1}{5}(\Sop_2(\psi))''(z)-\frac{6}{5}(\Sop_2(\psi)(z))^2
=\frac{1}{30}\big((\Sop(\psi))''(z)-(\Sop(\psi))^2\big),
\qquad z\in\D_\e.  
\]
\end{thm}

\begin{proof}
By Theorem \ref{thm:recursion-1} with $j=1$,
\begin{multline}
\Sop_{4}(\psi)(z)=-\frac{3}{5}C_{1,1}\partial_z^{2}\Sop_2(\psi)(z)
-\frac{3}{5}\binom{2}{1}\big(\oslash(\partial_z\partial_w Q_\e)(z)\big)^2
\\
=\frac15 \partial_z^2\Sop_2(\psi)(z)-\frac{6}{5}(\Sop_2(\psi)(z))^2,
\label{eq:beginning-1}
\end{multline}
and the claimed assertion follows.

\end{proof}  

\begin{rem}
It is possible to calculate the diagonal series Schwarz\-ian derivatives
$S_{2k}(\psi)$ for $k>2$ as well.
For instance, by Theorem \ref{thm:recursion-1}, with $j=2$, we have
\begin{multline}
\Sop_{6}(\psi)(z)=\frac{5}{7}
\sum_{k=1}^{2}(-1)^{k+1}C_{k,2}\partial_z^{6-2k}\Sop_{2k}(\psi)(z)
\\
+\frac{5}{7}
\sum_{l=1}^{3}\binom{4}{l}\big(\oslash(\partial_z^{4-l}\partial_w
Q_\e)(z)\big)\big(\oslash(\partial_z^{l}\partial_w
Q_\e)(z)\big)
\end{multline}
where we evaluate using \eqref{eq:prep-1.1''} that $C_{1,2}=-\frac15$ and
$C_{2,2}=-\frac{13}{10}$. Next, using Corollary \ref{cor:offdiag-3} repeatedly, 
we find that
\begin{multline}
\sum_{l=1}^{3}\binom{4}{l}\big(\oslash(\partial_z^{4-l}\partial_w
Q_\e)(z)\big)\big(\oslash(\partial_z^{l}\partial_w
Q_\e)(z)\big)
\\
=2\binom{4}{1}\big(\oslash(\partial_z^{3}\partial_w
Q_\e)(z)\big)\big(\oslash(\partial_z\partial_w
Q_\e)(z)\big)+\binom{4}{2}\big(\oslash(\partial_z^{2}\partial_w
Q_\e)(z)\big)^2
\\
=4\Sop_2(\psi)\bigg(\partial_z^2\Sop_2(\psi)-2\Sop_4(\psi)\bigg)
+\frac{3}{2}(\partial_z\Sop_2(\psi))^2.
\end{multline}
Putting things together, then, we obtain the formula
\begin{multline}
\Sop_{6}(\psi)(z)=\frac{5}{7}
\bigg(-\frac15\,\partial_z^{4}\Sop_{2}(\psi)(z)+\frac{13}{10}
\partial_z^2\Sop_4(\psi)(z)\bigg)
\\
+\frac{5}{7}\Big(4\Sop_2(\psi)\bigg(\partial_z^2\Sop_2(\psi)
-2\Sop_4(\psi)\Big)
+\frac{3}{2}(\partial_z\Sop_2(\psi))^2\bigg)
\\
=-\frac17\,\partial_z^{4}\Sop_{2}(\psi)(z)+\frac{13}{14}
\partial_z^2\Sop_4(\psi)(z)
\\
+\frac{20}{7}\Sop_2(\psi)\Big(\partial_z^2\Sop_2(\psi)
-2
\Sop_4(\psi)\Big)
+\frac{15}{14}(\partial_z\Sop_2(\psi))^2.
\end{multline}
Should we prefer, we may be reduce this expression further so that it only
involves the Schwarz\-ian derivative $\Sop_2(\psi)=\frac16\Sop(\psi)$.
\end{rem}



\subsection{Pointwise bounds on the diagonal series Schwarz\-ian derivatives}
We need to obtain a good pointwise bound for the diagonal series Schwarz\-ian
derivatives. In fact, we shall do this in the context of a general contraction
$\Tope$ on $L^2(\D)$. Our starting point is the definition \eqref{eq:Qe-def-1},
which gives that
\begin{equation}
\partial_z^n\partial_w^n\mathscr{Q}_\e[\Tope](z,w)=\partial_z^n\partial_w^n
\langle \Tope(\overline{s_w^\e}),s_z^\e\rangle_{\D_\e}=
\langle \Tope(\overline{s_w^{\e,n}}),s_z^{\e,n}\rangle_{\D_\e},
\label{eq:QextT-derivn}
\end{equation}
where
\begin{equation}
s_z^{\e,n}(\xi):=\bar\partial_z^ns_z^\e(\xi)
=\bar\partial_z^n\frac{1}{\xi(1-\xi \bar z)}
=\frac{n!\xi^{n-1}}{(1-\xi\bar z)^{n+1}}.
\end{equation}
By \eqref{eq:QextT-derivn}, we get the pointwise bound
\begin{equation}
\big|\partial_z^n\partial_w^n\mathscr{Q}_\e[\Tope](z,w)\big|
\le\|\Tope\|\,\|s_w^{\e,n}\|_{L^2(\D_\e)}\|s_z^{\e,n}\|_{L^2(\D_\e)}\le
\|s_w^{\e,n}\|_{L^2(\D_\e)}\|s_z^{\e,n}\|_{L^2(\D_\e)},
\label{eq:QextT-derivn2}
\end{equation}
where in the second step, we used that $\Tope$ is a contraction. 
We are mainly interested in the values along the diagonal, where the bound is
\begin{equation}
\big|\oslash(\partial_z^n\partial_w^n\mathscr{Q}_\e[\Tope])(z)\big|
\le\|s_z^{\e,n}\|_{L^2(\D_\e)}^2.
\label{eq:QextT-derivn3}
\end{equation}
It remains to calculate the norm of $s_z^{\e,n}$:
\begin{multline}
\|s_z^{\e,n}\|_{L^2(\D_\e)}^2=(n!)^2\int_{\D_\e}
\frac{|\xi|^{2n-2}}{|1-\xi\bar z|^{2n+2}}\dA(\xi)=
\sum_{k=0}^{+\infty}(k+1)_n(k+2)_{n-1}|z|^{-2k-2n-2}
\\
\le\sum_{k=0}^{+\infty}(k+1)_{2n-1}|z|^{-2k-2n-2}=(2n-1)!\,|z|^{2n-2}(|z|^2-1)^{-2n},
\end{multline}
provided that $n=1,2,3,\ldots$ and $z\in\D_\e$. We should mention that the
latter inequality is only asymptotically sharp as $|z|\to1^+$. 
We have obtained the following result.

\begin{thm}
If $\Tope$ is a contraction on $L^2(\D)$, and $n=1,2,3,\ldots$, then the
inequality
\begin{equation}
\big|\oslash(\partial_z^n\partial_w^n\mathscr{Q}_\e[\Tope])(z)\big|
\le\sum_{k=0}^{+\infty}(k+1)_n(k+2)_{n-1}|z|^{-2k-2n-2}
\end{equation}
holds for all $z\in\D_\e$. In particular, if $\psi\in\Sigma$, then
the inequality
\begin{equation}
|\Sop_{2n}(\psi)(z)|\le \sum_{k=0}^{+\infty}(k+1)_n(k+2)_{n-1}|z|^{-2k-2n-2}
\end{equation}
holds for each $z\in\D_\e$ and all $n=1,2,3,\ldots$, where
$\Sop_{2n}(\psi)$ is from the diagonal sequence of Schwarz\-ian derivatives. 
\label{thm:Schwarzian-bound-1}
\end{thm}

\begin{proof}
The assertion for general contractions $\Tope$ was established above. For
$\Tope=\Gamma_\psi$, we obtain the estimate for diagonal Schwarz\-ian
derivatives.  
\end{proof}

\begin{rem}
For $n=1$, the estimate of $\Sop_2(\psi)$ is the classical bound
\[
|\Sop_2(\psi)(z)|=\frac16|\Sop(\psi)(z)|\le
\sum_{k=0}^{+\infty}(k+1)|z|^{-2k-4}=\frac{1}{(|z|^2-1)^2},
\qquad z\in\D_\e.  
\]
A family of mappings $\psi$ which achieves equality at any given point
are the airfold (or Joukowski) maps
\[
\psi_\alpha(z)=z+\frac{\alpha^2}{z}
\]
where $\alpha\in\T$. The mapping $\psi_\alpha$ then maps $\D_\e$ onto the
complement of the line segment which connects the two points $\pm2\alpha$.
One can check that the family of airfoil maps $\psi_\alpha$ achieves pointwise
equality in the estimate of Theorem \ref{thm:Schwarzian-bound-1} for $n=2$
as well:
\begin{equation}
|\Sop_{4}(\psi)(z)|\le \sum_{k=0}^{+\infty}(k+1)(k+2)^2|z|^{-2k-6}=2(2|z|^2+1)
(|z|^2-1)^{-4}.   
\end{equation}
The same phenomenon should extend to all $\Sop_{2n}(\psi)$ for
$n=3,4,5,\ldots$ as well.
\end{rem}

\subsection{A tentative Nehari-type theorem for
\texorpdfstring{$S_4(\psi)$}{S4(psi)}}
In \cite{Nehari}, Zeev Nehari established the following theorem.

\begin{thm} {\rm(Nehari)}
Suppose that $F:\D_e\to\C$ is holomorphic, with the bound
\[
|F(z)|\le \frac{2}{(|z|^2-1)^2},\qquad z\in\D_\e. 
\]
Then there exists a normalized conformal mapping $\psi\in\Sigma$ with
$\Sop(\psi)=6\Sop_2(\psi)=F$. Here, the constant $2$ is optimal and cannot
be replaced by any larger constant. 
\label{thm:Nehari}
\end{thm}

\begin{rem}
\noindent{$(a)$}
In view of Proposition \ref{prop:5equiv}, the mapping $\psi$ is uniquely
determined by the equation $\Sop(\psi)=6\Sop_2(\psi)=F$ up to additive
constants.

\noindent{$(b)$} If the constant $2$ is replaced by a constant $<2$, then
the conformal mapping $\psi$ extends to a global quasiconformal mapping of
the plane onto itself, and the boundary $\psi(\T)$ is a quasicircle. 
\label{rem:Nehari1}
\end{rem}

\begin{problem}
Find the optimal universal constant $\epsilon_0>0$ such that if $G:\D_\e\to\C$
is holomorphic with
\[
|G(z)|\le \epsilon_0(2|z|^2+1)(|z|^2-1)^{-4}, \qquad z\in\D_\e,
\]
then there exists a normalized conformal mapping $\psi\in\Sigma$ with
$\Sop_4(\psi)=G$.
\label{prob:Nehari}
\end{problem}

It is not obvious that such a positive constant $\epsilon_0$ exists.
Note that in view of Theorem \ref{thm:S4}, we can split the equation
$\Sop_4(\psi)=G$ into two steps: $(a)$ write $F:=\Sop_2(\psi)$, and $(b)$
solve for $F$ in the equation $F''(z)-6(F(z))^2=5G(z)$.

The following result was conveyed to us by Qiyuan (Alex) Gu, a student
at the University of Chicago. We have optimized his argument to get the
best constant the method allows for.

\begin{thm}
{\rm(Gu)}
Suppose $G:\D_\e\to\C$ is holomorphic with
\[
|G(z)|\le \epsilon_1(2|z|^2+1)(|z|^2-1)^{-4}, \qquad z\in\D_\e,
\]
where $\epsilon_1\le\frac{22}{45}$. Then there exists a holomorphic function
$F:\D_\e\to\C$ which solves the equation $F''(z)-6(F(z))^2=5G(z)$ with
\[
\sup\big\{(|z|^2-1)^2|F(z)|:\,\,z\in\D_\e\big\}\le\frac13.
\]
Finally, if $\epsilon_1<\frac{22}{45}$, then strict inequality holds in
the conclusion.
\label{thm:Gu}
\end{thm}

\begin{proof}
We work in the Banach space of holomorphic functions on $\D_\e$ with
finite Korenblum norm
\[
\|F\|_2:=\sup\big\{(|z|^2-1)^2|F(z)|:\,z\in\D_\e\big\}.
\]
Such functions satisfy $F(z)=\Ordo(|z|^{-4})$ at infinity, while the
hypothesis on $G$ gives $G(z)=\Ordo(|z|^{-6})$.
For a holomorphic function $h$ decaying sufficiently rapidly at infinity,
let $\mathrm{I}_\infty^2h$ denote its twice-integrated primitive which,
together with its first derivative, vanishes at infinity. Along a radial
ray it is given by
\[
\mathrm{I}_\infty^2h(z)
=\int_z^\infty(\xi-z)h(\xi)\diff\xi,
\qquad z\in\D_\e.
\]
The differential equation is therefore equivalent to the fixed-point
equation
\[
F=\Phi(F):=\mathrm{I}_\infty^2(5G+6F^2).
\]

Put $\omega(t)=(t^2-1)^{-2}$ for $t>1$. Direct differentiation gives
\[
\begin{aligned}
\omega''(t)
&=\frac{24}{(t^2-1)^4}+\frac{20}{(t^2-1)^3}\\
&=8\frac{2t^2+1}{(t^2-1)^4}+\frac4{(t^2-1)^3}.
\end{aligned}
\]
Since $\omega$ and $\omega'$ vanish at infinity, integration twice yields
\[
\int_r^\infty\frac{t-r}{(t^2-1)^4}\diff t
\le\frac{\omega(r)}{24},
\qquad
\int_r^\infty\frac{(t-r)(2t^2+1)}{(t^2-1)^4}\diff t
\le\frac{\omega(r)}8.
\]
Using these estimates along radial rays, we obtain
\[
\|\Phi(F)\|_2
\le\frac{5\epsilon_1}{8}+\frac14\|F\|_2^2
\]
and
\[
\|\Phi(F)-\Phi(H)\|_2
\le\frac14(\|F\|_2+\|H\|_2)\|F-H\|_2.
\]
For $\|F\|_2\le\frac13$ and $\epsilon_1\le\frac{22}{45}$, we have
\[
\|\Phi(F)\|_2
\le\frac5{8}\times\frac{22}{45}+\frac14\times\frac19
=\frac13.
\]
Moreover, $\Phi$ is a contraction on this closed ball, with Lipschitz
constant at most $\frac16$. By the contraction mapping principle there
exists a solution $F$ with $\|F\|_2\le\frac13$.
If $\epsilon_1<\frac{22}{45}$, the preceding norm estimate shows that
$\|F\|_2<\frac13$. The proof is complete.
\end{proof}


\begin{cor}
Suppose $G:\D_\e\to\C$ is holomorphic with
\[
|G(z)|\le \epsilon_1(2|z|^2+1)(|z|^2-1)^{-4}, \qquad z\in\D_\e,
\]
where $\epsilon_1\le\frac{22}{45}$. Then there exists a normalized
conformal mapping $\psi\in\Sigma$ with $\Sop_4(\psi)=G$. Moreover, if
$\epsilon_1<\frac{22}{45}$, $\psi$ extends to a global quasiconformal mapping
of the plane onto itself, and the boundary $\psi(\T)$ is a quasicircle.
\label{cor:S4-2}
\end{cor}

\begin{proof}
The assertion is a consequence of Theorem \ref{thm:Gu} and Theorem
\ref{thm:Nehari} (Nehari's theorem), together with Remark
\ref{rem:Nehari1}$(b)$. 
\end{proof}

\begin{rem}
It follows from Corollary \ref{cor:S4-2} that the optimal universal constant
$\epsilon_0$ of Problem \ref{prob:Nehari} exists and has
$\epsilon_0\ge\frac{22}{45}=0.4888888\ldots$.
\end{rem}

Next, we look into uniqueness issues for the equation
$F''(z)-6(F(z))^2=5G(z)$.

\begin{thm}
Suppose $F,G$ are meromorphic functions on $\D_\e$, with decays
$F(z)=\Ordo(|z|^{-4})$ and $G(z)=\Ordo(|z|^{-6})$ at infinity. If these
functions are connected via the differential equation
$F''(z)-6(F(z))^2=5G(z)$, then $F$ is determined by $G$, and vice versa.
\label{thm:uniq-infty1}
\end{thm}

\begin{proof}
Clearly, the relation $F''(z)-6(F(z))^2=5G(z)$ determines $G$ if $F$ is given,
so we concentrate our effort on showing that $F$ is determined by $G$ as
well.  

By the given decay, the point at infinity is a regular point for both
functions. Hence they both have convergent Laurent expansions near infinity,
\[
F(z)=\sum_{n=4}^{+\infty}f_nz^{-n},\quad
G(z)=\sum_{n=6}^{+\infty}g_nz^{-n},
\]
Since $F''(z)=5G(z)+6(F(z))^2$, we may expand both sides in Laurent series,
\begin{equation}
\sum_{n=6}^{+\infty}(n-2)(n-1)\,f_{n-2}\,z^{-n}=5\sum_{n=6}^{+\infty}g_nz^{-n}
+6\sum_{n=8}^{+\infty}\sum_{k=4}^{n-4}f_{k}f_{n-k}z^{-n}, 
\end{equation}
which gives the equations of coefficients
\begin{equation}
(n-2)(n-1)f_{n-2}=5g_n,\qquad n=6,7,
\label{eq:rec-1}
\end{equation}
and
\begin{equation}
(n-2)(n-1)f_{n-2}=5g_n+6\sum_{k=4}^{n-4}f_kf_{n-k},\qquad n=8,9,10,\ldots.
\label{eq:rec-2}
\end{equation}
The first equation gives us the coefficient relations $f_4=\frac14 g_6$ and
$f_5=\frac16 g_7$. Knowing already $f_4,f_5$, the second equation then gives us
$f_6,f_7$ in terms of $f_4,f_5$ and $g_8,g_9$. Armed with $f_k$ for
$k=4,5,6,7$, we then obtain from \eqref{eq:rec-2} the coefficients $f_k$
for all $k=4,\ldots,12$. Proceeding iteratively, we recursively obtain all the
coefficients $f_k$ for $k=4,5,6,\ldots$. 
Finally, given that both $F$ and $G$ are meromorphic functions, they are
uniquely determined by their Laurent coefficients at infinity. Hence $F$
is fully determined by $G$.
\end{proof}

\begin{cor}
Given that $\psi\in\Sigma$ and a holomorphic $G:\D_\e\to\C$ with decay
$G(z)=\Ordo(|z|^{-6})$ at infinity, the equation $\Sop_4(\psi)=G$, if it has
a solution, determines $\psi$ up to an additive constant.
\end{cor}

\begin{proof}
We declare that $F:=\Sop_2(\psi)$, so that in view of Theorem \ref{thm:S4},
$F$ solves the differential equation $F''(z)-6(F(z))^2=5G(z)$.
Moreover, given the estimates of Theorem \ref{thm:Schwarzian-bound-1},
with $n=1,2$, we have the decay $F(z)=\Ordo(|z|^{-4})$ and
$G(z)=\Ordo(|z|^{-6})$ at infinity. This puts us in the setting of
Theorem \ref{thm:uniq-infty1}, so that  $F$ is uniquely determined by $G$.
It remains to use Proposition \ref{prop:5equiv} to obtain that the equation
$\Sop_2(\psi)=F$ determines $\psi\in\Sigma$ uniquely up to an additive
constant. The proof is complete.
\end{proof}

Second-order linear or nonlinear differential equations typically have a
two-dimensional manifold of solutions (i.e., the general solution will
involve two free parameters). Why that is not the case in Theorem
\ref{thm:uniq-infty1} is because of the decay data at infinity. It is then
a natural question to ask for the full manifold of solutions. We carry this
out first for vanishing right-hand side $G=0$. Interestingly, the
Weierstrass elliptic $\wp$ function makes an appearance. 

\begin{thm}
Suppose $F:\D_\e\to\C$ is a meromorphic solution to the nonlinear
differential equation $F''-6(F)^2=0$. Then one of the following will occur:

\noindent{$(a)$} $F(z)\equiv0$.

\noindent{$(b)$} $F(z)=(z+\gamma_0)^{-2}$ for some point $\gamma_0\in\C$. 

\noindent{$(c)$} $F(z)=\wp(z+\gamma_0,\Lambda)$ for some point $\gamma_0\in\C$,
where $\wp(z,\Lambda)$ is the Weierstrass elliptic function associated with a
hexagonal lattice $\Lambda$.
\label{thm:secondordereq}
\end{thm}

\begin{rem}
Here, we may understand that the two free complex parameters are the
translation parameter $\gamma_0$ and the dilation-rotation constant
associated with the hexagonal lattice $\Lambda$.   
\end{rem}

\begin{proof}[Proof of Theorem \ref{thm:secondordereq}]
We first verify that the claimed solutions are indeed solutions. Under $(a)$,
if $F=0$, the equation is clearly fulfilled. Moreover, under $(b)$,
if instead $F(z)=(z+\gamma_0)^{-2}$,
then $F$ is clearly meromorphic in $\C$ as a rational function.
Moreover, we calculate that $F''(z)=6(z+\gamma_0)^{-4}$, while
$6(F(z))^2=6(z+\gamma_0)^{-4}$, so that $F''(z)-6(F(z))^2=0$ holds.
Finally, under $(c)$, the indicated solution $F(z)=\wp(z+\gamma_0,\Lambda)$
solves the Weierstrass elliptic equation
\[
(F'(z))^2-4(F(z))^3=\mathrm{g}_2 F(z)+\mathrm{g}_3,
\]
where the coefficients $\mathrm{g}_2,\mathrm{g}_3$ are given by the
Eisenstein series
\[
\mathrm{g}_2=-60\sum_{\lambda\in\Lambda\setminus\{0\}}\frac{1}{\lambda^4},
\quad \mathrm{g}_3=-140\sum_{\lambda\in\Lambda\setminus\{0\}}\frac{1}{\lambda^6}.
\]
Here, we point out the explicit series representation of the Weierstrass
elliptic function 
\[
\wp(z,\Lambda):=\frac{1}{z^2}+\sum_{\lambda\in\Lambda\setminus\{0\}}
\bigg(\frac{1}{(z-\lambda)^2}-\frac{1}{\lambda^2}\bigg), 
\]
which is clearly meromorphic in $\C$. 
In the given setting, because $\Lambda$ is a hexagonal lattice, we know that
$\mathrm{g}_2=0$. Differentiating the Weierstrass elliptic equation with
$\mathrm{g}_2=0$ gives that
\[
2F'(z)F''(z)-12F(z)^2F'(z)=0,
\]
and since the derivative $F'$ does not vanish identically, we may divide by
it:
\[
F''(z)-6F(z)^2=0,
\]
which is our given differential equation.

We turn to the harder part, which is to show that the given list of solutions
is complete. We multiply the differential equation by $2F'$, and get
$2F'F''-12F^2F'=0$. Taking primitives on both sides, we find that
\[
(F')^2-4(F)^3=C,
\]
where $C$ is a constant.

\noindent{\sc Case I}: $C=0$. Then the equation reads
$(F')^2=4(F)^3$, which is a first-order nonlinear autonomous differential
equation. As such, it is specified uniquely by an initial value $F(z_0)$
for a given point $z_0\in\D_\e$. If $F(z_0)=0$, the only solution is then
$F(z)\equiv0$. 
If $F(z_0)\ne0$, we can form the square root $\sqrt{F}$ at least in a
neighborhood of $z_0$, and extend it to all of $\D_\e$ if we accept some
branching between the zeros of $F$. The differential equation $(F')^2=4(F)^3$
then gives that 
\[
\diff \frac{1}{\sqrt{F}}=\pm\diff z,
\]
so that for some constant $c\in\C$, 
\[
\frac{1}{\sqrt{F}}=c\pm z \quad\Longleftrightarrow\quad F=(c\pm z)^{-2}.
\]
In case the ``$\pm$'' sign is a ``$+$'', we simply choose $\gamma_0=c$,
so that $F=(z+\gamma_0)^{-2}$. 
If instead the ``$\pm$'' sign is a ``$-$'', we choose $\gamma_0=-c$, so that
$F=(c-z)^{-2}=(-\gamma_0-z)^{-2}=(z+\gamma_0)^{-2}$. 

\noindent{\sc Case II}: $C\ne0$. The differential equation is
$(F')^2-4(F)^3-C=0$, where $C$ is a constant. Again, this is a first-order
autonomous nonlinear differential equation, which by general theory is
determined by an initial value $F(z_0)$ at a given point $z_0\in\D_e$.
This is a classical ordinary differential equation with general solution
\[
F(z)=\wp(z+\gamma_0,\Lambda),
\]
where $\wp$ is the Weierstrass elliptic function for he lattice $\Lambda$,
and $\gamma_0\in\C$ is a translation parameter. The lattice is special
because the in the Weierstrass elliptic equation, we have
$\mathrm{g}_2=0$ and $\mathrm{g}_3=C$. The condition $\mathrm{g}_2=0$ forces
the $\Lambda$ to be a hexagonal lattice.
The proof is complete.
\end{proof}



\subsection{The $\Sop_4$-equation with more general right-hand side}

As mentioned earlier in the context of Nehari's theorem (Theorem
\ref{thm:Nehari}), it is natural to study the equation
$\Sop_4(\psi)=G$ for some given holomorphic function $G$ which does not
necessarily obey the decay condition at infinity of Theorem
\ref{thm:uniq-infty1}. Putting $F=\Sop_2(\psi)$, we are then led by
Theorem \ref{thm:S4} to first try and solve the differential equation
\begin{equation}
F''(z)-6(F(z))^2=5\,G(z).
\label{eq:Pain-1}
\end{equation}
If we relax the decay requirements on $G$ and just ask for a meromorphic
solution $\psi$, for a given meromorphic $G$, we arrive at interesting
questions. We would then search for
meromorphic solutions $F$ to the above differential equation, and in a
second step try to find a meromorphic $\psi$ with
\[
\Sop_2(\psi)=\frac16\bigg(\frac{\psi''}{\psi'}\bigg)'-
\frac{1}{12}\bigg(\frac{\psi''}{\psi'}\bigg)^2=F.
\]
The latter equation is of Riccati type, which is why in principle we know
how to solve it once we found $F$. If $G$ is allowed a moducum of growth at
infinity, we might consider, e.g., $G=z/5$.  In this case, the equation
\eqref{eq:Pain-1} is called the \emph{Painlev\'e I equation}, with deep roots
in Integrable Systems and Hamiltonian Dynamics. But already the case when $G$
is constant is interesting.

\begin{thm}
The solutions to the second order equation $F''(z)-6(F(z))^2=\alpha$, where
$\alpha\in\C\setminus\{0\}$ is a constant, are of the form
$F(z)=\wp(z+\gamma_0,\Lambda)$, where $\wp$ is the
Weierstrass elliptic function associated with a lattice $\Lambda$ (which
usually has two $\R$-linearly independent generators, but may degenerate
to a single generator) with the parameter choices
\[
\mathrm{g}_2 =-60 \sum_{\lambda\in\Lambda\setminus\{0\}}\frac{1}{\lambda^4}=2\alpha,
\qquad
\mathrm{g}_3=-140\sum_{\lambda\in\Lambda\setminus\{0\}}\frac{1}{\lambda^6}\in\C.
\]
\end{thm}

\begin{proof}
We multiply the differential equation by $2F'$ and obtain that
$2F'F''-12(F)^2F'=2\alpha F'$. Forming the primitive of both sides of
this equation, we find that
\[
(F'(z))^2-4(F(z))^3=2\alpha F+C,
\]
where $C$ is an integration constant. This autonomous first-order equation is
associated with the Weierstrass elliptic function, and the solution is of the
form $F(z)=\wp(z+\gamma_0,\Lambda)$, where the lattice is given implicitly by
\[
\sum_{\lambda\in\Lambda\setminus\{0\}}\frac{1}{\lambda^4}=-\frac{\alpha}{30},
\qquad
\sum_{\lambda\in\Lambda\setminus\{0\}}\frac{1}{\lambda^6}=-\frac{C}{140}.
\]
The proof is complete.
\end{proof}
  
\subsection{The diagonal Schwarz\-ian derivative $\Sop_4$ and compositions}

By \eqref{eq:cocycle-2}, the rescaled Schwarz\-ian derivative
$\Sop_2=\frac16\Sop$ enjoys the composition rule
\[
\Sop_2(f\circ\psi)=(\psi')^2(\Sop_2(f)\circ\psi)+\Sop_2(\psi),
\]
if $f$ and $\psi$ are (locally) conformal mappings. Is there a corresponding
rule which applies in the context of $\Sop_4(\psi)$? 

\begin{thm}
If $f$ and $\psi$ are locally conformal, then
\[
\Sop_4(f\circ \psi)=
[(\Sop_4(f))\circ \psi](\psi^{\prime})^4+
S_4(\psi)+[(S_2(f))^{\prime}\circ \psi]
(\psi^{\prime})^2\psi^{\prime \prime}+
[(S_2(f))\circ \psi] (\psi^{\prime\prime})^2.
\]
\end{thm}

\begin{proof}
The identity follows from Theorem \ref{thm:S4} and the chain rule
for the Schwarz\-ian derivative $\Sop_2$.
\end{proof}

\section{Parametrization of the manifold $\Sigma/\C$}
\label{sec:parameter}

\subsection{The nonlinear wave equation and parametrization of $\Sigma/\C$}

According to the Theorem \ref{Thm_A}, together with Proposition
\ref{prop:5equiv}, the holomorphic functions $Q_\e:\D_\e^2\to\C$ with the
properties $(a)$ and $(b)$ of Theorem \ref{Thm_A} define uniquely (up to
an additive constant) a normalized conformal mapping $\psi\in\Sigma$ such that
$Q_\e=\mathscr{Q}_\e[\Gamma_\psi]$. So, in a sense, such holomorphic solutions
to the nonlinear wave equation provide a way to parametrize the manifold
$\Sigma/\C$. By the way, $\Sigma/\C$ may be viewed as a large universal
Teichm\"uller space, the usual universal Teichm\"uller space $\mathscr{T}_1$
in the sense of Lipman Bers \cite{Bers}
consisting of a ``smooth `` portion of $\Sigma/\C$ (the elements $\psi+C\in
\Sigma/\C$ with the property that $\psi(\T)$ is a quasicircle). Please note
that if $\psi$ is such that $\psi+C\in\Sigma/\C$ is in the universal
Teichm\"uller space $\mathscr{T}_1$, then $\psi$ extends to a quasi-conformal
homoeomorphism of the plane $\C$ with
$\bar\partial_z\psi(z)=\mu(z)\partial_z\psi(z)$, where the \emph{Beltrami
coefficient} $\mu\in L^\infty(\C)$ with $\mathrm{supp}\,\mu\subset\bar\D$ and
$\sup|\mu|<1$. But then the formula for $Q_\e$ extends too, and it solves the
differential equation
\[
\bar\partial_z Q_\e(z,w)=\mu(z)\partial_z Q_\e(z)+\frac{\mu(z)}{z-w}.
\]

The basic underlying idea of the diagonal localization of the nonlinear
wave equation in Proposition \ref{cor:local-1} is the equivalence
\[
Q_\e \longleftrightarrow \{Q_{j,k}\}_{j,k},\qquad Q_{j,k}:=\oslash(\partial^j
\partial_w^k Q_\e),  
\]
where we reduce the complexity of a holomorphic function of two complex
variables to that of a sequence of holomorphic functions of a single variable.
From this perpective, the functions $Q_{j,k}$ act as coordinates of the manifold
$\Sigma/\C$, and at times we refer to them as the \emph{jet hierarchy}.
A natural question then is how to determine when our coordinates come from a
point of $\Sigma/\C$.

\begin{problem}
Given a doubly indexed sequence $Q_{j,k}$, $j,k=0,1,2,\ldots$, of
holomorphic functions on $\D_\e$, can we tell when it comes from a
Dirichlet symbol $Q_\e=\mathscr{Q}_\e[\Gamma_\psi]$, for $\psi\in\Sigma$?
\end{problem}

To answer this problem, we need to introduce some concepts.
 
\begin{defn}
A sequence $\{Q_{j,k}\}_{j,k}$ of holomorphic functions on $\D_\e$ is
\emph{symmetric} if $Q_{j,k}=Q_{k,j}$ holds for $j,k=0,1,2,\ldots$.
\end{defn}

\begin{defn}
A sequence $\{Q_{j,k}\}_{j,k}$ of holomorphic functions on $\D_\e$ form a
\emph{consistent system} if
\[
Q_{j,k}'(z)=Q_{j+1,k}(z)+Q_{j,k+1}(z),\qquad z\in\D_\e,
\]
holds for $j,k=0,1,2,\ldots$.
\end{defn}

The consistent system property results from the bivariate chain rule:
\[
\partial_z\oslash(\partial_z^j\partial_w^k Q_\e)(z)=
\oslash\big((\partial_z+\partial_w)\partial_z^j\partial_w^k Q_\e\big)(z)
=\oslash(\partial_z^{j+1}\partial_w^k Q_\e\big)(z)+
\oslash(\partial_z^{j}\partial_w^{k+1} Q_\e\big)(z).
\]
  
\begin{defn}
A symmetric sequence $\{Q_{j,k}\}_{j,k}$ of holomorphic functions on $\D_\e$
is said to solve the \emph{jet nonlinear wave equation} if
\[
Q_{k+2,0}-(k+2)Q_{k+1,1}-(k+1)\sum_{l=0}^{k}\binom{k}{l}Q_{l,1}(z)Q_{k+1-l,0}=0  
\]
holds for each $k=0,1,2,\ldots$. 
\end{defn}

\begin{defn}
A symmetric sequence $\{Q_{j,k}\}_{j,k}$ is said to solve the \emph{jet Riccati
equation} if 
\[
Q_{1,0}'(z)-3Q_{1,1}(z)-\big(Q_{1,0}(z)\big)^2 =0, 
\]
and if, for $k=1,2,3,\ldots$, we have
\begin{equation}
Q_{k,1}'(z)-\frac{k+3}{k+1}Q_{k+1,1}(z)
+\sum_{l=1}^{k-1}
\binom{k}{l}Q_{k-l,1}(z)Q_{l,1}(z)=0.
\end{equation}
\end{defn}

We know from the symmetry of $Q_\e=\mathscr{Q}_\e[\Gamma_\psi]$ that if
\[
Q_{j,k}=\oslash(\partial_z^j\partial_w^k Q_\e),
\qquad j,k=0,1,2,\ldots,  
\]  
then $\{Q_{j,k}\}_{j,k}$ constitutes a symmetric and consistent system
of holomorphic functions on $\D_\e$, which solves the jet nonlinear wave
equation as well as the jet Riccati equation, by Corollaries \ref{cor:local-1}
and \ref{cor:Riccati}. Moreover, the property $(a)$ of
Theorem \ref{Thm_A} ensures that
\[
|Q_\e(z,w)|=\Ordo(|zw|^{-1})
\]
as $z$ or $w$ approaches infinity, or both. 
This then leads to the decay condition
\begin{equation}
|Q_{j,k}(z)|=\Ordo(|z|^{-j-k-2})
\label{eq:decay-j+k}
\end{equation}
as $|z|\to+\infty$. What about putting all of this together? For a given
$\alpha\in\D_\e$, we might form the sum
the function
\begin{equation}
Q^{\langle\alpha\rangle}(z,w):=\sum_{j,k=0}^{+\infty}\frac{Q_{j,k}(\alpha)}{j!k!}
(z-\alpha)^j(w-\alpha)^k,  
\label{eq:Qalpha}
\end{equation}
which should hopefully represent the tentative function $Q_\e(z,w)$.
Since the domain of convergence is given in terms of the domain of
holomorphicity, the series is expected to converge for
$z,w\in\D(\alpha,|\alpha|-1)$. In particular, we would need a condition of
the form
\begin{equation}
\frac{|Q_{j,k}(\alpha)|}{j!k!}=\Ordo\big((1+\epsilon)^{j+k}|\alpha|^{-j-k}\big),
\label{eq:growth-1.001}
\end{equation}
for each $\alpha\in\D_\e$ and each $\epsilon>0$. We note that by the
consistency of the system $\{Q_{j,k}\}_{j,k}$, the function defined by
\eqref{eq:Qalpha} actually does not depend on the parameter $\alpha$. Indeed,
we then calculate that
\begin{multline}
\partial_\alpha Q^{\langle\alpha\rangle}(z,w)=
\sum_{j,k=0}^{+\infty}\frac{Q_{j,k}'(\alpha)}{j!k!}
(z-\alpha)^j(w-\alpha)^k
\\
-\sum_{j,k=0}^{+\infty}\frac{Q_{j,k}(\alpha)}{j!k!}
\Big(j (z-\alpha)^{j-1}(w-\alpha)^k+k(z-\alpha)^j(w-\alpha)^{k-1}\Big)=0.
\label{eq:Qalpha-2}
\end{multline}
Also, clearly, $Q^{\langle\alpha\rangle}(z,z)=Q_{0,0}(z)$.
By letting $\alpha$ approach infinity in any given direction, while using the
fact that the expression is constant in $\alpha$, we realize that
the formula \eqref{eq:Qalpha} defines a holomorphic function of two variables
in the product of half-planes $\mathbb{H}_\theta\times\mathbb{H}_\theta\subset
\C^2$, where we use the notation
\[
\mathbb{H}_\theta:=\big\{z\in\C:\,\re(\e^{\imag\theta}z)>1\big\},
\]
for real values of $\theta$. While this defines an interesting exercise in
the Hartogs extension phenomenon, it is not enough to know that a function
is holomorphic the domain
\[
\bigcup_{\theta\in\R} (\mathbb{H}_\theta\times\mathbb{H}_\theta)\subset\C^2
\]
to be able to deduce that the function extends holomorphically to $\D_\e^2$.
An interesting counterexample is the function $q(z,w)=(z+w)^{-2}$, which
has the appropriate diagonal decay
$q(z,z)=(2z)^{-2}$, is holomorphic and bounded in each
$\mathbb{H}_\theta\times\mathbb{H}_\theta$ for $\theta\in\R$, but has a
singularity along the alternative diagonal $z+w=0$.
This means that we should find another expansion than Taylor's formula
\eqref{eq:Qalpha} to define $Q_\e$ out of the sequence of functions $Q_{j,k}$
for $j,k=0,1,2,\ldots$. After all, we know that $Q_\e$ has a bivariate
Laurent expansion (or bivariate Taylor series at infinity)
\begin{equation}
Q_\e(z,w)=\sum_{m,n=1}^{+\infty}\gamma_{m,n}z^{-m}w^{-n}, \qquad (z,w)\in\D_\e^2,
\label{eq:Laurent-1}
\end{equation}
so the natural condition on the coefficients $\gamma_{m,n}$ which is
equivalent to convergence of the Laurent series in $\D_\e^2$ reads
\begin{equation}
|\gamma_{m,n}|=\Ordo(1+\epsilon)^{m+n}
\label{eq:importantgrowth-1}
\end{equation}
for each $\epsilon>0$.
Of course we know much more about these so-called \emph{Grunsky coefficients}
$\gamma_{m,n}$,
but for now we are only interested in this minimal condition
\eqref{eq:importantgrowth-1} which makes $Q_\e$ holomorphic on $\D_\e^2$. 
Differentiation of \eqref{eq:Laurent-1} gives that
\begin{equation}
\partial_z^j\partial_w^k Q_\e(z,w)=
(-1)^{j+k}\sum_{m,n=1}^{+\infty}(m)_j(n)_k \gamma_{m,n}z^{-m-j}w^{-n-k},
\qquad (z,w)\in\D_\e^2,
\label{eq:Laurent-2}
\end{equation}
where we recall the Pochhammer notation $(x)_j=x(x+1)\cdots(x+j-1)$. Diagonal
restriction in \eqref{eq:Laurent-2} then gives that
\begin{equation}
Q_{j,k}(z)=\oslash\big(\partial_z^j\partial_w^k Q_\e\big)(z)=
(-1)^{j+k}\sum_{m,n=1}^{+\infty}(m)_j(n)_k \gamma_{m,n}z^{-m-n-j-k},
\qquad z\in\D_\e,
\label{eq:Laurent-3}
\end{equation}
and consequently, we have that
\begin{equation}
\frac{(-1)^{j+k}}{j!k!} z^{j+k+2}Q_{j,k}(z)=
\sum_{m,n=1}^{+\infty}\frac{(m)_j(n)_k}{j!k!}\gamma_{m,n}z^{-m-n+2},
\qquad z\in\D_\e.
\label{eq:Laurent-4}
\end{equation}
We now introduce the finite difference operators $\diff_{\langle j\rangle}$ and
$\diff_{\langle k\rangle}$, acting on functions $f(j,k)$ via
\[
\diff_{\langle j\rangle}f(j,k):=f(j+1,k)-f(j,k),\qquad
\diff_{\langle k\rangle}f(j,k):=f(j,k+1)-f(j,k),  
\]
and note that for $a,b=0,1,2,\ldots$,
\[
\diff^a_{\langle j\rangle}\frac{(m)_j}{j!}=\frac{(m)_j(m-a)_a}{(j+a)!},
\qquad
\diff^b_{\langle k\rangle}\frac{(n)_k}{k!}=\frac{(n)_k(n-b)_b}{(k+b)!}.
\]
We apply these operations one ofter the other to both sides of
the equality \eqref{eq:Laurent-4}, which gives that
\begin{multline}
\diff^a_{\langle j\rangle}\diff^b_{\langle k\rangle}
\bigg(\frac{(-1)^{j+k}}{j!k!} z^{j+k+2}Q_{j,k}(z)\bigg)
\\
=\sum_{m=a+1}^{+\infty}\sum_{n=b+1}^{+\infty}
\frac{(m)_j(m-a)_a(n)_k(n-b)_b}{(j+a)!(k+b)!}\gamma_{m,n}z^{-m-n+2},
\qquad z\in\D_\e.
\label{eq:Laurent-4'}
\end{multline}
Taking the limit as $|z|\to+\infty$, we obtain, after simplification, that
\begin{equation}
\gamma_{m,n}=
\lim_{|z|\to+\infty}  \diff^{m-1}_{\langle j\rangle}\diff^{n-1}_{\langle k\rangle}
\bigg(\frac{(-1)^{j+k}}{j!k!} z^{j+k+m+n}Q_{j,k}(z)\bigg),
\label{eq:Laurent-5}
\end{equation}
For a genuine Grunsky kernel this limit is independent of $(j,k)$.
The theorem below shows that, once the jet equations hold, this independence
is automatic. 

\begin{thm}
A doubly indexed sequence $\{Q_{j,k}\}_{j,k}$, where $j,k=0,1,2,\ldots$,
consisting of holomorphic functions on $\D_\e$, constitute the coordinates
of a point $\psi+C\in\Sigma/\C$ if and only if the following conditions
$(a),(b),(c),(d)$ or, alternatively, $(a),(b),(c'),(d)$, are all fulfilled,
where

\noindent{$(a)$}  Each $Q_{j,k}$ has the decay
\begin{equation}
|Q_{j,k}(z)|=\Ordo(|z|^{-j-k-2})
\end{equation}
as $|z|\to+\infty$.

\noindent {$(b)$} The sequence is symmetric and forms a consistent system.

\noindent{$(c)$} The sequence solves the jet nonlinear wave equation.

\noindent{$(c')$} The sequence solves the jet Riccati equation.

\noindent{$(d)$} The coefficients $\gamma_{m,n}$ obtained from
\eqref{eq:Laurent-5} by evaluating the finite differences at $(j,k)=(0,0)$
satisfy the growth condition \eqref{eq:importantgrowth-1} for each
$\epsilon>0$.

\label{thm:characterizationofsymbols}
\end{thm}

\begin{proof}
Given $\psi\in\Sigma$, knowing $\psi+C\in\Sigma/\C$ is equivalent to
knowing the Dirichlet symbol $Q_\e=\mathscr{Q}_\e[\Gamma_\psi]$ as a
holomorphic function on $\D_\e^2$, by Proposition \ref{prop:5equiv}. 
On the other hand, a function $Q_\e$ is such a Dirichlet symbol if and only
if it has properties $(a)$ and $(b)$ of Theorem \ref{Thm_A}. 

{\sc Step I}: \emph{The forward implication}.
We now show how the point $\psi+C\in\Sigma/\C$ gives rise to a coordinate
sequence $\{Q_{j,k}\}_{j,k}$ with the listed properties $(a),(b),(c),(c'),(d)$. 
We first form $Q_\e=\mathscr{Q}_\e[\Gamma_\psi]$ and define
\[
Q_{j,k}(z):=\oslash(\partial_z^j\partial_w^k Q_\e)(z),\qquad z\in\D_\e.
\]
By the symmetry of $Q_\e$, the sequence is symmetric, and it also has
to form a consistent system, by the bivariate chain rule, which settles
$(b)$. Since property $(b)$ of Theorem \ref{Thm_A} demands that $Q_\e$
should solve the nonlinear wave equation, the localized version of the
wave equation follows as well, which is the assertion of Corollary
\ref{cor:local-1}, and hence property $(c)$ follows.
In the same vein, by Corollary \ref{cor:Riccati}, property $(c')$ follows as
well. 
Next, by the property $(a)$ of Theorem \ref{Thm_A}, we obtain
the claimed property $(a)$ listed above.
We have also seen that $Q_\e$ has the Laurent expansion \eqref{eq:Laurent-1},
where the coefficients $\gamma_{m,n}$ enjoy the property \eqref{eq:Laurent-5},
and that they must have the growth bound \eqref{eq:importantgrowth-1}.
From this, property $(d)$ follows. The implication that we get the properties
$(a),(b),(c),(c'),(d)$ from the given point $\psi+C\in\Sigma/\C$ is now
finished.

{\sc Step II}: \emph{The reverse implication}.
We turn to the reverse implication, that the properties $(a),(b),(c),(d)$,
or alternatively, the properties $(a),(b),(c'),(d)$, determine
a unique point $\psi+C\in\Sigma/\C$. In view of property $(d)$, we have
coefficients $\gamma_{m,n}$ extracted at $(j,k)=(0,0)$, which allows us to form the bivariate
Laurent series
\[
Q_\circledast(z,w):=\sum_{m,n=1}^{+\infty}\gamma_{m,n}z^{-m}w^{-n},
\qquad (z,w)\in\D_\e,
\]
which converges and defines a holomorphic function on $\D_\e^2$. By the
symmetry $Q_{j,k}=Q_{k,j}$ we see that $\gamma_{m,n}=\gamma_{n,m}$ holds as well,
and hence the function $Q_\circledast$ is symmetric:
$Q_\circledast(z,w)=Q_\circledast(w,z)$.
Clearly,
$Q_\circledast(z,\infty)=Q_\circledast(\infty,w)=0$ holds for all $z,w\in\D_\e$,
so that property $(a)$ of Theorem \ref{Thm_A} is fulfilled with
$Q_\e=Q_\circledast$.
We should connect this function $Q_\circledast$ with our coordinates
$\{Q_{j,k}\}_{j,k}$, and show that with $Q_\e=Q_\circledast$, the nonlinear
wave equation of property $(b)$ of Theorem \ref{Thm_A} is fulfilled as well.
Theorem \ref{Thm_A} will then guarantee the existence of a function
$\psi\in \Sigma$ such that $Q_\circledast=Q_\e=\mathscr{Q}_\e[\Gamma_\psi]$
holds, and by Proposition \ref{prop:5equiv}, this $\psi$ is unique up to an
additive constant, corresponding to $\psi+C\in\Sigma/\C$. 
As a first step to connect the function $Q_\circledast$ with the coordinate
functions $Q_{j,k}$, we introduce the functions $Q^\circledast_{j,k}$ via the
formula
\[
Q^\circledast_{j,k}(z)=\oslash(\partial_z^j\partial_w^k Q_\circledast)(z)
\]
for $j,k=0,1,2,\ldots$. We would expect to have that $Q^\circledast_{j,k}=Q_{j,k}$,
but this needs to be demonstrated. 

{\sc Step II$a$}: \emph{Identification of the reconstructed diagonal jets}.
We first observe that the symmetry and consistency assumptions, together
with either of the jet differential systems, determine the entire jet
hierarchy uniquely from $Q_{0,0}$. Indeed, at total order one, symmetry
and consistency give $Q_{1,0}=Q_{0,1}=\frac12 Q_{0,0}'$.
Suppose that the jets of total order less than $n$ are known, where
$n\ge2$. Put $X_j=Q_{j,n-j}$, $0\le j\le n$. Consistency gives
\[
X_j+X_{j+1}=Q_{j,n-1-j}',\qquad 0\le j<n.
\]
The homogeneous system has the one-dimensional solution space
$X_j=(-1)^jX_0$. In the jet nonlinear wave formulation, the equation
with index $n-2$ determines the additional combination
$X_n-nX_{n-1}$ from lower-order jets. Its value on the above
homogeneous solution is $(n+1)(-1)^nX_0$, which vanishes only for
$X_0=0$. In the jet Riccati formulation, the equation with index
$n-2$ determines $X_{n-1}$ for $n\ge3$, while the initial Riccati
equation determines $X_1$ for $n=2$. In either case, the jets at
total order $n$ are uniquely determined.

Now put $q(z)=Q_{0,0}(z)$. Since $q(z)=\mathrm O(z^{-2})$ at infinity,
there is a holomorphic function $\psi$ on a sufficiently small
neighborhood of infinity, normalized by
\[
\psi'(z)=\exp q(z),\qquad \psi(z)=z+\mathrm O(1).
\]
After shrinking this neighborhood if necessary, the divided difference
$(\psi(z)-\psi(w))/(z-w)$ is nonzero for all $z,w$ sufficiently large.
Consequently
\[
Q^{\mathrm{loc}}(z,w)=
\log\frac{\psi(z)-\psi(w)}{z-w}
\]
is holomorphic on a product neighborhood of infinity and has a
convergent double Laurent expansion there. Its diagonal equals $q$.
By the local nonlinear wave and Riccati identities, its diagonal jets
satisfy both jet differential systems, as well as symmetry and
consistency. The uniqueness just established therefore gives
\[
Q_{j,k}(z)=\oslash\bigl(\partial_z^j\partial_w^k
Q^{\mathrm{loc}}\bigr)(z)
\]
for all $j,k\ge0$ and $z$ sufficiently large. The identity theorem
extends each equality to $\D_\e$.

Write the double Laurent coefficients of $Q^{\mathrm{loc}}$ as
$\gamma^{\mathrm{loc}}_{m,n}$. Applying the already established
finite-difference extraction formula \eqref{eq:Laurent-5} to the
local kernel shows that the limits exist for every starting pair
$(j,k)$, agree with $\gamma^{\mathrm{loc}}_{m,n}$, and are independent
of that pair. In particular, the coefficients stipulated in (d),
extracted at $(0,0)$, equal $\gamma^{\mathrm{loc}}_{m,n}$.
The growth condition in (d) now extends this local Laurent series to
$Q_\circledast$ on $\D_\e^2$. Its diagonal jets agree with those of
$Q^{\mathrm{loc}}$ near infinity and hence with the given $Q_{j,k}$
throughout $\D_\e$. This proves the required identification.

{Step II$b$}: \emph{Under condition $(c)$, the function $Q_\circledast$
solves the nonlinear wave equation}. We need to show that the function
$Q_\e=Q_\circledast$ solves the nonlinear wave equation of property $(b)$
in Theorem \ref{Thm_A}. But this is a consequence of property $(c)$ and
of Taylor's formula, as implemented in Subsection \ref{ss:diagonal-nlweq},
given the symmetry of $Q_{j,k}$ stated under $(b)$ and the already established
fact that $Q_{j,k}=Q^\circledast_{j,k}
=\oslash(\partial_z^j\partial_w^k Q_\circledast)$. By the nonlinear wave equation
characterization of Dirichlet symbols of Theorem \ref{Thm_A}, then, the fact
that the function $Q_\e=Q_\circledast$ has the properties $(a)$ and $(b)$
of Theorem \ref{Thm_A} ensures the existence of a normalized conformal mapping
$\psi\in\Sigma$ with 
$Q_\e=Q_\circledast=\mathscr{Q}_\e[\Gamma_\psi]$. Finally, by Proposition
\ref{prop:5equiv}, the mapping $\psi$ is unique up to additive constants,
that is, $\psi+C\in\Sigma/\C$ is unique.

{Step II$c$}: \emph{Under condition $(c')$, the function $Q_\circledast$
solves the Riccati equation}. We need to check that the function
$q=Q_\e=Q_\circledast$ has the properties $(a)$ and $(b)$ of Theorem
\ref{thm:Riccati}. We already checked that property $(a)$ of Theorem
\ref{thm:Riccati} holds, since it forms a part of property $(a)$ of Theorem
\ref{Thm_A}, which we already discussed. By the Taylor expansion method of
Subsection \ref{ss:Riccati-Aharonov}, it follows from the symmetry of property
$(a)$ and the jet Riccati equation property of $(c')$ that $q=Q_\e=Q_\circledast$
solves the Riccati equation of Theorem \ref{thm:Riccati}$(b)$, given that we
already know that $Q^{\circledast}_{j,k}=
\oslash(\partial_z^j\partial_w^k Q_\circledast)
=Q_{j,k}$. By the Riccati equation characterization of the Dirichlet
symbols of Theorem \ref{thm:Riccati}, then, the fact
that the function $q=Q_\e=Q_\circledast$ has the properties $(a)$ and $(b)$
of Theorem \ref{Thm_A} ensures the existence of a normalized conformal
mapping $\psi\in\Sigma$ with $q=Q_\e=Q_\circledast=\mathscr{Q}_\e[\Gamma_\psi]$.
Finally, by Proposition \ref{prop:5equiv}, the mapping $\psi$ is unique up
to additive constants, that is, $\psi+C\in\Sigma/\C$ is unique.

The proof of the theorem is complete.
\end{proof}

\begin{rem}
The point with Theorem \ref{thm:characterizationofsymbols} is that it
supplies us with a constructive way to build elements of the space
$\Sigma/\C$.
We then use the relations of Corollary \ref{cor:offdiag-3}, for $j>k$,
\begin{equation}
Q_{j,k}(z)=Q_{k,j}(z)=
\frac{j-k}{2}\sum_{l=k}^{\lfloor \frac{j+k}{2} \rfloor}
\frac{(-1)^{l+k}}{j-l}\binom{j-l}{l-k}\,Q_{l,l}^{(j+k-2l)}(z),\qquad z\in\D_\e.
\label{eq:Qjk-001}
\end{equation}
and the localized Riccati equation, which gives first
\begin{equation}
Q_{1,1}(z)=\frac13 Q_{1,0}'(z)-\frac13(Q_{1,0}(z))^2
\label{eq:Qjk-002}
\end{equation}
and then, for $j=1,2,3,\ldots$, 
\begin{multline}
Q_{j+1,j+1}(z)=\frac{2j+1}{2j+3}
\sum_{l=1}^{j}(-1)^{j+l-1}C_{l,j}Q^{(2j+2-2l)}_{l,l}(z)
\\
+(-1)^j\frac{2j+1}{2j+3}
\sum_{l=1}^{2j-1}\binom{2j}{l}Q_{2j-l,1}(z)Q_{l,1}(z),
\label{eq:Qjk-003}
\end{multline}
where the coefficients $C_{l,j}$ are given by \eqref{eq:prep-1.1''}.  
Starting with an input function $Q_{0,0}$, the equation \eqref{eq:Qjk-001}
gives us $Q_{1,0}$ and $Q_{0,1}$, and then \eqref{eq:Qjk-002} gives us
$Q_{1,1}$. 
We then return to the relation \eqref{eq:Qjk-001} to give us $Q_{j,k}$
for $j+k\le3$. Next, we use \eqref{eq:Qjk-003} for $j=1$ to get $Q_{2,2}$.
After that, we use \eqref{eq:Qjk-001} to get $Q_{j,k}$ for $j+k\le5$.
Following this algorithm, we obtain successively all the $Q_{j,k}$ for
$j,k=0,1,2,\ldots$, given the first function $Q_{0,0}$.
From this point of view, the only condition in Theorem
\ref{thm:characterizationofsymbols}
which says that we started with $Q_{0,0}=\log\psi'$
for a $\psi\in\Sigma$ is, apart from the decay condition $(a)$, the condition
$(d)$ which actually only asks for the existence of the limits
\eqref{eq:Laurent-5} irrespectively of the point $(j,k)$,
combined with the mild growth bound \eqref{eq:importantgrowth-1}. 
\end{rem}

\section{Notions of asymptotic variance}
\label{sec:asymp}

\subsection{Overview}
In this section, we study the notion of asymptotic variance
introduced in the context of dynamics and thermodynamical formalism by
Curtis McMullen \cite{mcmullen}. For the unit disk $\D$, we explained how
this works for a given holomorphic function in Subsection \ref{ss:asymp}.
Here, we shall work with the exterior disk $\D_\e$ instead, so the notions
need some slight modification. Perhaps more interestingly, we also introduce
a \emph{Schwarz\-ian asymptotic variance}. The Schwarz\-ian derivative measures
the local deviation of a conformal
mapping from a M\"obius mapping. Its asymptotic variance therefore provides
a corresponding averaged measure.

\subsection{The exterior disk asymptotic variances
 $\sigma_\e^2,\sigma_{\e,1}^2,\sigma_{\e,2}^2$}
In analogy with the notions on the disk $\D$ introduced in
Subsection \ref{ss:asymp}, we define the following asymptotic variances of
a holomorphic function $g:\D_\e\to\C$ which is bounded at infinity:
\begin{equation}
\sigma_\e(g)^2:=\limsup_{R\to1^+}\frac{1}{\log\frac{1}{R^2-1}}
\int_{\T}|g(R\zeta)|^2\diff s(\zeta),
\end{equation}
and
\begin{equation}
\sigma_{\e,1}(g)^2:=\limsup_{R\to1^+}\frac{1}{\log\frac{1}{R^2-1}}
\int_{\D_\e}|g(Rz)|^2(1-|z|^{-2})|z|^{-4}\dA(z),
\end{equation}
while
\begin{equation}
\sigma_{\e,2}(g)^2:=\limsup_{R\to1^+}\frac{1}{\log\frac{1}{R^2-1}}
\int_{\D_\e}|g(Rz)|^2(1-|z|^{-2})^3|z|^{-4}\dA(z),
\end{equation}

The analogue of Proposition \ref{prop:asymp1} reads as follows.

\begin{prop}
\label{prop:asymp2}
Given a holomorphic function $g:\,\D_\e\to\C$ which is bounded at infinity,
and if $f(z)=g(1/z)$ which is then holomorphic in $\D$, we have the identity
of asymptotic variances
\[
\sigma_\e(g)^2=\sigma(f)^2=\sigma_{\e,1}(g')^2=\sigma_1(f')^2=
\frac16\sigma_{\e,2}(g'')^2=\frac16\sigma_2(f'')^2.
\]
\end{prop}

\begin{proof}
Write
\[
g(z)=\sum_{n=0}^{+\infty}a_nz^{-n},\qquad
f(z)=\sum_{n=0}^{+\infty}a_nz^n,
\qquad r=R^{-1}.
\]
By orthogonality and radial integration, the circular mean of $g$,
the weighted integral for $g'$, and one-sixth of the weighted integral
for $g''$ in the definitions above are, respectively,
\[
\begin{gathered}
\sum_{n=0}^{+\infty}|a_n|^2r^{2n},\\
r^2\sum_{n=1}^{+\infty}
\frac{n^2}{(n+2)(n+3)}|a_n|^2r^{2n},\\
r^4\sum_{n=1}^{+\infty}
\frac{n^2(n+1)^2}{(n+3)_4}|a_n|^2r^{2n}.
\end{gathered}
\]
The two coefficient factors in the latter sums tend to $1$ as
$n\to+\infty$. Contributions from finitely many indices disappear
upon division by $\log\frac1{1-r^2}$. Comparison of the remaining
nonnegative terms therefore gives the same upper limit in all three
cases. Since
\[
\log\frac1{R^2-1}\sim\log\frac1{1-r^2}
\]
as $r\to1^-$, we obtain
\[
\sigma_\e(g)^2=\sigma(f)^2=\sigma_{\e,1}(g')^2
=\frac16\sigma_{\e,2}(g'')^2.
\]
The remaining identities follow from Proposition \ref{prop:asymp1}.
\end{proof}

\subsection{Finite-step variance of arbitrary degree}
\label{ss:aaprox-av-arbdegree}

In the context of the unit disk $\D$ and a given holomorphic function
$f:\D\to\C$, we introduce the \emph{finite-step variance
of order $j\in\Z_{>0}$}, as given by
\begin{equation}
\sigma_{j}^{\langle r\rangle}(f)^2:=\frac{1}{\log\frac{1}{1-r^2}}
\int_\D|f(rz)|^2(1-|z|^2)^{2j-1}\dA(z) 
\label{eq:sigma-j-r}
\end{equation}
so that the asymptotic variance of order $j\in\Z_{>0}$ is defined by the upper limit
\begin{equation}
\sigma_j(f)^2:=\limsup_{r\to1^-}\sigma_{j}^{\langle r\rangle}(f)^2.
\end{equation}
This agrees with the concept as developed in Subsection \ref{ss:asymp} for
$j=1,2$, and extends it to general order $j=1,2,3,\ldots$.
The differentiation rule of Proposition \ref{prop:asymp1} extends to
higher orders. We first state a version under an additional growth
assumption. To this end, Proposition 4.7 from
\cite{IMS} comes in handy.
To formulate it, we need the notation of the classical weighted Bergman
spaces $A^2_\alpha(\D)$ on the disk $\D$ introduced in
\eqref{ss:weightedBergman} for $-1<\alpha<+\infty$, with limit case
$A^2_{-1}(\D)=H^2(\D)$ at the parameter edge $\alpha=-1$.

\begin{prop}
\label{prop:handyestimate-1}
Suppose that $f\in A^2_\alpha(\D)$ for some $\alpha$ with
$-1\le \alpha<+\infty$. For each $n=1,2,3,\ldots$, we then have the
estimate
\begin{equation}
0\le (\alpha+2)_{2n}\|f\|^2_{A^2_\alpha(\D)}-\|f^{(n)}\|^2_{A^2_{\alpha+2n}(\D)}
\le C_1(\alpha)\,n^2(\alpha+2)_{2n}\|f\|^2_{A^2_{\alpha+1}(\D)}  
\end{equation}
for some positive constant $C_1(\alpha)$. 
\end{prop}

\begin{proof}
This is a special instance of Proposition 4.7 in \cite{IMS}.  
\end{proof}

Since, as a matter of definition, the identity
\begin{equation}
\label{eq:sigma-norm-relation}
\sigma_j^{\langle r\rangle}(f)^2=
\frac{\|f_r\|^2_{A^2_{2j-1}(\D)}}{2j\log\frac{1}{1-r^2}}
\end{equation}
holds with $f_r(z)=f(rz)$, the proposition gives corresponding relations
for the finite-step variances.

\begin{cor}
\label{cor:sigma-comparison}
We have that
\begin{equation}
\sigma_j^{\langle r\rangle}(f)^2=\frac{1}{(2j)_{2n}}\sigma_{j+n}^{\la r\ra}(f^{(n)})^2
+\ordo(1),  
\end{equation}
provided that $\sigma_{j+\frac12}^{\langle r\rangle}(f)=\ordo(1)$ as $r\to1^-$. 
\end{cor}

\begin{rem}
Here, the finite-step variance $\sigma^{\langle r\rangle}_{j+\frac12}(f)^2$ is given
by the same formula even though $j+\frac12$ is a half-integer.
The assumption, which amounts to additional growth control on $f$, can
be removed, as in Proposition \ref{prop:asymp1}.
\end{rem}


We may apply the standard estimates of Hilbert space theory, such as the
Cauchy-Schwarz inequality.

\begin{lem}
For holomorphic $f,g:\D\to\C$, we have the estimate
\begin{equation}
\sigma_{m+n}^{\langle r\rangle}(fg)^2\le \sigma_{2m}^{\langle r\rangle}(f^2)\,
\sigma_{2n}^{\langle r\rangle}(g^2),
\end{equation}
provided that $m,n\in\frac{1}{2}\Z_{>0}$ are allowed to be
half-integers with $m+n\in\Z$.
\label{lem:Cauchy-Schwarz-1001}
\end{lem}

\begin{proof}
This follows from the Cauchy-Schwarz inequality.
\end{proof}

There is a version of the estimate of Lemma \ref{lem:Cauchy-Schwarz-1001}
which applies when $m>0$ and $n=0$, with $g=1$. 

\begin{lem}
Suppose $f$ is holomorphic in $\D$ with $f(z)=\Ordo(1-|z|^2)^{-m}$ uniformly
in $\D$, for some $m\in\Z_{>0}$. It then follows that for $j,k=1,2,3,\ldots$, 
\[
j<k\,\,\,\Longrightarrow\,\,\,
\sigma_{jm}^{\langle r\rangle}(f^j)^k\le\sigma_{km}^{\langle r\rangle}(f^k)^j
+\ordo(1)
\]
holds as $r\to1^-$. 
\label{lem:moments}
\end{lem}

\begin{rem}
Since the finite-step variances are normalized squared-amplitude averages,
the inequality corresponds to the usual moment inequality
$\mathbb{E}(|X|^j)^k\le\mathbb{E}(|X|^k)^j$ for $j<k$ from probability
theory. 
\end{rem}

\begin{proof}
The finite-step variance is, as a matter of definition, given by
\begin{equation}
\sigma_{jm}^{\langle r\rangle}(f^j)^2=
\frac{1}{\log\frac{1}{1-r^2}}\int_{\D}|f(rz)|^{2j}(1-|z|^2)^{2jm-1}\diffA(z).
\end{equation}
By H\"older's inequality, we have that
\begin{multline}
\int_\D|f(rz)|^{2j}(1-|z|^2)^{2jm-1}\dA(z)\le\bigg\{
\int_\D|f(rz)|^{2k}(1-r^2|z|^2)^{\frac{k}{j}-1}
(1-|z|^2)^{2km-\frac{k}{j}}\dA(z)\bigg\}^{\frac{j}{k}}
\\
\times\bigg\{\int_\D\frac{\dA(z)}{1-r^2|z|^2}\bigg\}^{1-\frac{j}{k}},
\end{multline}
and by elementary calculus, for $0<\alpha<+\infty$, which we split as
$\alpha=l+\alpha_1$, where $l\in\Z_{\ge0}$ and $0<\alpha_1\le1$, we have
\begin{multline}
0<(1-r^2|z|^2)^{\alpha}-(1-|z|^2)^{\alpha}=(1-r^2)\alpha|z|^2
(1-\theta|z|^2)^{\alpha-1}
\\
\le \alpha\, (1-r^2)(1-r^2|z|^2)^l\,(1-|z|^2)^{\alpha_1-1},
\end{multline}
where $\theta$ denotes a quantity with $r^2<\theta<1$. We apply this with
$\alpha:=\frac{k}{j}-1>0$, and using the given bound
$|f(z)|=\Ordo(1-|z|^2)^{-m}$, we may obtain that
\begin{multline}
\int_\D|f(rz)|^{2k}(1-r^2|z|^2)^{\frac{k}{j}-1}
(1-|z|^2)^{2km-\frac{k}{j}}\dA(z)
\\
=\int_\D|f(rz)|^{2k}(1-|z|^2)^{2km-1}\dA(z)+\Ordo(1),
\end{multline}
as $r\to1^-$. The assertion now follows from this.
\end{proof}

We shall also use the following form of the parallelogram law.

\begin{prop}
Suppose $0<r<1$ and that $a_k,b_k\in\C$ for
$k=1,\ldots,n$, with $a_1\bar b_1+\cdots+a_n\bar b_n=0$. For two
holomorphic functions $f,g$ on $\D$, we have, for $j=1,2,3,\ldots$,
\[
\sum_{k=1}^{n}\sigma_j^{\la r\ra}(a_kf+b_kg)^2=\sum_{k=1}^{n}\big(|a_k|^2
\sigma_j^{\la r\ra}(f)^2+|b_k|^2\sigma_j^{\la r\ra}(g)^2\big).
\]  
\label{prop:parallel}
\end{prop}

\begin{proof}
This follows by expanding the inner product.    
\end{proof}  
\subsection{Weighted Bergman spaces on the bidisk  $\D^2$}

We need a specific family of weighted Bergman $L^2$ spaces on the bidisk.  
For parameters $\alpha,\alpha',
\in\R$ with $-1<\alpha,\alpha'<+\infty$,
let $A^2_{\alpha,\alpha'}(\D^2)$ denote the weighted Bergman space on $\D^2$
that consists of all holomorphic functions
$f:\mathbb{\D}^2\rightarrow\mathbb{C}$ such that
\[
\|f\|^2_{A^2_{\alpha,\alpha'}(\mathbb{\D}^2)}:=\int_{\D}\int_{\D} |f(z,w)|^2
\dA_\alpha(z)\dA_{\alpha'}(w)<+\infty,
\]
where we recall the notation $\dA_\alpha(z)=(\alpha+1)(1-|z|^2)^\alpha$
and the same for $\dA_{\alpha'}$.
We also consider the limit case $\alpha=-1$, and write $A^2_{-1,\alpha'}(\D^2)$
for the Hilbert space of holomorphic functions $f:\D^2\raro\mathbb{C}$ such
that
\[
\|f\|_{A^2_{-1,\alpha'}(\mathbb{\D}^2)}:=\limsup_{r\to1^-}\int_{\D}\int_{\mathbb{T}}
|f(rz,w)|^2\diff s(z)\dA_{\alpha'}(w)<+\infty.
\]
The limit case $\alpha=\alpha'=-1$ also makes sense as the Hardy space
$H^2(\D^2)$:
\[
\|f\|_{A^2_{-1,-1;0}(\mathbb{\D}^2)}:=\limsup_{r\to1^-}\int_{\D}\int_{\mathbb{T}}
|f(rz,rw)|^2\diff s(z)\diff s(w)<+\infty.
\]
In \cite{IMS, AMS, HHSSAS}, Hedenmalm, Shimorin, and partially Sola
developed an expansion of the norm of functions in
$A^2_{\alpha,0;\beta}(\mathbb{\D}^2)$ was developed in terms of a diagonal
norm expansion, somewhat analogous to the Taylor expansion along the
diagonal.
This technique is quite useful, and in \cite{IMS, AMS, HHSSAS},
it was used to estimate from above the so-called
\emph{universal integral means spectrum of conformal mappings} \cite{AMS}.
Here, we will instead use it to estimate our various notions of asymptotic
variance.
We recall the basic result of \cite{IMS}, and further developed in
\cite{HHSSAS} (the roles of the variables $(z,w)$ are reversed).

\begin{thm} {\rm (Hedenmalm, Shimorin, Sola)}
If $f\in A^2_{\alpha,\alpha'}(\mathbb{\D}^2)$, then we have the following norm
expansion
\begin{equation}\label{diagonal restriction}
\|f\|^2_{A^2_{\alpha,\alpha'}(\D^2)}=
\sum_{n=0}^{+\infty}\frac{1}{\sigma(\alpha,\alpha',n)}
\Big\|\sum_{k=0}^{n}a_{k,n}\partial_z^{n-k}
\oslash[\partial_w^kf]\Big\|^2_{A^2_{\alpha+\alpha'+2n+2}(\D)},
\end{equation}
where
the constants are 
\[
\frac{1}{\sigma(\alpha,\alpha',n)}=
\int_{\D^2}|z-w|^{2n}\dA_\alpha(z)\dA_{\alpha'}(w)=  
\frac{n!(\alpha+\alpha'+n+3)_n}{(\alpha+2)_n(\alpha'+2)_n}. 
\]
and 
\[
a_{k,n}=\frac{(-1)^{n-k}}{k!(n-k)!}
\frac{(\alpha'+k+2)_{n-k}}{(\alpha+\alpha'+n+k+3)_{n-k}}.
\]
\label{thm:HS-1}
\end{thm}

\begin{proof}
This follows from Theorem 2.8 of \cite{HHSSAS}, because the generalized
hypergeometric function ${}_3F_2$ reduces to the classical hypergeometric
function ${}_2F_1$ when a top entry agrees with a bottom entry, and the
classical hypergometric function at the point $z=1$ is given by the Gauss
summation theorem. 
\end{proof}

As explained in \cite{HHSSAS}, Theorem \ref{thm:HS-1} applies also in the
limit cases when $\alpha=-1$ or $\alpha'=-1$, or both.  
If we plug in $(\alpha,\alpha')=(-1,0)$,
we obtain the norm expansion
\begin{equation}
\label{estimate:1}
\|f\|^2_{A^2_{-1,0}(\D^2)}=\sum_{n=0}^{+\infty}\frac{(2n+1)!}{((n+1)!)^2}
\bigg\Vert  \sum_{k=0}^n \frac{(-1)^{n-k}(k+2)_{n-k}}{k!(n-k)!(n+k+2)_{n-k}}
\partial_z^{n-k}\oslash[\partial_w^{k}f] \bigg\Vert^2_{A^2_{2n+1}(\D)}.
\end{equation}

\subsection{A basic estimate}

It is because the diagonal norm expansion was developed in the context of
the bidisk $\D^2$ in \cite{IMS, AMS, HHSSAS} that we will return to working
with the Grunsky operator $\Gamma_\vp$ for $\vp\in\mathscr{S}$. 
The following lemma is known.

\begin{lem}
Let $\varphi\in \cls$ and let $Q=\mathscr{Q}[\Gamma_{\varphi}]$. Then,
for $l'=0,1,2,\ldots$, we have
\begin{equation*}
\frac{1}{(2l')!}\int_{\D} \lvert \partial_w^{l'+1} Q(z,w) \rvert^2
(1-|w|^2)^{2l'}\dA (w)
\leq 
\int_{\D} \lvert \partial_w Q(z,w) \rvert^2\dA (w)\leq \log
\frac{1}{1-|z|^2},\quad z\in \D.
\end{equation*}
\label{lem:basic-999}
\end{lem}

\begin{proof}
The left-hand side inequality is a consequence of Proposition
\ref{prop:handyestimate-1}. As for the right-hand side inequality, we
observe that since
\[
Q(z,w)=Q(w,z)=\mathscr{Q}[\Gamma_\vp](w,z)=
zw\la \Gamma_{\vp}(\bar{s}_z),  s_w \ra_\D,
\]
it follows that
\[
\partial_wQ(z,w)=z\la \Gamma_{\vp}(\bar{s}_z), \Berg_w \ra_\D
= z\Gamma_{\vp}(\bar{s}_z)(w),
\]
where $\Berg_z(\xi)=(1-\xi\bar{z})^{-2}$ is the Bergman kernel of $A^2(\D)$.
Consequently, since $\Gamma_\vp$ is contractive, we have
\[
\int_{\D} \lvert \partial_w Q(z,w) \rvert^2\dA (w)=
|z|^2\|\Gamma_{\vp}(\bar{s}_z)\|_{A^2(\D)}^2\leq |z|^2
\|{s}_z\|_{A^2(\D)}^2=\log \frac{1}{1-|z|^2},\qquad z\in \D.
\]
This completes the proof.
\end{proof}

Let $0<r<1$. It then follows from Lemma \ref{lem:basic-999} that for
$l'=0,1,2,\ldots$, 
\[
\int_{\D}\lvert
(\partial_w^{l'+1}Q)(rz,w) \rvert^2(1-|w|^2)^{2l'}\dA(w)\leq(2l')!
\,\log\frac{1}{1-r^2},\qquad z\in\D.
\]
By the fact that the dilation $f\mapsto f_r$, $f_r(z):=f(rz)$, acts
contractively on the weighted Bergman spaces $A^2_{2l'}(\D)$, we conclude
from this that
\[
\int_{\D} \lvert(\partial_w^{l'+1}Q)(rz,rw) \rvert^2(1-|w|^2)^{2l'}\dA (w) 
\le \log\frac{1}{1-r^2},\qquad z\in\D.
\]
Finally, 
by integration over the circle $\T$, we find that
\begin{equation}
\label{inequality:1}
\int_{\mathbb{T}}\int_{\D} \big\lvert
(\partial_w^{l'+1}Q)(rz,rw)\big\rvert^2(1-|w|^2)^{2l'}\dA (w)
\diff s(z)\leq 
(2l')!\,\log \frac{1}{1-r^2}.
\end{equation}
In terms of the doubly dilated function $Q_r(z,w):=Q(rz,rw)$, this amounts to
having 
\begin{equation}
\label{inequality:2}
\int_{\mathbb{T}}\int_{\D} \big\lvert
\partial_w^{l'+1}Q_r(z,w)\big\rvert^2(1-|w|^2)^{2l'}\dA (w)
\diff s(z)\leq 
(2l')!\,r^{2l'+2}\log \frac{1}{1-r^2}.
\end{equation}
Next, we may apply Proposition \ref{prop:handyestimate-1} with respect to
the varable $z$ for $\alpha=-1$ to the left-hand side expression, which
gives that for $l=1,2,3\ldots$ and $l'=0,1,2,\ldots$, 
\begin{multline}
\label{inequality:3}
\int_{\mathbb{\D}}\int_{\D} \big\lvert
\partial_z^{l}\partial_w^{l'+1}Q_r(z,w)\big\rvert^2(1-|z|^2)^{2l-1}(1-|w|^2)^{2l'}
\dA (z)
\dA(w)
\\
\leq 
(2l-1)!(2l')!\,r^{2l'+2}\log \frac{1}{1-r^2}.
\end{multline}

\begin{rem}
Starting with a single estimate, namely the estimate of Lemma
\ref{lem:basic-999} with $l'=0$, we have produced a doubly indexed
family of estimates expressed by \eqref{inequality:2} and
\eqref{inequality:3}. Moreover, it is possible to view \eqref{inequality:2}
as the special case of \eqref{inequality:3} corresponding to $l=0$. 
The point with all this is that we wish to apply the diagonal norm expansion
of Theorem \ref{thm:HS-1}, and the diagonal norm expansion will be different
for each index pair $(l,l')$, and reveal separate aspects of the
underlying structure. 
\end{rem}

\subsection{Diagonal norm expansion with application to finite-step variances}

We apply the norm expansion of Theorem \ref{thm:HS-1} based on
\eqref{inequality:2} and \eqref{inequality:3} in the condensed form
\begin{equation}
\label{inequality:3'}
\big\lVert
\partial_z^{l}\partial_w^{l'+1}Q_r\big\rVert^2_{A^2_{2l-1,2l'}(\D^2)}
\leq (2l)!(2l'+1)! \,r^{2l'+2}\log \frac{1}{1-r^2},
\end{equation}
where $l,l'=0,1,2,\ldots$, and, as before, $0<r<1$. We apply the norm
expansion of Theorem \ref{thm:HS-1} to the left-hand side of
\eqref{inequality:3'}, which asserts that
\begin{multline}
\label{diagonal restriction2}
\big\|\partial_z^l\partial_w^{l'+1}Q_r\big\|^2_{A^2_{2l-1,2l'}(\D^2)}
\\
=
\sum_{n=0}^{+\infty}\frac{n!(2l+2l'+n+2)_n}{(2l+1)_n(2l'+2)_n}
\Big\|\sum_{k=0}^{n}a_{k,n}(l,l')\,\partial_z^{n-k}
\oslash[\partial_z^l\partial_w^{k+l'+1}Q_r]\Big\|^2_{A^2_{2l+2l'+2n+1}(\D)},
\end{multline}
where
\[
a_{k,n}(l,l'):=\frac{(-1)^{n-k}}{k!(n-k)!}\,
\frac{(2l'+k+2)_{n-k}}{(2l+2l'+n+k+2)_{n-k}}.
\]
It now follows that the bound \eqref{inequality:3'} is equivalent to
\begin{multline}
\label{inequality:3''}  
\sum_{n=0}^{+\infty}\frac{n!(2l+2l'+n+2)_n}{(2l+n)!(2l'+n+1)!}
\Big\|\sum_{k=0}^{n}a_{k,n}(l,l')\,\partial_z^{n-k}
\oslash[\partial_z^l\partial_w^{k+l'+1}Q_r]\Big\|^2_{A^2_{2l+2l'+2n+1}(\D)}
\\
\le r^{2l'+2}\log \frac{1}{1-r^2}.
\end{multline}
In order to make the expression more manageable, we introduce the notation
\begin{equation}
\label{eq:Wvar-1}
\Wvar_{l,l',n}:=
\sum_{k=0}^{n}a_{k,n}(l,l')\,\partial_z^{n-k}
\oslash[\partial_z^l\partial_w^{k+l'+1}Q],
\end{equation}
and observe that \eqref{inequality:3''} then expresses that
\begin{equation}
\label{inequality:3'''}  
\sum_{n=0}^{+\infty}\frac{n!(2l+2l'+n+2)_n}{(2l+n)!(2l'+n+1)!}
r^{2n+2l+2l'+2}\big\|(\Wvar_{l,l',n})_r\big\|^2_{A^2_{2l+2l'+2n+1}(\D)}
\le r^{2l'+2}\log \frac{1}{1-r^2},
\end{equation}
where, as usual, we write $(\Wvar_{l,l',n})_r(z)=\Wvar_{l,l',n}(rz)$.
We formulate this inequality in terms of the finite-step variances of Subsection \ref{ss:aaprox-av-arbdegree}.

\begin{thm}
Suppose $\vp\in\mathscr{S}$ and that $Q=\mathscr{Q}[\Gamma_\vp]$. 
If the functions $\Wvar_{l,l',n}$ are given by \eqref{eq:Wvar-1}, then,
for $l,l'=0,1,2,\ldots$, we have
\begin{equation}
\sum_{n=0}^{+\infty}\frac{n!(2l+2l'+n+2)_{n+1}}{(2l+n)!(2l'+n+1)!}
r^{2n+2l}\sigma_{l+l'+n+1}^{\langle r\rangle}(\Wvar_{l,l',n})^2
\le 1.
\end{equation}
\label{thm:l,l'}
\end{thm}

\begin{proof}
This is an immediate consequence of \eqref{eq:sigma-norm-relation}
and \eqref{inequality:3'''}. In turn, \eqref{eq:sigma-norm-relation}
follows directly from the definition of the finite-step variances. 
\end{proof}

At times, it is useful to apply the bivariate chain rule to get that
\begin{multline}
\partial_z^{n-k}\big(
\oslash (\partial_z^l\partial_w^{k+l'+1}Q)\big)
=\oslash\big((\partial_z+\partial_w)^{n-k}\partial_z^l\partial_w^{k+l'+1}Q\big)
\\
=\sum_{m=0}^{n-k}\binom{n-k}{m}
\oslash\big(\partial_z^{l+n-k-m}\partial_w^{k+l'+m+1}Q\big),
\end{multline}
so that the expression for $\Wvar_{l,l',n}$ becomes
\begin{equation}
\label{eq:Wvar-2}
\Wvar_{l,l',n}=
\sum_{k=0}^{n}\sum_{m=0}^{n-k}\binom{n-k}{m}a_{k,n}(l,l')\,
\oslash\big(\partial_z^{l+n-k-m}\partial_w^{k+l'+m+1}Q\big).
\end{equation}

\subsection{The case $l=l'=0$ and the folding bound}
The instance $l=l'=0$ of Theorem \ref{thm:l,l'} concerns the
nonlinearity $\Nop(\vp)=\vp''/\vp'$ and the asymptotic
variance of $h_\vp=\log\vp'$. The diagonal Grunsky expansion gives a
quantitative estimate in this case. The stronger, unrestricted bound is
provided by the finite-scale folding method of \cite{folding-topology}.
The connection between the folding method and Morse theory was
identified in the OpenAI preprints
\cite{OpenAI-Brennan,OpenAI-inverse-first} and subsequently refined
in the work of Binder, Ivrii, Prause, and Smirnov \cite{BIPS}.
This mind-boggling connection converts topological information about
critical points into analytic inequalities for conformal mappings and
underlies the later finite-scale folding approach to asymptotic variance.

The following maximal asymptotic variance theorem was established in
\cite{folding-topology}.
\begin{thm}
\label{thm:sigma-bound1}
For the normalized schlicht class $\mathscr{S}$, we have
\[
\limsup_{r\to1^-}\ \sup_{\vp\in\mathscr{S}}
\frac{1}{\log\frac{1}{1-r^2}}
\int_{\T}|\log\vp'(r\zeta)|^2\diff s(\zeta)\le1.
\]
In particular, for every $\vp\in\mathscr{S}$,
\[
\sigma(\log\vp')^2=\sigma_1(\Nop(\vp))^2\le1.
\]
\end{thm}

The inversion correspondence between the schlicht class $\mathscr{S}$
and the exterior class $\Sigma$ gives the following consequence of
Theorem \ref{thm:sigma-bound1}.
\begin{cor}\label{cor:exterior-folding}
For every $\psi\in\Sigma$, we have
\[
\sigma_\e(\log\psi')^2\le 1.
\]
\end{cor}

\begin{proof}
Write $\psi(1/z)=1/\vp(z)$ with $\vp\in\mathscr{S}$. Then
\[
\log\psi'(1/z)=\log\vp'(z)-2\log\frac{\vp(z)}z,
\qquad z\in\D.
\]
The logarithms are normalized to vanish at $z=0$.
The logarithmic-coefficient estimate of Duren and Leung
\cite{DurenLeung} gives a uniform $H^2$ bound for
$\log(\vp/z)$. Hence its circular $L^2$ norm is bounded,
independently of the radius and of $\vp$.
The triangle inequality shows that the finite-step variances
of $\log\psi'$ on $|w|=1/r$ and of $\log\vp'$ on $|z|=r$
have the same upper limit as $r\to1^-$. The result follows from
Theorem \ref{thm:sigma-bound1}.
\end{proof}

The general-contraction estimate provides another useful point of comparison.
\begin{rem}
Let $T$ be a contraction and let $\mathscr{Q}[T]$ be its Dirichlet
symbol. The work of Hedenmalm and Shimorin \cite{GAF} gives the bound
\[
\sigma(\oslash\mathscr{Q}[T])^2\le 2,
\qquad \|T\|\le 1,
\]
for the asymptotic variance of the diagonal restriction of the symbol.
Zachary Chase constructed examples based on permutation matrices for which
the asymptotic variance exceeds $1.72$, as reported in \cite{GAF}.
Han, Qiu, and Wang \cite{HanQiuWang} subsequently gave a direct proof
of the underlying weighted short anti-diagonal inequality and extended
it to a three-tensor analogue. For the Grunsky contractions
$T=\Gamma_{\vp}$ with $\vp\in\mathscr{S}$, the folding theorem
above gives the stronger bound $1$.
\end{rem}

The exterior estimate also has a direct operator-theoretic
interpretation in the interior disk.
\begin{rem}
Let $T=\Gamma_{\vp}$ for $\vp\in\mathscr{S}$, and let
$\mathscr{Q}[T]$ denote its Dirichlet symbol. If
$\psi(1/z)=1/\vp(z)$, then
\[
\oslash\mathscr{Q}[T](z)
=\mathscr{Q}[T](z,z)
=\log\frac{z^2\vp'(z)}{\vp(z)^2}
=\log\psi'(1/z),\qquad z\in\D.
\]
Here the first equality is the diagonal restriction of
\eqref{eq:DirsymbGrunsky-0}. Consequently, Corollary
\ref{cor:exterior-folding} gives
\[
\sigma(\oslash\mathscr{Q}[T])^2\le 1,
\qquad T=\Gamma_{\vp},\quad \vp\in\mathscr{S}.
\]
The bounded $H^2$ correction $-2\log(\vp/z)$ explains why the
folding estimate for $\log\vp'$ also controls this diagonal restriction.
\end{rem}

The folding result also provides a useful comparison with the diagonal
Grunsky approach.
\begin{rem}
The diagonal Grunsky norm expansion in Theorem \ref{thm:l,l'} provides
an exact identity, but extracting effective asymptotic variance bounds
from it presents additional difficulties. The expressions in the
expansion resemble higher derivatives of $\log\vp'$ without coinciding
with them. Consequently, a derivative comparison principle such as
Corollary \ref{cor:sigma-comparison} cannot be applied directly.
Additional inequalities are needed, and these may entail a loss of
sharpness. For instance, the earlier diagonal-expansion argument yielded
\[
\sigma(\log\vp')^2\le \frac25(\sqrt{96}-6)
=1.51918\ldots,
\]
which has been superseded by the folding bound above. The Schwarz\-ian
estimate in the following subsection is a separate application of the
Grunsky framework.
\end{rem}

The first three terms of the $l=l'=0$ diagonal expansion also
provide a useful joint estimate for the nonlinearity and its square.
Write $N=\Nop(\vp)$ and $S=\Sop(\vp)$.
The corresponding formal diagonal expressions are
\[
Z_1^{\mathrm{form}}=\tfrac12N,\qquad
Z_2^{\mathrm{form}}=\tfrac18N^2,\qquad
Z_3^{\mathrm{form}}=\tfrac1{12}(N^2)'.
\]
Keeping these three nonnegative terms in Theorem \ref{thm:l,l'}
and applying Corollary \ref{cor:sigma-comparison} to the third one,
we obtain
\begin{equation}
\label{eq:folding-grunsky-joint}
\frac14\sigma_1^{\langle r\rangle}(N)^2
+\frac5{96}\sigma_2^{\langle r\rangle}(N^2)^2
\le\frac12+\ordo(1).
\end{equation}
Here the inessential contributions of the logarithmic derivative
$\vp'/\vp-1/z$ are absorbed in the remainder, as in the diagonal
expansion argument. More precisely, after expanding the squares and absorbing the terms
containing $\mathrm L(\vp)=\vp'/\vp-1/z$ into an $o(L(r))$ remainder,
the first three diagonal terms give the coefficients
$\frac14$, $\frac1{48}$, and $\frac1{160}$ in front of the
weighted integrals of $|N|^2$, $|N|^4$, and $|NN'|^2$,
respectively, with weights $(1-|z|^2)$, $(1-|z|^2)^3$, and
$(1-|z|^2)^5$. Here $L(r)=\log\frac1{1-r^2}$.
The weighted differentiation comparison with $\alpha=3$ gives
\[
\int_\D |N(rz)N'(rz)|^2(1-|z|^2)^5\dA(z)
=\frac5{r^2}\int_\D |N(rz)|^4(1-|z|^2)^3\dA(z)+\Ordo(1).
\]
Thus the coefficient of the fourth-moment integral, after letting
$r\to1^-$ and dividing by $L(r)$, is
\[
\frac1{48}+\frac5{160}=\frac5{96}.
\]
The error terms vanish after division by $L(r)$.
This joint estimate, rather than the older separate bound for the
asymptotic variance of the nonlinearity, will be useful below.

\subsection{Schwarz\-ian asymptotic variance}
The diagonal Grunsky estimate \eqref{eq:folding-grunsky-joint}, in combination
with the folding theorem and the derivative comparison of Corollary
\ref{cor:sigma-comparison}, gives a strengthened estimate for the
Schwarz\-ian derivative. The earlier independent Grunsky calculation is
preserved in the auxiliary manuscript.

We recall that $\rho_{\mathbb{H}}$ denotes the hyperbolic discrepancy
density, arising from the optimal approximation of hyperbolic area
density by the modulus of a holomorphic function
\cite{geometric zero packing}. The estimate
$\rho_{\mathbb{H}}\leq 0.0523\ldots$, based on the $(12,3)$
hyperbolic configuration, is established in \cite{folding-topology}. The following theorem combines the resulting
lower bound with the upper bound proved here.

\begin{thm}
\label{thm:schwarzian-folding-grunsky}
The maximal Schwarz\-ian asymptotic variance satisfies
\[
6(1-\rho_{\mathbb{H}})
\le \sup_{\vp\in\mathscr S}\sigma_2(\Sop(\vp))^2
\le \frac{36}{5}+\frac{12}{\sqrt5}
=12.566563\ldots.
\]
\end{thm}

\begin{proof}
The lower bound follows from \cite{folding-topology}. We establish
only the upper bound here.
Here and below, $\ordo(1)$ denotes a quantity tending to zero as $r\to1^-$.
Put $N=\Nop(\vp)$ and $S=\Sop(\vp)=N'-\frac12N^2$.
Corollary \ref{cor:sigma-comparison}, together with the folding theorem,
gives
\[
\sigma_2^{\langle r\rangle}(N')^2\le 6+\ordo(1).
\]
By the joint diagonal estimate \eqref{eq:folding-grunsky-joint}
and the same derivative comparison,
\[
4\sigma_2^{\langle r\rangle}(N')^2
+5\sigma_2^{\langle r\rangle}(N^2)^2
\le 48+\ordo(1).
\]
The triangle inequality and these two constraints show that
\[
\begin{aligned}
\sigma_2^{\langle r\rangle}(S)
&\le \sigma_2^{\langle r\rangle}(N')
+\tfrac12\sigma_2^{\langle r\rangle}(N^2)\\
&\le \sqrt6+\sqrt{6/5}+\ordo(1).
\end{aligned}
\]
Indeed, the maximum of the sum in the first line, subject to the
preceding constraints, occurs when
\[
\sigma_2^{\langle r\rangle}(N')^2=6,\qquad
\sigma_2^{\langle r\rangle}(N^2)^2=\frac{24}{5},
\]
up to asymptotically vanishing errors.
Squaring and taking the upper limit gives
\[
\sigma_2(S)^2\le
\left(\sqrt6+\sqrt{6/5}\right)^2
=\frac{36}{5}+\frac{12}{\sqrt5}.
\]
\end{proof}

\begin{rem}
The contrast with the classical sufficient univalence criteria is
noteworthy. In their strict forms, Becker's criterion for the nonlinearity
and the Nehari--Ahlfors--Weill Schwarz\-ian criterion state that either
of the pointwise estimates
\[
(1-|z|^2)|\Nop(\vp)(z)|\le k<1
\quad\text{or}\quad
(1-|z|^2)^2|\Sop(\vp)(z)|\le 2k<2,
\qquad z\in\D,
\]
implies a quasiconformal extension and hence a quasicircle boundary
\cite{becker,AhlforsWeill}. At the endpoint $k=1$, either condition
still guarantees univalence, but no longer guarantees quasiconformal
extendability. These are sufficient, not necessary, conditions.
The corresponding bounds for asymptotic variance behave differently. For the nonlinearity, the folding theorem
\cite{folding-topology} gives the strict averaged bound
\[
\sup_{\vp\in\mathscr S}\sigma_1(\Nop(\vp))^2
=\sup_{\vp\in\mathscr S}\sigma(\log\vp')^2<1,
\]
so that Becker's numerical threshold holds on average even after
maximization over the schlicht class. By contrast,
Theorem \ref{thm:schwarzian-folding-grunsky} gives
\[
\sup_{\vp\in\mathscr S}\sigma_2(\Sop(\vp))^2
\ge6(1-\rho_{\mathbb{H}})\ge5.686\ldots>4,
\]
so that Nehari's squared threshold fails on average for a lot of
univalent mappings. In contrast with the Nehari sufficient condition
for univalence, we should mention the classical Kraus bound
\cite{Kraus}, which asserts the sharp necessary pointwise estimate
$(1-|z|^2)^2|\Sop(\vp)(z)|\le6$ for every $\vp\in\mathscr S$
and $z\in\D$.
\end{rem}

\begin{rem}
A heuristic comparison is obtained by modelling the hyperbolically
normalized nonlinearity $N=\Nop(\vp)$ and its derivatives as a centered,
rotationally symmetric complex Gaussian analytic field. We take its
variance parameter to be $\mathbb E|N|^2=\Sigma_{\max}^2$, where
$\Sigma_{\max}^2=\sup_{\vp\in\mathscr S}\sigma_1(\Nop(\vp))^2$.
Here the expectation is understood in the normalized boundary-scale model.
The Gaussian fourth-moment identity, the vanishing mixed third moment,
and the differentiation rule then suggest, in view of
$S=N'-\frac12 N^2$, that
\[
\sup_{\vp\in\mathscr S}\sigma_2(\Sop(\vp))^2\approx6\Sigma_{\max}^2+\frac12\Sigma_{\max}^4.
\]
Since $1-\rho_{\mathbb H}\le\Sigma_{\max}^2<1$, this predicts the range
\[
6(1-\rho_{\mathbb H})+\frac12(1-\rho_{\mathbb H})^2
\ \lesssim\ \sup_{\vp\in\mathscr S}\sigma_2(\Sop(\vp))^2
\ \lesssim\ \frac{13}{2},
\]
which, by the above estimate of $\rho_{\mathbb H}$, ranges approximately
from $6.135\ldots$ to $6.5$. This is a Gaussian moment heuristic, not an
additional bound for the schlicht class.
\end{rem}

\begin{figure}
\includegraphics[width=0.5\linewidth]{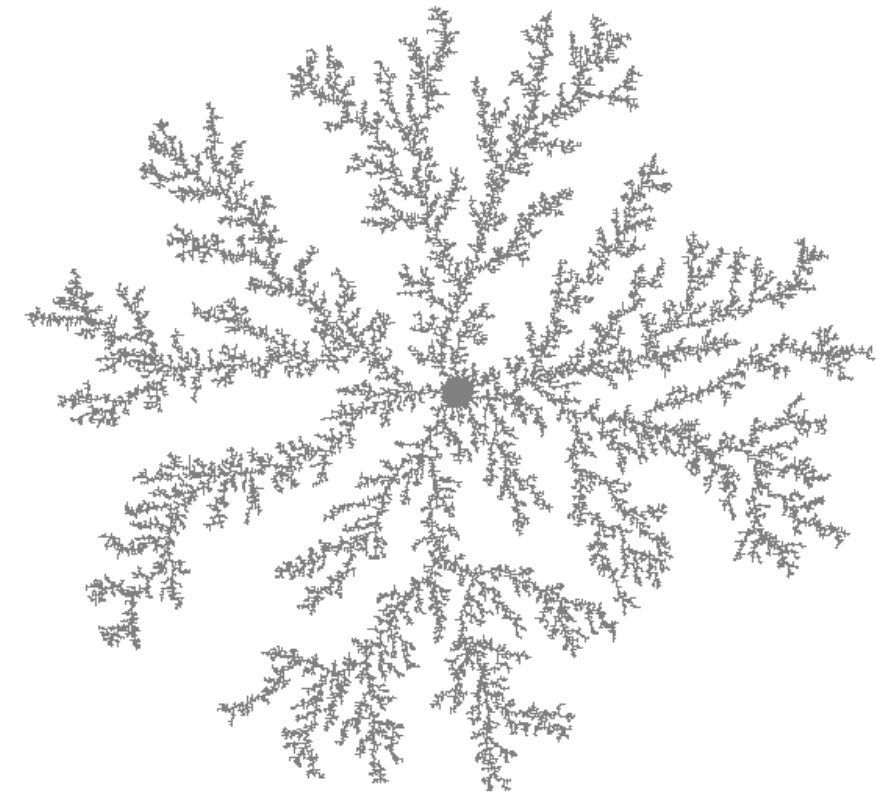}
\caption{A beautiful DLA cluster produced by Alexis Monnerot-Dumaine.
DLA stands for ``diffusion limited aggregation'', and appears physically
when water is injected into oil.  
It is not known which simply connected compacts should be optimal for the
maximizing the average Schwarz\-ian derivative amplitude. However, as the
pointwise maximum is achieved for slits pointing in the necessary direction,
the compact should try to achieve this in an average sense, and DLA looks
like it tries to do just that.}
\label{fig:DLA}
\end{figure}

\smallskip

\noindent\textbf{Acknowledgement:} 
Debnath acknowledges support by the Verg Foundation. 
Hedenmalm acknowledges support from the Ministry of Science and Higher
Education of the Russian Federation under agreement 075-15-2025-013, as
well as the hospitality of BIMSA.

\end{document}